\documentclass[12pt]{amsart}

\usepackage{amsmath,amssymb,latexsym,amsthm,newlfont,enumerate}
\usepackage{mathptmx}

\usepackage{graphicx}
\usepackage[all]{xy}

\theoremstyle{plain}

\newtheorem{thm}{Theorem}[section]

\newtheorem*{thmA}{Theorem A}
\newtheorem*{thmB}{Theorem B}
\newtheorem*{thmC}{Theorem C}
\newtheorem{pro}[thm]{Proposition}

\newtheorem{lem}[thm]{Lemma}

\newtheorem{cor}[thm]{Corollary}
\newtheorem{con}[thm]{Conjecture}

\theoremstyle{definition}

\newtheorem{dfn}[thm]{Definition}

\newtheorem{nt}[thm]{Notation}

\newtheorem{rem}[thm]{Remark}

\theoremstyle{remark}

\DeclareMathOperator{\mult}{mult}
\DeclareMathOperator{\ord}{ord}
\DeclareMathOperator{\res}{res}
\DeclareMathOperator{\im}{Im}
\DeclareMathOperator{\Int}{int}
\DeclareMathOperator{\relint}{relint}
\DeclareMathOperator{\Supp}{Supp}
\DeclareMathOperator{\Bs}{Bs}

\DeclareMathOperator{\Div}{Div}
\DeclareMathOperator{\WDiv}{WDiv}
\DeclareMathOperator{\ddiv}{div}
\DeclareMathOperator{\codim}{codim}
\DeclareMathOperator{\Exc}{Exc}

\DeclareMathOperator{\Mob}{Mob}
\DeclareMathOperator{\bMob}{\mathbf{Mob}}
\DeclareMathOperator{\Fix}{Fix}

\DeclareMathOperator{\Proj}{Proj}

\newcommand{\R}{\mathbb{R}}
\newcommand{\Q}{\mathbb{Q}}
\newcommand{\N}{\mathbb{N}}
\newcommand{\Z}{\mathbb{Z}}
\newcommand{\PP}{\mathbb{P}}

\newcommand{\C}{\mathbb{C}}

\newcommand{\B}{\mathbf{B}}

\newcommand{\m}{\mathbf{m}}

\newcommand{\NE}{\mathrm{NE}}

\newcommand{\bigcone}{\mathrm{Big}}

\newcommand{\OO}{\mathcal{O}}

\newcommand{\mcal}{\mathcal}
\newcommand{\mfrak}{\mathfrak}

\title{Adjoint rings are finitely generated}
\date{9 December 2009}
\author{Vladimir Lazi\'c}
\address{Max-Planck-Institut f\"ur Mathematik, Vivatsgasse 7, 53111 Bonn, Germany}
\email{lazic@mpim-bonn.mpg.de}
\thanks{This paper was written while I was a PhD student at the University of Cambridge, a research visitor at the Max-Planck-Institut
f\"ur Mathematik and a postdoc at the Institut Fourier.}

\begin{document}

\begin{abstract}
This paper proves finite generation of the log canonical ring without Mori theory.
\end{abstract}

\maketitle
\bibliographystyle{amsalpha}

\tableofcontents

\section{Introduction}

The main goal of this paper is to prove the following theorem while avoiding techniques of the Minimal Model Program.

\begin{thm}\label{cor:can}
Let $(X,\Delta)$ be a projective klt pair. Then the log canonical ring $R(X,K_X+\Delta)$ is finitely generated.
\end{thm}

Let me sketch the strategy for the proof of finite generation in this paper and present difficulties that arise on the way. The natural idea
is to pick a smooth divisor $S$ on $X$ and to restrict the algebra to it. If we are very lucky, the restricted algebra will be
finitely generated and we might hope that the generators lift to generators on $X$. There are several issues with this approach.

First, to obtain something meaningful on $S$, we require $S$ to be a log canonical centre of some pair $(X,\Delta')$ such that the rings
$R(X,K_X+\Delta)$ and $R(X,K_X+\Delta')$ share a common truncation.

Second, even if the restricted algebra were finitely generated, the same might not be obvious for the kernel of the restriction map.
So far this seems to have been the greatest conceptual issue in attempts to prove the finite generation by the plan just outlined.

Third, the natural strategy is to use the Hacon-M\textsuperscript{c}Kernan extension theorem,
and hence we must be able to ensure that $S$ does not belong to the stable base locus of $K_X+\Delta'$.

The idea to resolve the kernel issue is to view $R(X,K_X+\Delta)$ as a subalgebra of a much bigger algebra containing generators
of the kernel by construction. The new algebra is graded by a monoid whose rank corresponds roughly to
the number of components of $\Delta$ and of an effective divisor $D\sim_\Q K_X+\Delta$.
A basic example which models the general lines of the proof in \S\ref{proofmain} is presented in Lemma \ref{lem:restricted}.

It is natural to try and restrict to a component of $\Delta$, the issue of course being that $(X,\Delta)$ does not have log canonical centres.
Therefore I allow restrictions to components of some effective divisor $D\sim_\Q K_X+\Delta$, and a tie-breaking-like technique allows me
to create log canonical centres. Algebras encountered this way are, in effect, plt algebras, and their restriction is handled in
\S\ref{plt}. This is technically the most involved part of the proof.

Since the algebras we consider are of higher rank, not all divisors will have the same log canonical centres. I therefore restrict
to available centres, and lift generators from algebras that live on different divisors. Since the restrictions will also be
algebras of higher rank, the induction process must start from them. The contents of this paper can be summarised in the following result.

\begin{thm}\label{thm:cox}
Let $X$ be a projective variety, and let $D_i=k_i(K_X+\Delta_i+A)\in\Div(X)$, where $A$ is an ample $\Q$-divisor and
$(X,\Delta_i+A)$ is a klt pair for $i=1,\dots,\ell$ . Then the adjoint ring $R(X;D_1,\dots,D_\ell)$ is finitely generated.
\end{thm}

Theorem \ref{cor:can} is a corollary to the previous theorem. Techniques of the MMP were used to prove Theorem \ref{cor:can} in the seminal
paper \cite{BCHM}. A proof of finite generation of the canonical ring of general type by analytic methods is announced in \cite{Siu06}.

In the following result I recall some of the well known consequences of Theorems \ref{cor:can} and \ref{thm:cox}; further discussion is in the appendix.

\begin{cor}\label{cor:cor}
The following holds.
\begin{enumerate}
\item Klt flips exist.
\item Canonical models of klt pairs of log general type exist.
\item Log Fano klt pairs are Mori dream spaces.
\end{enumerate}
\end{cor}

In the appendix I give a very short history of Mori theory, and also outline a new approach which aims to
turn the conventional thinking about classification on its head. Finite generation comes at the beginning of the theory and all main results of the
Minimal Model Program should be derived from it. In light of this new viewpoint, it is my hope that the techniques of this paper could be adapted
to handle finite generation in the case of log canonical singularities and the Abundance Conjecture.\vspace{5mm}
\paragraph{\bf Acknowledgements}
I am indebted to my PhD supervisor Alessio Corti whose initial insight that higher rank algebras are a natural setting for the finite generation
fundamentally shaped the way I think about the problem. I would like to express my gratitude for his encouragement, support and continuous inspiration.
I am very grateful to C.~Hacon for suggesting that methods from \cite{Hac08} might be useful in the context of finite generation of
the restricted algebra. Many thanks to D.~Abramovich, S.~Boucksom, P.~Cascini, J.-P.~Demailly, S.~Druel, O.~Fujino, A.-S.\ Kaloghiros, C.~Maclean,
J.~M\textsuperscript{c}Kernan, M.~P\u{a}un and M.~Popa for useful comments, and to M.~Reid whose suggestions improved the organisation of the paper.

\section{Notation and conventions}\label{sec:2}

Unless stated otherwise, varieties in this paper are projective and normal over $\C$. However, all results hold when $X$ is, instead of being projective,
assumed to be projective over an affine variety $Z$. The group of Weil, respectively Cartier, divisors on a variety $X$ is denoted by $\WDiv(X)$,
respectively $\Div(X)$. Subscripts denote the rings in which the coefficients are taken. For a divisor $D$, $[D]$ denotes its class in $N^1(X)$.

We say an ample $\Q$-divisor $A$ on a variety $X$ is {\em general\/} if there is a sufficiently divisible
positive integer $k$ such that $kA$ is very ample and $kA$ is a general section of $|kA|$.
In particular we can assume that for some $k\gg0$, $kA$ is a smooth divisor on $X$.
In practice, we fix $k$ in advance, and generality is most often needed to ensure that $A$ does not make singularities of pairs worse.

For any two formal sums of prime divisors $P=\sum p_iE_i$ and $Q=\sum q_iE_i$ on $X$, set
$$P\wedge Q=\sum\min\{p_i,q_i\}E_i.$$

For definitions and basic properties of multiplier ideals used in this paper see \cite{HM08}.

The sets of non-negative (respectively non-positive) rational and real numbers are denoted by $\Q_+$ and $\R_+$ (respectively $\Q_-$ and $\R_-$),
and similarly for $\Z_{>0}$ and $\R_{>0}$. For two subsets of $A$ and $B$ of a vector space $V$, $A+B$ denotes their Minkowski sum, i.e.\ the set
$\{a+b:a\in A,b\in B\}$.
\vspace{5mm}
\paragraph{\bf b-Divisors}
I use basic properties of b-divisors, see \cite{Cor07}. The cone of mobile b-divisors on $X$ is denoted by $\bMob(X)$.
\begin{dfn}
Let $(X,\Delta)$ be a log pair. For a model $f\colon Y\rightarrow X$ we can write uniquely
$$K_Y+B_Y=f^*(K_X+\Delta)+E_Y,$$
where $B_Y$ and $E_Y$ are effective with no common components, and $E_Y$ is $f$-exceptional.
The {\em boundary\/} b-divisor $\B(X,\Delta)$ is defined by $\B(X,\Delta)_Y=B_Y$ for every model $Y\rightarrow X$.
\end{dfn}

\begin{lem}
If $(X,\Delta)$ is a log pair, then the b-divisor $\B(X,\Delta)$ is well-defined.
\end{lem}
\begin{proof}
Let $g\colon Y'\rightarrow X$ be a model such that there is a proper birational morphism $h\colon Y'\rightarrow Y$. Pushing forward
$K_{Y'}+B_{Y'}=g^*(K_X+\Delta)+E_{Y'}$ via $h_*$ yields
$$K_Y+h_*B_{Y'}=f^*(K_X+\Delta)+h_*E_{Y'},$$
and thus $h_*B_{Y'}=B_Y$ since $h_*B_{Y'}$ and $h_*E_{Y'}$ have no common components.
\end{proof}

If $\{D\}$ denotes the fractional part of a divisor $D$, we have:

\begin{lem}\label{disjoint}
Let $(X,\Delta)$ be a log canonical pair. There exists a log resolution $Y\rightarrow X$ such that the components of $\{\B(X,\Delta)_Y\}$ are disjoint.
\end{lem}
\begin{proof}
See \cite[2.36]{KM98} or \cite[6.7]{HM05}.
\end{proof}
\vspace{5mm}
\paragraph{\bf Convex geometry}
If $\mcal{S}=\sum\N e_i$ is a submonoid of $\N^n$, I denote $\mcal{S}_\Q=\sum\Q_+e_i$ and $\mcal{S}_\R=\sum\R_+e_i$.
A monoid $\mcal{S}\subset\N^n$ is {\em saturated\/} if $\mcal{S}=\mcal{S}_\R\cap\N^n$.

If $\mcal{S}=\sum_{i=1}^n\N e_i$ and $\kappa_1,\dots,\kappa_n$ are positive integers,
the submonoid $\sum_{i=1}^n\N \kappa_ie_i$ is called a {\em truncation\/} of $\mcal{S}$.
If $\kappa_1=\dots=\kappa_n=\kappa$, I denote $\mcal S^{(\kappa)}=\sum_{i=1}^n\N \kappa e_i$, and this truncation does not depend on a
choice of generators of $\mcal S$.

A submonoid $\mcal{S}=\sum\N e_i$  of $\N^n$ (respectively a cone $\mcal{C}=\sum\R_+ e_i$ in $\R^n$) is called {\em simplicial\/}
if its generators $e_i$ are linearly independent in $\R^n$, and the $e_i$ form a {\em basis\/} of $\mcal{S}$
(respectively $\mcal{C}$).

I often use without explicit mention that if $\lambda\colon\mcal M\rightarrow\mcal S$
is an additive surjective map between finitely generated saturated monoids, and if $\mcal C$ is a rational polyhedral cone in $\mcal S_\R$, then
$\lambda^{-1}(\mcal S\cap\mcal C)=\mcal M\cap\lambda^{-1}(\mcal C)$. In particular, if $\mcal M$ and $\mcal S$ are saturated,
the inverse image of a saturated finitely generated submonoid of $\mcal S$ is a saturated finitely generated submonoid of $\mcal M$.

For a polytope $\mcal P\subset\R^n$, I denote $\mcal P_\Q=\mcal P\cap\Q^n$. A polytope is {\em rational\/} if it is the convex hull of finitely
many rational points. The dimension of a polytope $\mcal P$, denoted $\dim\mcal P$, is the dimension of the smallest affine space containing $\mcal P$.

If $\mcal B\subset\R^n$ is a convex set, then $\R_+\mcal B$ denotes the set $\{rb:r\in\R_+,b\in\mcal B\}$. In particular,
if $\mcal B$ is a rational polytope, $\R_+\mcal B$ is a rational polyhedral cone.

\begin{rem}\label{rem:4}
The following will be used in the proof of Theorem \ref{lem:PL}. Assume that $\mcal P\subset\R^n$ is an $n$-dimensional polytope
and let $\{\mcal H_\alpha\}$ be a collection of countably many affine hyperplanes in $\R^n$. Then $\mcal P\not\subset\bigcup_\alpha\mcal H_\alpha$.
To see this, I argue by induction on $n$. The statement obviously stands for $n=1$, so assume that $n>1$. Fix a point $p$ in the interior of $\mcal P$,
and assume $\mcal P\subset\bigcup_\alpha\mcal H_\alpha$. Since the number of affine hyperplanes passing through $p$ is uncountable, there is an affine
hyperplane $\mcal H\ni p$ different from all $\mcal H_\alpha$. Now $\mcal P\cap\mcal H\subset\bigcup_\alpha(\mcal H_\alpha\cap\mcal H)$, but this is a
contradiction since $\{\mcal H_\alpha\cap\mcal H\}$ is at most countable collection of hyperplanes in $\mcal H$.
\end{rem}

Let $\mcal C\subset \R^n$ be a rational polyhedral cone and $V$ an $\R$-vector space with an ordering.
A function $f\colon\mcal{C}\rightarrow V$ is: {\em positively homogeneous\/} if $f(\lambda x)=\lambda f(x)$ for $x\in\mcal C,\lambda\in\R_+$,
and {\em superlinear\/} if $\lambda f(x)+\mu f(y)\leq f(\lambda x+\mu y)$ for
$x,y\in\mcal C,\lambda,\mu\in\R_+$. It is {\em piecewise linear\/} if there is a finite polyhedral decomposition $\mcal{C}=\bigcup\mcal{C}_i$ such that
$f_{|\mcal{C}_i}$ is linear for every $i$; additionally, if each $\mcal C_i$ is a rational cone, it is {\em rationally piecewise linear\/}.
Similarly for {\em sublinear, superadditive, subadditive, (rationally) piecewise linear\/}.
Assume furthermore that $f$ is linear and $\dim\mcal{C}=n$. The {\em linear extension of $f$ to $\R^n$\/} is the unique linear function
$\ell\colon\R^n\rightarrow V$ such that $\ell_{|\mcal{C}}=f$.

Unless otherwise stated, cones considered in this paper do not contain lines,
and the {\em relative interior\/} of a cone $\mcal{C}=\sum\R_+e_i\subset\R^n$, denoted by $\relint\mcal{C}$, is the origin union the topological interior
of $\mcal{C}$ in the space $\sum\R e_i$; if $\dim\mcal{C}=n$, we call it the {\em interior\/} of $\mcal{C}$ and
denote it by $\Int\mcal{C}$. The boundary of a closed set $\mcal D$ is denoted by $\partial\mcal D$. If a norm $\|\cdot\|$ on $\R^n$ is given,
then for $x\in\R^n$ and for any $r>0$, the closed ball of radius $r$ with centre at $x$ is denoted by $B(x,r)$.
Unless otherwise stated, the norm considered is always the sup-norm $\|\cdot\|_\infty$, and note that then $B(x,r)$ is a hypercube in the
Euclidean norm.

I will need the following lemma in the proof of Theorem \ref{cor:linear}.

\begin{lem}\label{lem:4}
Let $\mcal D\subset\R^\ell$ be a closed cone, let $z_i\in\R^\ell\backslash\mcal D$ be linearly independent points, and denote $\mcal Z=\sum_i\R_+z_i$
and $\mcal C=\mcal Z+\mcal D$. Assume that $x_m\in\mcal C$ is a sequence converging to $x_\infty=\sum_i\alpha_iz_i$ with $\alpha_i>0$
for all $i$, and that $x_m\neq x_\infty$ for all $m$. Assume further that if $x_\infty=z+d$ with $z\in\mcal Z$ and $d\in\mcal D$,
then $d=0$, and in particular $x_\infty\notin\mcal D$. Then for every $m_0\in\N$ there exist $m\geq m_0$ and $x_m'\in\mcal C$ such that $x_m\in(x_\infty,x_m')$.
\end{lem}
\begin{proof}
Let $H\ni x_\infty$ be an affine hyperplane such that $\mcal Z_H=\mcal Z\cap H$ is a polytope. For every $m\gg0$, the intersection
of $\R_+x_m$ and $H$ is a point, and denote it by $y_m$. If there is $y_m'\in\mcal C$ such that $y_m\in(x_\infty,y_m')$, then it is
easy to see that there is $x_m'\in\R_+y_m'$ such that $x_m\in(x_\infty,x_m')$. Therefore, replacing
$x_m$ by $y_m$ and passing to a subsequence, I can assume that $x_m\in H$ for all $m$.

Write $x_m=s_m'+d_m'$ for every $m$, where $s_m'\in\mcal Z$ and $d_m'\in\mcal D$; note that $s_m'\neq0$ for $m\gg0$ as $x_\infty\notin\mcal D$.
Since $\R_+s_m'$ intersects $H$, then $\R_+d_m'$ also intersects $H$, and denote the intersection points by $s_m$ and $d_m$, respectively.
Setting $\alpha_m=\frac{\|x_m-d_m\|}{\|s_m-d_m\|}$, we have
\begin{equation}\label{equ:5}
x_m=\alpha_ms_m+(1-\alpha_m)d_m.
\end{equation}
Observe that $\|s_m-x_m\|$ is bounded from above for $m\gg0$ since $\mcal Z_H$ is compact, and that $\|x_m-d_m\|$ is bounded from below for $m\gg0$
as $x_m\notin\mcal D$ and $\mcal D$ is closed. Therefore $\frac{1}{\alpha_m}=1+\frac{\|s_m-x_m\|}{\|x_m-d_m\|}$ is bounded from above, and thus
$\alpha_m$ is bounded away from zero as $m\rightarrow\infty$. By passing to a subsequence we can assume that
$\lim\limits_{m\rightarrow\infty}\alpha_m=\alpha_\infty>0$, and that $\lim\limits_{m\rightarrow\infty}s_m=s_\infty\in\mcal Z_H$
since $\mcal Z_H$ is compact.
Therefore, $d_\infty=\lim\limits_{m\rightarrow\infty}d_m$ exists in $\mcal D$, and $x_\infty=\alpha_\infty s_\infty+(1-\alpha_\infty)d_\infty$. But then
$(1-\alpha_\infty)d_\infty=0$ by assumptions of the lemma.

Thus $\lim\limits_{m\rightarrow\infty}\alpha_ms_m=x_\infty$. If $x_\infty=\alpha_ms_m$ for some $m$, then by \eqref{equ:5} we have $x_m\in(x_\infty,x_m')$,
where $x_m'=(1-\alpha_m)d_m$. Therefore, I can assume that $x_\infty\neq\alpha_ms_m$ for all $m$. Since $x_\infty\in\relint\mcal Z$ by
assumption and $\alpha_ms_m\in\mcal Z$, for $m\gg0$ there exist $\hat s_m\in\mcal Z$ and $t_m\in(0,1)$ such that $\alpha_ms_m=t_mx_\infty+(1-t_m)\hat s_m$. Let
$\hat d_m=\frac{1-\alpha_m}{1-t_m}d_m\in\mcal D$ and set $x_m'=\hat s_m+\hat d_m$. Then it is easy to check that $x_m=t_mx_\infty+(1-t_m)x_m'$, and we are done.
\end{proof}

\begin{rem}\label{rem:3}
The following situation will appear in the proof of Theorem \ref{cor:linear}. Let $\mcal W_i$ be finitely many half-spaces of $\R^\ell$ bounded by
affine hyperplanes $\mcal H_i$, and let $\mcal Q=\bigcap_i\mcal W_i$. Let $x_m\in\R^\ell\backslash\mcal Q$ be a convergent sequence of points
and fix $z\in\R^\ell$. Assume that for each $m\in\N\cup\{\infty\}$ there exists a point $y_m\in(x_m+\R_-z)\cap\partial\mcal Q$ closest to $x_m$,
and that $x_\infty=\lim\limits_{m\rightarrow\infty}x_m\in\mcal H_i$ for all $i$. Then I claim that a subsequence of $y_m$ converges to $x_\infty$.
To see this, by passing to a subsequence I can assume that $y_m\in\mcal H_{i_0}$ for all $m$ and for a fixed $i_0$.
If $\R z\cap\mcal H_{i_0}=\emptyset$, then $x_m\in\mcal H_{i_0}$ for all $m$,
and by replacing $\R^\ell$ by $\mcal H_{i_0}$ we can finish by induction on $\ell$. If $\R z\cap\mcal H_{i_0}\neq\emptyset$, then
$\{y_m\}=(x_m+\R_-z)\cap\mcal H_{i_0}$, and it is easy to see that $\lim\limits_{m\rightarrow\infty}y_m=x_\infty$.
\end{rem}
\paragraph{\bf Asymptotic invariants}
The standard references on asymptotic invariants arising from linear series are \cite{Nak04,ELMNP}.

\begin{dfn}
Let $X$ be a variety and $D\in\WDiv(X)_\R$. For $k\in\{\Z,\Q,\R\}$, define
$$|D|_k=\{C\in\WDiv(X)_k:C\geq0,C\sim_kD\}.$$
If $T$ is a prime divisor on $X$ such that $T\not\subset\Fix|D|$, then $|D|_T$ denotes the image of the linear system $|D|$ under restriction
to $T$. The {\em stable base locus\/} of $D$ is $\B(D)=\bigcap_{C\in|D|_\R}\Supp C$ if $|D|_\R\neq\emptyset$, otherwise we set $\B(D)=X$.
The {\em diminished base locus\/} is $\B_-(D)=\bigcup_{\varepsilon>0}\B(D+\varepsilon A)$ for an ample divisor $A$; this does not depend on
a choice of $A$. In particular, $\B_-(D)\subset\B(D)$.
\end{dfn}

We denote $\WDiv(X)^{\kappa\geq0}=\{D\in\WDiv(X):|D|_\R\neq\emptyset\}$, and similarly for $\Div(X)^{\kappa\geq0}$ and for versions of these sets with
subscripts $\Q$ and $\R$. Observe that when $D\in\WDiv(X)$, the condition $|D|_\R\neq\emptyset$ is equivalent to $\kappa(X,D)\geq0$ by Lemma
\ref{lem:2} below, where $\kappa$ is the Iitaka dimension.

It is elementary that $\B(D_1+D_2)\subset\B(D_1)\cup\B(D_2)$ for $D_1,D_2\in\WDiv(X)_\R$. In other words, the set
$\{D\in\WDiv(X)_\R:x\notin\B(D)\}$ is convex for every point $x\in X$.
By \cite[3.5.3]{BCHM}, $\B(D)=\bigcap_{C\in|D|_\Q}\Supp C$ when $D$ is a $\Q$-divisor, which is the standard definition of the stable base locus.

\begin{dfn}
Let $Z$ be a closed subvariety of a smooth variety $X$ and let $D\in\Div(X)_\R^{\kappa\geq0}$.
The {\em asymptotic order of vanishing of $D$ along $Z$\/} is
$$\ord_Z\|D\|=\inf\{\mult_ZC:C\in|D|_\R\}.$$
\end{dfn}

\begin{rem}
In the case of rational divisors, the infimum above can be taken over rational divisors, see Lemma \ref{lem:2} below.
More generally, one can consider any discrete valuation $\nu$ of $k(X)$ and define
$$\nu\|D\|=\inf\{\nu(C):C\in|D|_\Q\}$$
for an effective $\Q$-divisor $D$. Then \cite{ELMNP} shows that $\nu\|D_1\|=\nu\|D_2\|$ if $D_1$ and $D_2$ are numerically equivalent big divisors,
and that $\nu$ extends to a sublinear function on $\bigcone(X)_\R$. When $E$ is a prime divisor on a birational model over $X$, I write
$\ord_E\|\cdot\|$ for the corresponding geometric valuation.
\end{rem}

\begin{rem}\label{rem:10}
Nakayama \cite{Nak04} defines a function $\sigma_Z\colon\overline{\bigcone(X)}\rightarrow\R_+$ by
$$\sigma_Z\|D\|=\lim_{\varepsilon\downarrow0}\ord_Z\|D+\varepsilon A\|$$
for any ample $\R$-divisor $A$, and shows that it agrees with $\ord_Z\|\cdot\|$ on big classes. Then we define the formal sum
$N_\sigma\|D\|=\sum\sigma_Z\|D\|\cdot Z$ over all prime divisors $Z$ on $X$. Analytic properties of these
invariants were studied in \cite{Bou04}.
\end{rem}

We now define the restricted version of the invariant introduced.

\begin{dfn}
Let $S$ be a smooth divisor on a smooth variety $X$ and let $D\in\Div(X)_\R^{\kappa\geq0}$ be such that $S\not\subset\B(D)$.
Let $P$ be a closed subvariety of $S$. The {\em restricted asymptotic order of vanishing of $D$ along $P$\/} is
$$\ord_P\|D\|_S=\inf\{\mult_P C_{|S}:C\in|D|_\R,S\not\subset\Supp C\}.$$
\end{dfn}

\begin{lem}\label{lem:2}
Let $X$ be a smooth variety, $D\in\Div(X)_\Q^{\kappa\geq0}$ and let $D'\geq0$ be an $\R$-divisor such that $D\sim_\R D'$. Then for every
$\varepsilon>0$ there is a $\Q$-divisor $D''\geq0$ such that $D\sim_\Q D''$, $\Supp D'=\Supp D''$ and $\|D'-D''\|<\varepsilon$. In particular,
if $S\subset X$ is a smooth divisor such that $S\not\subset\B(D)$, then for every closed subvariety $P\subset S$ we have
$$\ord_P\|D\|_S=\inf\{\mult_P C_{|S}:C\in|D|_\Q,S\not\subset\Supp C\}.$$
\end{lem}
\begin{proof}
Let $D'=D+\sum_{i=1}^pr_i(f_i)$ for $r_i\in\R$ and $f_i\in k(X)$. Let $F_1,\dots,F_N$ be the components of $D$ and of all $(f_i)$, and assume that
$\mult_{F_j}D'=0$ for $j\leq\ell$ and $\mult_{F_j}D'>0$ for $j>\ell$. Let $(f_i)=\sum_{j=1}^N\varphi_{ij}F_j$ for all $i$,
and $D=\sum_{j=1}^N\delta_jF_j$. Then we have $\delta_j+\sum_{i=1}^p\varphi_{ij}r_i=0$ for $j=1,\dots,\ell$. Let $\mcal K\subset\R^p$
be the space of solutions of the system $\sum_{i=1}^p\varphi_{ij}x_i=-\delta_j$ for $j=1,\dots,\ell$. Then $\mcal K$ is a rational affine subspace
and $(r_1,\dots,r_p)\in\mcal K$, thus for $0<\eta\ll1$ there is a rational point $(s_1,\dots,s_p)\in\mcal K$ with $\|s_i-r_i\|<\eta$ for all $i$.
Therefore for $\eta$ sufficiently small, setting $D''=D+\sum_{i=1}^ps_i(f_i)$ we have the desired properties.
\end{proof}

\begin{rem}\label{rem:11}
Let $\mcal B^S_-(X)\subset\overline{\bigcone(X)}$ be the set of classes of divisors $D$ such that $S\not\subset\B_-(D)$.
Similarly as in Remark \ref{rem:10}, \cite{Hac08} introduces the function $\sigma_P\|\cdot\|_S\colon\mcal B^S_-(X)\rightarrow\R_+$ by
$$\sigma_P\|D\|_S=\lim_{\varepsilon\downarrow0}\ord_P\|D+\varepsilon A\|_S$$
for any ample $\R$-divisor $A$.
Then one can define a formal sum $N_\sigma\|D\|_S=\sum\sigma_P\|D\|_S\cdot P$ over all prime divisors $P$ on $S$. If $S\not\subset\B(D)$, then
$\lim_{\varepsilon\downarrow\varepsilon_0}\ord_P\|D+\varepsilon A\|_S=\ord_P\|D+\varepsilon_0A\|_S$ for every $\varepsilon_0>0$ and
for any ample divisor $A$ on $X$ similarly as in \cite[2.1.1]{Nak04}.
\end{rem}

In this paper I need a few basic properties cf.\ \cite{Hac08}.

\begin{lem}\label{lem:restrictedord}
Let $S$ be a smooth divisor on a smooth projective variety $X$ and let $P$ be a closed subvariety of $S$.
\begin{enumerate}
\item Let $D\in\Div(X)_\R^{\kappa\geq0}$ be such that $S\not\subset\B(D)$. If $A$ is an ample $\R$-divisor on $X$, then $\ord_P\|D+A\|_S\leq\ord_P\|D\|_S$,
and in particular $\sigma_P\|D\|_S\leq\ord_P\|D\|_S$.
\item If $D\in\mcal B^S_-(X)$, and if $A_m$ is a sequence of ample $\R$-divisors on $X$ such that
$\lim\limits_{m\rightarrow\infty}\|A_m\|=0$, then
$$\lim\limits_{m\rightarrow\infty}\ord_P\|D+A_m\|_S=\sigma_P\|D\|_S.$$
\item If $D,E\in\mcal B^S_-(X)$, then
$$\lim\limits_{\varepsilon\downarrow0}\sigma_P\|(1-\varepsilon)D+\varepsilon E\|_S=\sigma_P\|D\|_S.$$
\item Let $D$ be a pseudo-effective $\Q$-divisor on $X$ such that $\sigma_P\|D\|_S=0$. If $A$ is an ample $\Q$-divisor on $X$,
then there is $l\in\Z_{>0}$ such that $\mult_P\Fix|l(D+A)|_S=0$.
\end{enumerate}
\end{lem}
\begin{proof}
Statement (1) is trivial. The proof of (2) is standard: fix an ample divisor $A$ on $X$, and let $0<\varepsilon\ll1$. For $m\gg0$
the divisor $\varepsilon A-A_m$ is ample, and so by (1) we have
$$\ord_P\|D+\varepsilon A\|_S=\ord_P\|D+A_m+(\varepsilon A-A_m)\|_S\leq\ord_P\|D+A_m\|_S.$$
Letting $m\rightarrow\infty$, and then $\varepsilon\downarrow0$ we obtain
$$\sigma_P\|D\|_S\leq\lim\limits_{m\rightarrow\infty}\ord_P\|D+A_m\|_S,$$
and similarly for the opposite inequality.

For (3), let $A$ be an ample $\Q$-divisor such that $E-D+A$ is ample. Then by convexity,
\begin{align*}
\sigma_P\|D\|_S&=\lim_{\varepsilon\downarrow0}\sigma_P\|D+\varepsilon(E-D+A)\|_S\leq\lim_{\varepsilon\downarrow0}\sigma_P\|D+\varepsilon(E-D)\|_S\\
&\leq\lim_{\varepsilon\downarrow0}\big((1-\varepsilon)\sigma_P\|D\|+\varepsilon\sigma_P\|E\|_S\big)=\sigma_P\|D\|_S,
\end{align*}
thus the desired equality follows.

Finally, for (4), set $n=\dim X$, let $H$ be a very ample divisor on $X$, and fix a positive integer $l$ such that $l(D+A)$ is Cartier and $H'=\frac l2A-(K_X+S)-(n+1)H$
is ample. Since $\ord_P\|D+\frac12A\|_S\leq\sigma_P\|D\|_S=0$ by (1), there exists a $\Q$-divisor
$\Delta\sim_\Q D+\frac12A$ such that $S\not\subset\Supp\Delta$ and $\mult_P\Delta_{|S}<1/l$.
Since $l(D+A)_{|S}\sim_\Q K_S+l\Delta_{|S}+H'_{|S}+(n+1)H_{|S}$, by Nadel vanishing we have
$$H^i(S,\mcal J_{l\Delta_{|S}}(l(D+A)+mH))=0$$
for $m\geq-n$, and so the sheaf $\mcal J_{l\Delta_{|S}}(l(D+A))$ is globally generated
by \cite[5.7]{HM08} and its sections lift to $H^0(X,l(D+A))$ by \cite[4.4(3)]{HM08}. Since $\mult_P(l\Delta_{|S})<1$, the pair $(S,l\Delta_{|S})$
is klt around the generic point $\eta$ of $P$. Therefore the sheaf $\mcal J_{l\Delta_{|S}}$ is trivial at $\eta$, and so $\mult_P\Fix|l(D+A)|_S=0$.
\end{proof}
\begin{rem}\label{rem:5}
Analogously to Lemma \ref{lem:restrictedord}(4), one can prove that if $D$ is a pseudo-effective $\R$-divisor such that $\sigma_Z\|D\|=0$ for a closed subvariety of $Z$ of $X$,
then $Z\not\subset\B_-(D)$. In particular, if $D_{|Z}$ is defined, then $D_{|Z}-N_\sigma\|D\|_Z$ is pseudo-effective.
Further, let $f\colon Y\rightarrow X$ be a log resolution and denote $Z'=f_*^{-1}Z$. Then I claim $\sigma_{Z'}\|f^*D\|=0$.
To that end, we have first that $Z\not\subset\B(D+\varepsilon A)$ for an ample divisor $A$ and for any $\varepsilon>0$.
Therefore $Z'\not\subset\B(f^*D+\varepsilon f^*A)$, and thus $\sigma_{Z'}\|f^*D+\varepsilon f^*A\|\leq\ord_{Z'}\|f^*D+\varepsilon f^*A\|=0$. But then
$$\sigma_{Z'}\|f^*D\|=\lim_{\varepsilon\downarrow0}\sigma_{Z'}\|f^*D+\varepsilon f^*A\|=0$$
by \cite[2.1.4(2)]{Nak04}.
\end{rem}
\paragraph{\bf Convex sets in $\WDiv(X)_\R$}
Let $X$ be a variety and let $V$ be a finite dimensional affine subspace of $\WDiv(X)_\R$. Fix an ample $\Q$-divisor $A$ and a
prime divisor $G$ on $X$, and define
\begin{align*}
\mcal L_V&=\{\Phi\in V:K_X+\Phi\textrm{ is log canonical}\},\\
\mcal E_{V,A}&=\{\Phi\in\mcal L_V:K_X+\Phi+A\textrm{ is pseudo-effective}\},\\
\mcal B_{V,A}^G&=\{\Phi\in\mcal L_V:G\not\subset\B(K_X+\Phi+A)\},\\
\mcal B_{V,A}^{G=1}&=\{\Phi\in\mcal L_V:\mult_G\Phi=1,G\not\subset\B(K_X+\Phi+A)\}.
\end{align*}
If $V$ is a rational affine subspace, the set $\mcal L_V$ is a rational polytope by \cite[3.7.2]{BCHM}. Similarly as in Lemma \ref{bounded}
below, one can prove that Theorem \ref{thm:cox} implies that then also $\mcal E_{V,A}$, $\mcal B_{V,A}^G$ and $\mcal B_{V,A}^{G=1}$ are rational polytopes.
\begin{rem}\label{rem:8}
Assume the notation as above. In the proofs of Theorems \ref{thm:EimpliesLGA} and \ref{cor:linear} we will have the following situation. Let
$\mcal P_1$ and $\mcal P_2$ be properties of divisor classes in $N^1(X)$; namely, assume $\mcal P_1$ is the property that the class
is pseudo-effective, and $\mcal P_2$ that $\sigma_G\|\psi\|=0$ for a class $\psi\in\overline{\bigcone(X)}$. Denote
$\mcal P^1_{V,A}=\{\Phi\in\mcal L_V:K_X+\Phi+A\textrm{ has }\mcal P_1\}$, $\mcal P^2_{V,A}=\{\Phi\in\mcal L_V:\mult_G\Phi=1,K_X+\Phi+A\textrm{ has }\mcal P_2\}$
and $\mcal C^i=\R_+(K_X+A+\mcal P^i_{V,A})\subset\Div(X)_\R$. Assume that we know that $\mcal P^i_{V,A}$ are closed convex sets, and in particular
that $\mcal C^i$ are closed cones, and that we need to prove that $\mcal C^i$ are polyhedral.

The strategy is as follows. Fix $i$ and assume the contrary, i.e.\ that $\mcal C^i$ has infinitely many extremal rays.
Then there are distinct divisors $\Delta_m\in\mcal P^i_{V,A}$ for $m\in\N\cup\{\infty\}$ such that the rays $\R_+\Upsilon_m$
are extremal in $\mcal C^i$ and $\lim\limits_{m\rightarrow\infty}\Delta_m=\Delta_\infty$, where $\Upsilon_m=K_X+\Delta_m+A$.
I achieve contradiction by showing that for some $m\gg0$ there is a point $\Upsilon_m'\in\mcal C^i$ such that $\Upsilon_m\in(\Upsilon_\infty,\Upsilon_m')$,
so that the ray $\R_+\Upsilon_m$ cannot be extremal in $\mcal C^i$.

I make the following observations. Let $W$ be the vector space spanned by the components of $K_X$, $A$ and by the prime divisors in $V$.
I claim that we can assume that $[\Upsilon_m]\neq[\Upsilon_\infty]$ for all $m\gg0$. Assuming the contrary and passing to a subsequence, we have
$[\Upsilon_m]=[\Upsilon_\infty]$ for all $m$, and let $\phi\colon W\rightarrow N^1(X)$ be the map sending a divisor to its numerical class.
Then since $\phi^{-1}([\Upsilon_\infty])$ is an affine subspace of $W$,
there is a divisor $\Phi_m\in\phi^{-1}([\Upsilon_\infty])$ such that $\Upsilon_m\in(\Upsilon_\infty,\Phi_m)$, and note that $\Phi_m$ has $\mcal P^i$
since $[\Phi_m]=[\Upsilon_\infty]$. Since $\R_+(K_X+A+\mcal L_V)$ is a rational polyhedral cone, for $m\gg0$ we have
$$[\Upsilon_\infty,\Upsilon_m]\subsetneq\big(\Upsilon_\infty+\R_+(\Upsilon_m-\Upsilon_\infty)\big)\cap\R_+(K_X+A+\mcal L_V),$$
so in particular there exists a divisor $\Upsilon_m'\in(\Upsilon_m,\Phi_m)\cap\R_+(K_X+A+\mcal L_V)$ which has $\mcal P^i$ since
$[\Upsilon_m']=[\Upsilon_\infty]$. Also $\mult_G\Upsilon_m'=\mult_G\Upsilon_\infty$, so
$\Upsilon_m'\in\mcal C^i$ and $\R_+\Upsilon_m\subset\relint(\R_+\Upsilon_\infty+\R_+\Upsilon_m')$, which implies that $\R_+\Upsilon_m$ is not
an extremal ray of $\mcal C^i$.

Further, I claim that in order to achieve contradiction, it is enough to prove that for $m\gg0$, there is a class $\widehat\Phi_m\in N^1(X)$ which has
$\mcal P^i$ and a real number $0<t<1$ such that $[\Upsilon_m]=t[\Upsilon_\infty]+(1-t)\widehat\Phi_m$. To see this, let
$\Phi_m=\frac{1}{1-t}(\Upsilon_m-t\Upsilon_\infty)$. Then $[\Phi_m]=\widehat\Phi_m$, and thus $\Phi_m$ has $\mcal P^i$. Since
$\Upsilon_m\in(\Upsilon_\infty,\Phi_m)$ and $\mcal P_{V,A}^i$ is convex, we finish the proof of the claim as above.

Therefore, I am allowed to, and will without explicit mention in the proofs of Theorems \ref{thm:EimpliesLGA} and \ref{cor:linear}, increase $V$ and
consider divisors up to $\R$-linear equivalence, since this does not change their numerical classes.
\end{rem}

\section{Outline of the induction}

As part of the induction, I will prove the following three theorems.

\begin{thmA}
Let $X$ be a smooth projective variety, and let $D_i=k_i(K_X+\Delta_i+A)\in\Div(X)$, where $A$ is an ample $\Q$-divisor and $(X,\Delta_i+A)$ is a log
smooth log canonical pair with $|D_i|\neq\emptyset$ for $i=1,\dots,\ell$ . Then the adjoint ring $R(X;D_1,\dots,D_\ell)$ is finitely generated.
\end{thmA}

\begin{thmB}
Let $X$ be a smooth projective variety, $B$ a simple normal crossings divisor and
$A$ a general ample $\Q$-divisor on $X$. Let $V\subset\Div(X)_\R$ be the vector space spanned by the components of $B$.
Then for any component $G$ of $B$, the set $\mcal B_{V,A}^{G=1}$ is a rational polytope, and we have
$$\mcal B_{V,A}^{G=1}=\{\Phi\in\mcal L_V:\mult_G\Phi=1,\sigma_G\|K_X+\Phi+A\|=0\}.$$
\end{thmB}

\begin{thmC}
Let $X$ be a smooth projective variety, $B$ a simple normal crossings divisor and
$A$ a general ample $\Q$-divisor on $X$. Let $V\subset\Div(X)_\R$ be the vector space spanned by the components of $B$.
Then the set $\mcal E_{V,A}$ is a rational polytope, and we have
$$\mcal E_{V,A}=\{\Phi\in\mcal L_V:|K_X+\Phi+A|_\R\neq\emptyset\}.$$
\end{thmC}

Let me give an outline of the paper, where e.g.\ ``Theorem A$_n$'' stands for ``Theorem A in dimension $n$.''

Sections \ref{sec:convex} and \ref{sec:3} develop tools to deal with algebras of higher rank and to test whether functions are
piecewise linear. Section \ref{sec:diophant} contains results from Diophantine approximation which will be necessary
in Sections \ref{plt}, \ref{sec:stable} and \ref{sec:non-vanishing}.

In \S\ref{sec:stable} I prove that Theorems A$_{n-1}$ and C$_{n-1}$ imply Theorem B$_n$, and this part of the
proof uses techniques from \S\ref{plt}.
In \S\ref{sec:non-vanishing} I show how Theorems A$_{n-1}$, B$_n$ and C$_{n-1}$ imply Theorem C$_n$.
Finally, Sections \ref{plt} and \ref{proofmain} contain the proof that Theorems A$_{n-1}$, B$_n$
and C$_{n-1}$ imply Theorem A$_n$. Section \ref{plt} is technically the most difficult part of the proof,
whereas \S\ref{proofmain} contains the main new idea on which the whole paper is based.

At the end of this section, let me sketch the proofs of Theorems A, B and C when $X$ is a curve of genus $g$. Since by Riemann-Roch the condition that
a divisor $E$ on $X$ is pseudo-effective is equivalent to $\deg E\geq0$, and this condition is linear on the coefficients, this proves Theorem C.
For Theorem A, when $g\geq1$ every divisor $D_i$ is ample, and when $g=0$, since $\deg D_i\geq0$ we have that $D_i$
is basepoint free, so the statement follows from \cite[2.8]{HK00}. Furthermore, this shows that every divisor of the form $K_X+\Phi+A$ is semiample,
so $\mcal B_{V,A}^G=\mcal E_{V,A}$ and Theorem B follows.

\section{Convex geometry}\label{sec:convex}

Results of this section will be used in the rest of the paper to study relations between superadditive and superlinear functions, and to test their
piecewise linearity. The following proposition can be found in \cite{HUL93} and I add the proof for completeness.

\begin{pro}\label{Lip}
Let $\mcal{C}\subset\R^n$ be a cone and $f\colon\mcal{C}\rightarrow\R$ a concave function.
Then $f$ is locally Lipschitz continuous on the topological interior of $\mcal{C}$ with respect to any norm $\|\cdot\|$ on $\R^n$.

In particular, if $\mcal{C}$ is rational polyhedral and $g\colon\mcal{C}_\Q\rightarrow\Q$ is a superadditive map which satisfies
$g(\lambda x)=\lambda g(x)$ for all $x\in\mcal{C}_\Q$, $\lambda\in\Q_+$, then $g$ extends in a unique way to a superlinear
function on $\mcal{C}$.
\end{pro}
\begin{proof}
Since $f$ is locally Lipschitz if and only if $-f$ is locally Lipschitz, we can assume $f$ is convex. Fix $0\neq x=(x_1,\dots,x_n)\in\Int\mcal{C}$, and
let $\Delta=\{(y_1,\dots,y_n)\in\R_+^n:\sum y_i\leq1\}$. It is easy to check that translations of the domain do not affect the result,
so we may assume $x\in\Int\Delta\subset\Int\mcal{C}$.

First, let us prove that $f$ is locally bounded from above around $x$. Let $e_i$ be the standard basis vectors of $\R^n$ and set
$M=\max\{f(0),f(e_1),\dots,f(e_n)\}$. If $y=(y_1,\dots,y_n)\in\Delta$ and $y_0=1-\sum y_i\geq0$, then
$$f(y)=f\big(\sum y_ie_i+y_0\cdot0\big)\leq\sum y_if(e_i)+y_0f(0)\leq M.$$
Now choose $\delta$ such that $B(x,2\delta)\subset\Int\Delta$. Again by
translating the domain and composing $f$ with a linear function we may assume that $x=0$ and $f(0)=0$. Then for all $y\in B(0,2\delta)$ we have
$$-f(y)=-f(y)+2f(0)\leq-f(y)+\big(f(y)+f(-y)\big)=f(-y)\leq M,$$
so $|f|\leq M$ on $B(0,2\delta)$.

Set $L=2M/\delta$. Fix $u,v\in B(0,\delta)$, and set $\alpha=\frac1\delta\|v-u\|$ and $w=v+\frac1\alpha(v-u)\in B(0,2\delta)$, so that
$v=\frac{\alpha}{\alpha+1}w+\frac{1}{\alpha+1}u$. Then by convexity,
\begin{align*}
f(v)-f(u)&\leq\textstyle\frac{\alpha}{\alpha+1}f(w)+\frac{1}{\alpha+1}f(u)-f(u)\\
&=\textstyle\frac{\alpha}{\alpha+1}\big(f(w)-f(u)\big)\leq2M\alpha=L\|v-u\|,
\end{align*}
and similarly $f(u)-f(v)\leq L\|u-v\|$, which proves the first claim.

For the second one, observe that the sup-norm $\|\cdot\|_{\infty}$ takes values in $\Q$ on $\mcal{C}_\Q$. The proof above applied to the interior
of $\mcal{C}$ and to the relative interiors of the faces of $\mcal{C}$ shows that $g$ is locally Lipschitz, and the claim follows.
\end{proof}

The following result is classically referred to as Gordan's lemma, and I often use it without explicit mention.

\begin{lem}\label{lem:gordan}
Let $\mcal{S}\subset\N^r$ be a finitely generated monoid and let $\mcal{C}\subset\R^r$ be a rational polyhedral cone. Then the monoid
$\mcal{S}\cap\mcal{C}$ is finitely generated.
\end{lem}
\begin{proof}
Assume first that $\dim\mcal{C}=r$.
Let $\ell_1,\dots,\ell_m$ be linear functions on $\R^r$ with integral coefficients such that $\mcal{C}=\bigcap_{i=1}^m\{z\in\R^r:\ell_i(z)\geq0\}$,
and define $\mcal{S}_0=\mcal{S}$ and $\mcal{S}_i=\mcal{S}_{i-1}\cap\{z\in\R^r:\ell_i(z)\geq0\}$ for $i=1,\dots,m$; observe that
$\mcal{S}\cap\mcal{C}=\mcal{S}_m$. Assuming by induction that $\mcal{S}_{i-1}$ is finitely generated, by \cite[Theorem 4.4]{Swa92} we have that
$\mcal{S}_i$ is finitely generated.

Now assume $\dim\mcal{C}<r$ and let $\mcal{H}$ be a rational hyperplane containing $\mcal{C}$. Let $\ell$ be a linear function with integral
coefficients such that $\mcal{H}=\ker(\ell)$. The monoid $\mcal{S}\cap\mcal{H}$ is finitely generated by the first part of the proof applied to
the functions $\ell$ and $-\ell$. Now we conclude by induction on $r$.
\end{proof}

The next lemma will turn out to be indispensable and it shows that it is enough to check additivity of a superadditive map at one point only.

\begin{lem}\label{linear}
Let $\mcal{S}=\sum\N e_i$ be a finitely generated monoid and let $f\colon\mcal{S}\rightarrow G$ be a superadditive map to an
ordered monoid $G$ (respectively let $f\colon\mcal S_\R\rightarrow V$ be a superlinear map to a cone $V$ with an ordering). Assume that there is a point
$s_0=\sum s_ie_i\in\mcal{S}$ with all $s_i>0$, such that $f(s_0)=\sum s_if(e_i)$ and $f(\kappa s_0)=\kappa f(s_0)$ for every positive integer
$\kappa$ (respectively there is a point $s_0=\sum s_ie_i\in\mcal S_\R$ with all $s_i>0$ such that $f(s_0)=\sum s_if(e_i)$).
Then $f$ is additive (respectively linear).
\end{lem}
\begin{proof}
I prove the lemma when $f$ is superadditive, the other claim is proved analogously. For $p=\sum p_ie_i\in\mcal{S}$, choose $\kappa_0\in\Z_{>0}$
so that $\kappa_0s_i\geq p_i$ for all $i$. Then
\begin{align*}
\sum\kappa_0s_if(e_i)&=\kappa_0f(s_0)=f(\kappa_0s_0)\geq f(p)+\sum f\big((\kappa_0s_i-p_i)e_i\big)\\
&\geq\sum p_if(e_i)+\sum(\kappa_0s_i-p_i)f(e_i)=\sum\kappa_0s_if(e_i).
\end{align*}
Therefore all inequalities are equalities and $f(p)=\sum p_if(e_i)$.
\end{proof}

Now we are ready to prove the main result of this section, which will be crucial \S\ref{plt}.

\begin{thm}\label{piecewise}
Let $f$ be a superlinear function on a polyhedral cone $\mcal{C}\subset\R^{r+1}$ with $\dim\mcal{C}=r+1$, such that for every
$2$-plane $H\subset\R^{r+1}$ the function $f_{|H\cap\mcal{C}}$ is piecewise linear. Then $f$ is piecewise linear.
\end{thm}
\begin{proof}
I prove the lemma by induction on $r$. In the proof, $\|\cdot\|$ denotes the standard Euclidean norm and $S^r\subset\R^{r+1}$ is
the unit sphere.\\[2mm]
{\em Step 1.\/}
Fix a ray $R\subset\mcal{C}$. In this step I prove there is a collection of $(r+1)$-dimensional polyhedral cones
$\{\mcal{C}_{\alpha}\}_{\alpha\in\mcal I_R}$ with $\mcal C_\alpha\subset\mcal C$, such that
\begin{enumerate}
\item[(i)] for every $c\in\mcal{C}\backslash R$ there is $\alpha\in\mcal I_R$ such that $R\subsetneq\mcal{C}_{\alpha}\cap(R+\R_+c)$,
\item[(ii)] for every $\alpha\in\mcal I_R$ the map $f_{|\mcal{C}_{\alpha}}$ is linear,
\item[(iii)] for every two distinct $\alpha,\beta\in\mcal I_R$ the linear extensions of $f_{|\mcal{C}_{\alpha}}$ and $f_{|\mcal{C}_{\beta}}$ to $\R^{r+1}$
are different.
\end{enumerate}

Fix $c\in\mcal C\backslash R$, and choose a hyperplane $\mcal H_r\supset R+\R_+c$. By induction there is an $r$-dimensional polyhedral cone
$\mcal{C}_{(r)}=\sum_{i=1}^r\R_+e_i\subset\mcal H_r\cap\mcal{C}$ such that $R\subsetneq\mcal{C}_{(r)}\cap(R+\R_+c)$ and $f_{|\mcal{C}_{(r)}}$
is linear. Then $f(e_0)=\sum_{i=1}^rf(e_i)$, where $e_0=\sum_{i=1}^re_i$, and
let $\mcal P$ be a $2$-plane such that $\mcal P\cap\mcal H_r=\R_+e_0$. By assumption, there is a point $e_{r+1}\in(\mcal P\cap\mcal{C})\backslash\R_+e_0$
such that $f|_{\R_+e_0+\R_+e_{r+1}}$ is linear, and in particular $f(e_0+e_{r+1})=f(e_0)+f(e_{r+1})$.
Setting $\mcal{C}_{(r+1)}=\sum_{i=1}^{r+1}\R_+e_i$, we have
$$f\big(\sum\nolimits_{i=1}^{r+1}e_i\big)=f(e_0+e_{r+1})=f(e_0)+f(e_{r+1})=\sum f(e_i),$$
so the map $f_{|\mcal{C}_{(r+1)}}$ is linear by Lemma \ref{linear}.
Let $\ell$ be the linear extension of $f_{|\mcal{C}_{(r+1)}}$ to $\R^{r+1}$, and set
$\mcal{C}_c=\{z\in\mcal{C}:f(z)=\ell(z)\}$. I claim $\mcal{C}_c$ is a closed cone. To that end,
if $u,v\in\mcal{C}_c$, then there are real numbers $u_i,v_i$ such that $u=\sum_{i=1}^{r+1}u_ie_i$ and $v=\sum_{i=1}^{r+1}v_ie_i$, and set
$e=\sum_{i=1}^{r+1}(1+|u_i|+|v_i|)e_i$. Note that $e$ and $e+u+v$ belong to $\mcal{C}_{(r+1)}\subset\mcal{C}_c$, thus
$$f(e+u+v)=\ell(e+u+v)=\ell(e)+\ell(u)+\ell(v)=f(e)+f(u)+f(v),$$
so $f$ is linear on the cone $\mcal{C}_{(r+1)}+\R_+u+\R_+v$ by Lemma \ref{linear}. In particular $f(u+v)=f(u)+f(v)=\ell(u)+\ell(v)=\ell(u+v)$,
hence $\mcal{C}_c$ is a cone. Denote by $\mcal Q$ the closure of $\mcal{C}_c$, and fix $q\in\mcal Q$. Then for every
$p\in\mcal{C}_c$, the function $f|_{\R_+p+\R_+q}$ is piecewise linear by assumption, and in particular continuous.
Since $\R_+p+R_{>0}q\subset\Int\mcal Q\subset\mcal{C}_c$, and $f$ and $\ell$ agree on $\mcal{C}_c$, this implies that $f$ is linear on
$\R_+p+\R_+q$, so $\mcal{C}_c$ is closed. Now, by varying $c\in\mcal C\backslash R$ we obtain the desired collection of cones.\\[2mm]
{\em Step 2.\/}
In this step I prove that $\mcal I_R$ is a finite set for every ray $R\subset\mcal C$.

Fix $R$, and arguing by contradiction assume that $\mcal I_R$ is infinite; we can assume that $\N\subset\mcal I_R$.
For each $n\in\N$ choose $x_n\in\Int\mcal{C}_n\backslash R$, and denote $\mcal H_n=(R+\R_+x_n)\cup(-R+\R_+x_n)$.
Let $R_n\subset\mcal H_n$ be the unique ray orthogonal to $R$ and let $S^r\cap R_n=\{Q_n\}$. By passing to a subsequence, we can assume that points
$Q_n$ converge to $Q_\infty\in S^r$, and set $\mcal H_\infty=(R+\R_+Q_\infty)\cup(-R+\R_+Q_\infty)$.

Pick $x\in R\backslash\{0\}$. By assumption, there is a point $y\in\mcal H_\infty\backslash R$ such that $f|_{R+\R_+y}$ is linear, and in particular
$f(x+y)=f(x)+f(y)$. If $\mcal{H}$ is a hyperplane such that $\mcal{H}\cap(\R x+\R y)=\R(x+y)$ and $\mcal H\cap\mcal C$ is an $r$-dimensional cone,
by induction there are finitely many $r$-dimensional polyhedral cones
$\mcal Q_i\subset\mcal{H}\cap\mcal{C}$ containing $x+y$ such that the map $f_{|\mcal Q_i}$ is linear for every $i$,
and for every $c\in(\mcal H\cap\mcal{C})\backslash\R_+(x+y)$ we have $\R_+(x+y)\subsetneq\bigcup_i\mcal Q_i\cap(\R_+(x+y)+\R_+c)$.
If $g_{ij}$ are finitely many generators of $\mcal Q_i$, then
$$f\big(\sum\nolimits_j g_{ij}+x+y\big)=\sum\nolimits_j f(g_{ij})+f(x+y)=\sum f(g_{ij})+f(x)+f(y),$$
so $f$ is linear on the cone $\widetilde{\mcal Q}_i=\mcal Q_i+\R_+x+\R_+y$ by Lemma \ref{linear}.

I claim that for every $c\in\mcal C\backslash\R_+(x+y)$ we have $\R_+(x+y)\subsetneq\bigcup_i\widetilde{\mcal Q}_i\cap(\R_+(x+y)+\R_+c)$, and in particular
there exists $0<\varepsilon\ll1$ such that $\R_+B(x+y,\varepsilon)\cap\mcal{C}=\R_+B(x+y,\varepsilon)\cap\bigcup_i\widetilde{\mcal Q_i}$. To that end,
if $c\in\mcal H$, then the claim follows by assumption on the cones $\mcal Q_i$. Otherwise, let $\{t\}=\mcal H\cap(x,y)$, let $c_m\in(c,t)$
be a sequence converging to $t$, and without loss of generality assume that $y$ and all $c_m$ are on the same side of $\mcal H$. If
$\{z_m\}=\mcal H\cap(x,c_m)$, then $c_m-z_m=\alpha_m(z_m-x)$ and $t-x=\beta(y-t)$ for some $\alpha_m,\beta\in\R_{>0}$, and thus
$c_m=(1-\alpha_m\beta)(t+\frac{\alpha_m+1}{1-\alpha_m\beta}(z_m-t))+\alpha_m\beta y$. We have $z_m\rightarrow t$ and $\alpha_m\rightarrow0$ when $m\rightarrow\infty$, hence
$t+\frac{\alpha_m+1}{1-\alpha_m\beta}(z_m-t)\in\bigcup_i\mcal Q_i$ for $m\gg0$ by assumption, so $c_m\in\bigcup_i\widetilde{\mcal Q}_i$,
proving the claim.

Since $\lim\limits_{n\rightarrow\infty}Q_n=Q_\infty$, we obtain $\mcal H_n\cap\Int B(x+y,\varepsilon)\neq\emptyset$ for $n\gg0$, and therefore,
by the claim above and by passing to a subsequence, there is an index $i_0$ such that $\widetilde{\mcal Q}_{i_0}$ intersects all $\mcal H_n\backslash R$.
In particular, since $(x,x_n)\subset\Int\mcal C_n$ by the choice of $x_n$, we have $\widetilde{\mcal Q}_{i_0}\cap\Int\mcal{C}_n\neq\emptyset$,
and therefore $\widetilde{\mcal Q}_{i_0}\cap\mcal{C}_n$ is an $(r+1)$-dimensional cone for all $n$.
Thus the linear extensions of all $f_{|\mcal{C}_n}$ to $\R^{r+1}$ are the same since they coincide with
the linear extension of $f_{|\widetilde{\mcal Q}_{i_0}}$, a contradiction.\\[2mm]
{\em Step 3.\/}
Therefore, for every ray $R\subset\mcal{C}$ the map $f|_{\mcal{C}_i}$ is linear for $i\in\mcal I_R$, and
there is small ball $B_R$ centred at $R\cap S^r$ such that $B_R\cap S^r\cap\mcal{C}=B_R\cap S^r\cap\bigcup_{i\in\mcal I_R}\mcal{C}_i$.
There are finitely many open sets $\Int B_R$ which cover the compact set $S^r\cap\mcal{C}$, and therefore we can choose finitely many cones
$\mcal C_i$ with $\mcal{C}=\bigcup_i\mcal C_i$. Note that by the construction in Step 1, the linear extensions of $f_{|\mcal{C}_i}$ to $\R^{r+1}$
are pairwise different.

It remains to show that all $\mcal C_i$ are polyhedral cones. Assume that $\mcal C_{i_0}$ is not polyhedral for some $i_0$, and let $L_n$
be its distinct extremal rays for $n\in\N$. If infinitely many $L_n$ do not belong to any other cone $\mcal C_i$, then passing to a subsequence I
can assume that they belong to a face of $\mcal C$, and we derive contradiction by induction. Therefore, I can assume that for every $n\gg0$ there
is an index $i_n\neq i_0$ such that $L_n\subset\mcal C_{i_n}$. Passing to a subsequence, there is an index $j_0\neq i_0$ such that
$L_n\subset\mcal C_{i_0}\cap\mcal C_{j_0}$ for all $n$. As before, we can assume that there does not exist a hyperplane containing infinitely many $L_n$,
so there are finitely many indices $k$ such that $\dim(\sum L_k)=r+1$. Thus the linear extensions of $f_{|\mcal{C}_{i_0}}$ and $f_{|\mcal{C}_{j_0}}$
to $\R^{r+1}$ are the same since they coincide with the linear extension of $f|_{\sum L_k}$, a contradiction.
\end{proof}

\section{Higher rank algebras}\label{sec:3}

\begin{dfn}\label{dfn:1}
Let $X$ be a variety, $\mcal{S}\subset\N^r$ a finitely generated monoid, let $\mu\colon\mcal{S}\rightarrow\WDiv(X)^{\kappa\geq0}$ be an additive map
and $\bMob_\mu\colon\mcal S\rightarrow\bMob(X)$ the subadditive map defined by $\bMob_\mu(s)=\bMob(\mu(s))$ for every $s\in\mcal S$. Then
$$R(X,\mu(\mcal{S}))=\bigoplus_{s\in\mcal{S}}H^0(X,\OO_X(\mu(s)))$$
is the {\em divisorial $\mcal S$-graded algebra associated to $\mu$.\/} The {\em b-divisorial $\mcal S$-graded algebra associated to $\mu$\/} is
$$R(X,\bMob_\mu(\mcal{S}))=\bigoplus_{s\in\mcal{S}}H^0(X,\OO_X(\bMob_\mu(s))),$$
and we obviously have $R(X,\bMob_\mu(\mcal S))\simeq R(X,\mu(\mcal S))$. If $e_1,\dots,e_\ell$ are generators of $\mcal S$ and if
$\mu(e_i)=k_i(K_X+\Delta_i)$, where $\Delta_i$ is an effective $\Q$-divisor for every $i$, the algebra $R(X,\mu(\mcal S))$ is the {\em adjoint ring
associated to $\mu$\/}.
\end{dfn}

\begin{rem}
When $\mcal S=\bigoplus_{i=1}^\ell\N e_i$ is a simplicial cone, the algebra $R(X,\mu(\mcal S))$ is denoted also by $R(X;\mu(e_1),\dots,\mu(e_\ell))$.
If $\mcal S'$ is a finitely generated submonoid of $\mcal S$, $R(X,\mu(\mcal S'))$ is used
to denote $R(X,\mu_{|\mcal S'}(\mcal S'))$. If $\mcal S$ is a submonoid of $\WDiv(X)^{\kappa\geq0}$ and $\iota\colon\mcal S\rightarrow\mcal S$
is the identity map, $R(X,\mcal S)$ is used to denote $R(X,\iota(\mcal S))$.
\end{rem}

\begin{rem}\label{rem:1}
Algebras considered in this paper are {\em algebras of sections\/} when varieties are smooth. I will occasionally, and without explicit mention,
view them as algebras of rational functions, in particular to be able to write $H^0(X,D)\simeq H^0(X,\Mob(D))\subset k(X)$.

Assume now that $X$ is smooth, $D\in\Div(X)$ and that $\Gamma$ is a prime divisor
on $X$. If $\sigma_\Gamma$ is the global section of $\OO_X(\Gamma)$ such that $\ddiv \sigma_\Gamma=\Gamma$, from the exact sequence
$$0\rightarrow H^0(X,\OO_X(D-\Gamma))\stackrel{\cdot \sigma_\Gamma}{\longrightarrow}H^0(X,\OO_X(D))\stackrel{\rho_{D,\Gamma}}{\longrightarrow}
H^0(\Gamma,\OO_\Gamma(D))$$
we define $\res_\Gamma H^0(X,\OO_X(D))=\im(\rho_{D,\Gamma})$. For $\sigma\in H^0(X,\OO_X(D))$, denote
$\sigma_{|\Gamma}=\rho_{D,\Gamma}(\sigma)$. Observe that
\begin{equation}\label{eq:1}
\ker(\rho_{D,\Gamma})=H^0(X,\OO_X(D-\Gamma))\cdot\sigma_\Gamma,
\end{equation}
and that $\res_\Gamma H^0(X,\OO_X(D))=0$ if $\Gamma\subset\Bs|D|$. If $D\sim D'$ is such that the restriction $D'_{|\Gamma}$ is defined,
then
$$\res_\Gamma H^0(X,\OO_X(D))\simeq\res_\Gamma H^0(X,\OO_X(D'))\subset H^0(\Gamma,\OO_\Gamma(D'_{|\Gamma})).$$
The {\em restriction of $R(X,\mu(\mcal S))$ to $\Gamma$\/} is defined as
$$\res_\Gamma R(X,\mu(\mcal S))=\bigoplus_{s\in\mcal{S}}\res_\Gamma H^0(X,\OO_X(\mu(s))).$$
This is an $\mcal S$-graded, not necessarily divisorial algebra.
\end{rem}

The following lemma summarises basic properties of higher rank finite generation.

\begin{lem}\label{lem:1}
Let $\mcal{S}\subset\N^n$ be a finitely generated monoid and let $R=\bigoplus_{s\in\mcal S}R_s$ be an $\mcal S$-graded algebra.
\begin{enumerate}
\item Let $\mcal S'$ be a truncation of $\mcal S$. If the $\mcal S'$-graded algebra $R'=\bigoplus_{s\in\mcal S'}R_s$
is finitely generated over $R_0$, then $R$ is finitely generated over $R_0$.

\item Assume furthermore that $\mcal S$ is saturated and let $\mcal S''\subset\mcal S$ be a finitely generated saturated submonoid.
If $R$ is finitely generated over $R_0$,
then the $\mcal S''$-graded algebra $R''=\bigoplus_{s\in\mcal S''}R_s$ is finitely generated over $R_0$.

\item Let $X$ be a variety and let
$\mu\colon\mcal{S}\rightarrow\WDiv(X)^{\kappa\geq0}$ be an additive map. If there exists a rational polyhedral subdivision
$\mcal{S}_\R=\bigcup_{i=1}^k\Delta_i$ such that $\bMob_{\mu|\Delta_i\cap\mcal{S}}$ is additive up to truncation for each $i$, then
$R(X,\mu(\mcal{S}))$ is finitely generated.
\end{enumerate}
\end{lem}
\begin{proof}
If $\mcal S=\sum_{i=1}^n\N e_i$ and $\mcal S'=\sum_{i=1}^n\N\kappa_i e_i$, then for any $f\in R$ we have $f^{\kappa_1\cdots\kappa_n}\in R'$,
so $R$ is an integral extension of $R'$, and (1) follows by the theorem of Emmy Noether on finiteness of integral closure.

Claim (2) is \cite[4.8]{ELMNP}.

For (3), denote $\m=\bMob_\mu$ and $\mcal S_i=\Delta_i\cap\mcal{S}$. By Lemma \ref{lem:gordan}, choose a set of generators
$\{e_{ij}:j=1,\dots,k_i\}$ of $\mcal{S}_i$, and let $\kappa$ be a positive integer such that $\m$ is additive on each $\mcal S_i^{(\kappa)}$.
By (1), it is enough to show that the algebra $R(X,\m(\mcal{S}^{(\kappa)}))$ is finitely generated.
To that end, let $Y\rightarrow X$ be a model such that $\m(\kappa e_{ij})$ descend to $Y$ for all $i,j$. Then it is easy to see that $\m(s)$ descends
to $Y$ for every $s\in\bigcup_i\mcal S_i^{(\kappa)}$, and thus $R(X,\m(\mcal{S}_i^{(\kappa)}))\simeq\bigoplus_{s\in\mcal{S}_i^{(\kappa)}}H^0(Y,\m(s)_Y)$
for every $i$. Since $R(Y;\m_Y(e_{i1}),\dots,\m_Y(e_{ik_i}))$ is finitely generated by \cite[2.8]{HK00}, so is the algebra
$R(X,\m(\mcal{S}_i^{(\kappa)}))$ by projection. Since $\mcal{S}_i^{(\kappa)}=\Delta_i\cap\mcal{S}^{(\kappa)}$,
the union of sets of generators of all $R(X,\m(\mcal{S}_i^{(\kappa)}))$ spans $R(X,\m(\mcal{S}^{(\kappa)}))$.
\end{proof}

I will need the next result in the proof of Proposition \ref{pro:1} and in \S\ref{plt}.

\begin{lem}\label{concave}
Let $X$ be a variety, let $\mcal{S}\subset\N^r$ be a finitely generated monoid and let $f\colon\mcal{S}\rightarrow G$ be
a superadditive map to a monoid $G$ which is a subset of $\WDiv(X)$ or $\bMob(X)$, such that for every $s\in\mcal{S}$
the map $f_{|\N s}$ is additive up to truncation.

Then there is a unique superlinear function $f^\sharp\colon\mcal{S}_\R\rightarrow G_\R$ such that
for every $s\in\mcal{S}$ there exists $\lambda_s\in\Z_{>0}$ with $f(\lambda_s s)=f^\sharp(\lambda_s s)$.
Furthermore, if $\mcal{C}\subset\R^r$ is a rational polyhedral cone, then $f_{|\mcal{C}\cap\mcal{S}}$ is additive up to
truncation if and only if $f^\sharp_{|\mcal C}$ is linear.

In particular, if $G=\Div(X)$ and $f=\bMob_\mu$, where $\mu\colon\mcal S\rightarrow\Div(X)$ is an additive map, then
\begin{equation}\label{eq:2}
f^\sharp(s)=\overline{\mu(s)}-\sum\big(\ord_E\|\mu(s)\|\big)E,
\end{equation}
where the sum runs over all geometric valuations $E$ on $X$.
\end{lem}
\begin{proof}
The construction will show that $f^\sharp$ is the unique function with the stated properties.
To start, fix a point $s\in\mcal{S}_\Q$, choose $\kappa\in\Z_{>0}$ so that $\kappa s\in\mcal{S}$ and $f_{\N\kappa s}$ is additive, and set
$$f^\sharp(s)=f(\kappa s)/\kappa.$$
This is well-defined: if $\kappa'$ is another positive integer such that $\kappa's\in\mcal{S}$ and $f_{\N\kappa's}$ is additive, then
$\kappa f(\kappa's)=f(\kappa\kappa's)=\kappa'f(\kappa s)$, so $f(\kappa s)/\kappa=f(\kappa's)/\kappa'$.

Now fix $\xi\in\Q_+$, and let $\lambda$ be a positive integer such that $\lambda\xi\in\N$, $\lambda\xi s\in\mcal{S}$ and
$f_{\N\lambda\xi s}$ is additive. Then
$$f^\sharp(\xi s)=f(\lambda\xi s)/\lambda=\xi f(\lambda\xi s)/\lambda\xi=\xi f^\sharp(s),$$
so $f^\sharp$ is positively homogeneous with respect to rational scalars.
Further, let $s_1,s_2\in\mcal{S}_\Q$ and $\kappa\in\Z_{>0}$ be such that
$f(\kappa s_1)=f^\sharp(\kappa s_1)$, $f(\kappa s_2)=f^\sharp(\kappa s_2)$ and
$f\big(\kappa(s_1+s_2)\big)=f^\sharp\big(\kappa(s_1+s_2)\big)$. By superadditivity of $f$ we have
$f(\kappa s_1)+f(\kappa s_2)\leq f\big(\kappa(s_1+s_2)\big)$,
so dividing this by $\kappa$ we obtain superadditivity of $f^\sharp$.

Let $E$ be any divisor on $X$, respectively any geometric valuation $E$ over $X$, when $G\subset\WDiv(X)$, respectively $G\subset\bMob(X)$.
Consider the function $f^\sharp_E$ given by $f^\sharp_E(s)=\mult_Ef^\sharp(s)$.
Proposition \ref{Lip} applied to each $f^\sharp_E$ shows that $f^\sharp$ extends to a superlinear function on $\mcal{S}_\R$.
Now \eqref{eq:2} is just a restatement of the definition of $f^\sharp$ when $f=\bMob_\mu$.

As for the statement on cones, necessity is clear. Now assume $f^\sharp|_{\mcal{C}}$ is linear and let $e_i$ be finitely many generators of
$\mcal{C}\cap\mcal{S}$, cf.\ Lemma \ref{lem:gordan}. Let $s_0=\sum e_i$, and let $\mu$ be a positive integer such that
$f(\mu s_0)=f^\sharp(\mu s_0)$ and $f(\mu e_i)=f^\sharp(\mu e_i)$ for all $i$. Then from $f^\sharp(s_0)=\sum f^\sharp(e_i)$ we obtain
$f(\mu s_0)=\sum f(\mu e_i)$, and Lemma \ref{linear} implies that $f^\sharp$ is additive on $(\mcal C\cap\mcal S)^{(\mu)}$.
\end{proof}

\begin{dfn}
In the context of Lemma \ref{concave}, $f^\sharp$ is called {\em the straightening of $f$\/}.
\end{dfn}

\begin{pro}\label{pro:1}
Let $X$ be a variety, $\mcal S\subset\N^r$ a finitely generated saturated monoid and $\mu\colon\mcal S\rightarrow \WDiv(X)^{\kappa\geq0}$ an
additive map. Let $\mcal L$ be a finitely generated submonoid of $\mcal S$ and assume $R(X,\mu(\mcal S))$ is finitely generated. Then
$R(X,\mu(\mcal L))$ is finitely generated. Moreover, the map $\m=\bMob_{\mu|\mcal L}$ is rationally piecewise additive up to truncation.
In particular, there is a positive integer $p$ such that $\bMob_\mu(ips)=i\bMob_\mu(ps)$ for every $i\in\N$ and every $s\in\mcal L$.
\end{pro}
\begin{proof}
Denote $\mcal M=\mcal L_\R\cap\N^r$. By Lemma \ref{lem:1}(2), $R(X,\mu(\mcal M))$
is finitely generated, and by the proof of \cite[4.1]{ELMNP}, there is a finite rational polyhedral subdivision
$\mcal M_\R=\bigcup\Delta_i$ such that for every geometric valuation $E$ on $X$, the map $\ord_E\|\cdot\|$ is linear on
$\Delta_i$ for every $i$. Since for every saturated rank $1$ submonoid $\mcal R\subset\mcal M$ the algebra $R(X,\mu(\mcal R))$ is
finitely generated by Lemma \ref{lem:1}(2), the map $\m_{|\mcal R}$ is additive up to truncation by \cite[2.3.53]{Cor07},
and thus there is the well-defined straightening $\m^\sharp\colon\mcal L_\R\rightarrow\bMob(X)_\R$ since $\mcal M_\R=\mcal L_\R$.
Then equation \eqref{eq:2} implies that $\m^\sharp|_{\Delta_i}$ is linear for every $i$, hence by Lemma \ref{concave}
the map $\m$ is rationally piecewise additive up to truncation. Therefore $R(X,\mu(\mcal L))$ is finitely generated by Lemma \ref{lem:1}(3).
\end{proof}

The following lemma shows that finite generation implies certain boundedness on the convex geometry of boundaries.

\begin{lem}\label{bounded}
Let $X$ be a smooth variety of dimension $n$, let $B$ be a simple normal crossings divisor and let
$A$ be a general ample $\Q$-divisor on $X$. Let $V\subset\Div(X)_\R$ be the vector space spanned by the components of $B$.
Assume Theorems A$_n$ and C$_n$.

Then for each prime divisor $G$ on $X$, the set $\mcal B_{V,A}^G$ is a rational polytope. Furthermore, there exists
a positive integer $r$ such that:
\begin{enumerate}
\item for each prime divisor $G$ on $X$, for every $\Phi\in(\mcal B_{V,A}^G)_\Q$, and for every positive integer $k$ such that
$k(K_X+\Phi+A)/r$ is Cartier, we have $G\not\subset\Fix|k(K_X+\Phi+A)|$,
\item for every $\Phi\in(\mcal E_{V,A})_\Q$, and for
every positive integer $k$ such that $k(K_X+\Phi+A)/r$ is Cartier, we have $|k(K_X+\Phi+A)|\neq\emptyset$.
\end{enumerate}
\end{lem}
\begin{proof}
Let $K_X$ be a divisor such that $\OO_X(K_X)\simeq\omega_X$ and $\Supp A\not\subset\Supp K_X$, and let $\Lambda\subset\Div(X)$ be the monoid
spanned by components of $K_X,B$ and $A$. Let $G$ be a prime divisor on $X$. By Theorem C$_n$ the set $\mcal E_{V,A}$ is a rational polytope, and
let $D_1,\dots,D_\ell$ be generators of the finitely generated monoid $\mcal C=\R_+(K_X+A+\mcal E_{V,A})\cap\Lambda$, cf.\ Lemma \ref{lem:gordan}.
Since every $D_i$ is proportional to an adjoint bundle, by Theorem A$_n$ and by Lemma \ref{lem:1}(1) the ring $R(X;D_1,\dots,D_\ell)$ is
finitely generated, and thus so is the algebra $R(X,\mcal C)$ by projection. By Proposition \ref{pro:1} the map $\bMob_{\iota|\mcal C\cap\Lambda^{(r)}}$
is rationally piecewise additive for some positive integer $r$, where
$\iota\colon\Lambda\rightarrow\Lambda$ is the identity map. Now (2) is straightforward.

Furthermore, the set
$\mcal O=\{\Upsilon\in\mcal C_\R:\ord_G\|\Upsilon\|=0\}$ is a rational polyhedral cone by the proof of \cite[4.1]{ELMNP}, and
$\R_+(K_X+A+\mcal B_{V,A}^G)\subset\mcal O$. Since for every $\Upsilon\in\mcal O_\Q$ we have $G\not\subset\B(\Upsilon)$ by Theorem A$_n$
and by \cite[2.3.53]{Cor07}, this implies
$\mcal O\subset\R_+(K_X+A+\mcal B_{V,A}^G)$ as extremal rays of $\mcal O$ are rational. Therefore $\mcal B_{V,A}^G$ is a rational polytope,
and now (1) follows similarly as above.
\end{proof}

\section{Diophantine approximation}\label{sec:diophant}

I need a few results from Diophantine approximation theory.

\begin{lem}\label{diophant2}
Let $\Lambda\subset\R^n$ be a lattice spanned by rational vectors, and let $V=\Lambda\otimes_\Z\R$. Fix $v\in V$ and denote $X=\N v+\Lambda$.
Then the closure of $X$ is symmetric with respect to the origin. Moreover,
if $\pi\colon V\rightarrow V/\Lambda$ is the quotient map, then the closure of $\pi(X)$ is a finite disjoint union of connected components.
If $v$ is not contained in any proper rational affine subspace of $V$, then $X$ is dense in $V$.
\end{lem}
\begin{proof}
I am closely following the proof of \cite[3.7.6]{BCHM}. Let $G$ be the closure of $\pi(X)$.
Since $G$ is infinite and $V/\Lambda$ is compact, $G$ has an accumulation point. It then follows that zero is also an accumulation point
and that $G$ is a closed subgroup.
The connected component $G_0$ of the identity in $G$ is a Lie subgroup of $V/\Lambda$ and so by \cite[Theorem 15.2]{Bum04},
$G_0$ is a torus. Thus $G_0=V_0/\Lambda_0$, where $V_0=\Lambda_0\otimes_\Z\R$ is a rational subspace of $V$.
Since $G/G_0$ is discrete and compact, it is finite, and it is straightforward that $X$ is symmetric with respect to the origin.
Therefore a translate of $v$ by a rational vector is contained in $V_0$,
and so if $v$ is not contained in any proper rational affine subspace of $V$, then $V_0=V$.
\end{proof}

\begin{dfn}
Let $x\in\R^n$, $\varepsilon\in\R_{>0}$ and $k\in\Z_{>0}$. We say that
$(x_i,k,k_i,r_i)\in\Q^n\times\Z_{>0}^2\times\R_{>0}$ {\em uniformly approximate $x$ with error $\varepsilon$\/}, for $i=1,\dots,p$, if
\begin{enumerate}
\item $k_ix_i/k$ is integral for every $i$,
\item $\|x-x_i\|<\varepsilon/k_i$ for every $i$,
\item $x=\sum r_ix_i$ and $\sum r_i=1$.
\end{enumerate}
\end{dfn}

The next result is \cite[3.7.7]{BCHM}.

\begin{lem}\label{lem:surround}
Let $x\in\R^n$ and let $W$ be the smallest rational affine space containing $x$. Fix a positive integer $k$ and a positive real number $\varepsilon$.
Then there are finitely many $(x_i,k,k_i,r_i)\in(W\cap\Q^n)\times\Z_{>0}^2\times\R_{>0}$ which uniformly approximate $x$ with error $\varepsilon$.
\end{lem}

I will need a refinement of this lemma when the approximation is not necessarily happening in the smallest rational affine space containing a point.

\begin{lem}\label{lem:approximation}
Let $x\in\R^n$ and let $W$ be the smallest rational affine space containing $x$. Let $0<\varepsilon,\eta\ll1$ be rational numbers, $k$
a positive integer, and assume that there are $x_1\in\Q^n$ and $k_1\in\Z_{>0}$ such that $\|x-x_1\|<\varepsilon/k_1$ and
$k_1x_1/k$ is integral. Then there are finitely many $x_i\in\Q^n$ and $k_i\in\Z_{>0}$ for $i\geq 2$, and positive real numbers $r_i$ for $i\geq1$,
such that $(x_i,k,k_i,r_i)$ uniformly approximate $x$ with error $\varepsilon$.
Furthermore, we can assume that $x_i\in W$ for $i\geq3$, and we can write
$$x=\frac{k_1}{k_1+k_2}x_1+\frac{k_2}{k_1+k_2}x_2+\xi,$$
with $\|\xi\|<\eta/(k_1+k_2)$.
\end{lem}
\begin{proof}
Rescaling by $k$, we can assume that $k=1$. Let $\pi\colon\R^n\rightarrow\R^n/\Z^n$ be the quotient map
and let $G$ be the closure of the set $\pi(\N x+\Z^n)$. Then by Lemma \ref{diophant2} we have $\pi(-k_1x)\in G$ and
there is $k_2\in\N$ such that $\pi(k_2x)$ is in the connected component of $\pi(-k_1x)$ in $G$ and
$\|k_2x-y\|<\eta$ for some $y\in\R^n$ with $\pi(y)=\pi(-k_1x)$. Thus there is a point $x_2\in\Q^n$ such that
$k_2x_2\in\Z^n$, $\|k_2x-k_2x_2\|<\varepsilon$ and the open segment $(x_1,x_2)$ intersects $W$ at a point $u$.

By Lemma \ref{lem:surround}, there exist $(x_i,1,k_i,p_i)\in(W\cap\Q^n)\times\Z_{>0}\times\R_{>0}$ for $i\geq3$ which uniformly approximate
$x$ with error $\varepsilon$. In particular, $x$ is in the interior of the rational polytope with vertices $x_i$ for $i\geq3$, so there exists a point
$v=\sum_{i\geq3}q_ix_i$ with $q_i>0$ and $\sum q_i=1$, such that $x\in(u,v)$. Let $\alpha,\beta\in(0,1)$ be such that $u=\alpha x_1+(1-\alpha)x_2$
and $x=\beta u+(1-\beta)v$, and set $r_1=\alpha\beta$, $r_2=(1-\alpha)\beta$, and $r_i=(1-\beta)q_i$ for $i\geq3$. Then
$(x_i,k,k_i,r_i)$ uniformly approximate $x$ with error $\varepsilon$.

Finally, observe that the vector $y/k_2-x_2$ is parallel to the vector $x-x_1$ and $\|y-k_2x_2\|=\|k_1x-k_1x_1\|$. Denote $z=x-y/k_2$. Then
$$\frac{x-x_1}{(x_2+z)-x}=\frac{x-x_1}{x_2-y/k_2}=\frac{k_2}{k_1},$$
so
$$x=\frac{k_1}{k_1+k_2}x_1+\frac{k_2}{k_1+k_2}(x_2+z)=\frac{k_1}{k_1+k_2}x_1+\frac{k_2}{k_1+k_2}x_2+\xi,$$
where $\|\xi\|=\|k_2z/(k_1+k_2)\|<\eta/(k_1+k_2)$.
\end{proof}

\begin{rem}\label{rem:2}
Assuming notation from the previous proof, the connected components of $G$ are precisely the connected components of the closure of the set
$\pi(\bigcup_{k>0}kW)$. Therefore $y/k_2\in W$.
\end{rem}

\begin{rem}\label{rem:6}
Suppose points $(y_i,k,k_i,r_i)$ uniformly approximate $x\in\R^k$ with error $\varepsilon$, in the sup-norm.
Let $x_j$ denote the $j$-th coordinate of $x$, and similarly for other vectors. I claim that by choosing $\varepsilon\ll1$ and $k_i\gg0$
we have $y_{ip}\geq y_{iq}$ whenever $x_p\geq x_q$. To that end, if $x_p=x_q$, then by triangle inequality
$|k_i(y_{ip}-y_{iq})|\leq|k_i(x_p-y_{ip})|+|k_i(x_q-y_{iq})|<2\varepsilon$, so $y_{ip}=y_{iq}$ since $k_iy_{ip}$ and $k_iy_{iq}$ are
integers. If $x_p>x_q$, then since $|x_p-y_{ip}|+|x_q-y_{iq}|<2\varepsilon/k_i<x_p-x_q$, we must have $y_{ip}>y_{iq}$, and the claim follows.
\end{rem}

\section{Restricting plt algebras}\label{plt}

In this section I establish one of the technically most difficult steps in the scheme of the proof, that Theorems A$_{n-1}$, B$_n$ and
C$_{n-1}$ imply Theorem A$_n$. Crucial techniques will be those developed in \cite{HM08} and in Sections \ref{sec:convex} and \ref{sec:3}.

The key result is the following Hacon-M\textsuperscript{c}Kernan extension theorem \cite[6.3]{HM08}, whose proof relies
on deep techniques initiated by \cite{Siu98}.

\begin{thm}\label{thm:hmck}
Let $(X,\Delta=S+A+B)$ be a projective plt pair such that $S=\lfloor\Delta\rfloor$ is irreducible, $\Delta\in\WDiv(X)_\Q$,
$(X,S)$ is log smooth, $A$ is a general ample $\Q$-divisor and $(S,\Omega+A_{|S})$ is canonical,
where $\Omega=(\Delta-S)_{|S}$. Assume $S\not\subset\B(K_X+\Delta)$, and let
$$F=\liminf_{m\rightarrow\infty}\textstyle\frac1m\Fix|m(K_X+\Delta)|_S.$$
If $\varepsilon>0$ is any rational number such that $\varepsilon(K_X+\Delta)+A$ is ample and if $\Phi$ is any $\Q$-divisor on $S$
and $k>0$ is any integer such that both $k\Delta$ and $k\Phi$ are Cartier, and $\Omega\wedge(1-\frac{\varepsilon}{k})F\leq\Phi\leq\Omega$, then
$$|k(K_S+\Omega-\Phi)|+k\Phi\subset|k(K_X+\Delta)|_S.$$
\end{thm}

The immediate consequence is:

\begin{cor}\label{cor:hmck}
Let $(X,\Delta=S+A+B)$ be a projective plt pair such that $S=\lfloor\Delta\rfloor$ is irreducible, $\Delta\in\WDiv(X)_\Q$,
$(X,S)$ is log smooth, $A$ is a general ample $\Q$-divisor and $(S,\Omega+A_{|S})$ is canonical,
where $\Omega=(\Delta-S)_{|S}$. Assume $S\not\subset\B(K_X+\Delta)$, and let $\Phi_m=\Omega\wedge\frac1m\Fix|m(K_X+\Delta)|_S$ for every $m$
such that $m\Delta$ is Cartier. Then
$$|m(K_S+\Omega-\Phi_m)|+m\Phi_m=|m(K_X+\Delta)|_S.$$
\end{cor}

The following result will be used several times to test inclusions of linear series. It is extracted and copied almost verbatim from \cite{Hac08},
and Step 2 of the proof below first appeared in \cite{Tak06}. Similar techniques in the analytic setting appeared in \cite{Pau08}.

\begin{pro}\label{pro:2}
Let $(X,\Delta=S+A+B)$ be a projective plt pair such that $S=\lfloor\Delta\rfloor$ is irreducible, $\Delta\in\WDiv(X)_\Q$,
$(X,S)$ is log smooth, $A$ is a general ample $\Q$-divisor and $(S,\Omega+A_{|S})$ is canonical,
where $\Omega=(\Delta-S)_{|S}$. Let $0\leq\Theta\leq\Omega$ be a $\Q$-divisor on $S$, let $k$ be a positive integer
such that $k\Delta$ and $k\Theta$ are integral, and denote $A'=A/k$. Assume that $S\not\subset\B(K_X+\Delta+A')$
and that for any $l>0$ sufficiently divisible we have
\begin{equation}\label{eq:33}
\Omega\wedge\textstyle\frac1l\Fix|l(K_X+\Delta+A')|_S\leq\Omega-\Theta.
\end{equation}
Then
$$|k(K_S+\Theta)|+k(\Omega-\Theta)\subset|k(K_X+\Delta)|_S.$$
\end{pro}
\begin{proof}
{\em Step 1.\/}
We first prove that there exists an effective divisor $H$ on $X$ not containing $S$ such that for all sufficiently divisible
positive integers $m$ we have
\begin{equation}\label{eq:6}
|m(K_S+\Theta)|+m(\Omega-\Theta)+(mA'+H)_{|S}\subset|m(K_X+\Delta)+mA'+H|_S.
\end{equation}
Taking $l$ as in \eqref{eq:33} sufficiently divisible, we can assume $S\not\subset\Bs|l(K_X+\Delta+A')|$.
Let $f\colon Y\rightarrow X$ be a log resolution of
$(X,\Delta+A')$ and of $|l(K_X+\Delta+A')|$. Denote $\Gamma=\B(X,\Delta+A')_Y$ and $E=K_Y+\Gamma-f^*(K_X+\Delta+A')$, and define
$$\Xi=\Gamma-\Gamma\wedge\textstyle\frac1l\Fix|l(K_Y+\Gamma)|.$$
Then $l(K_Y+\Xi)$ is Cartier, $\Fix|l(K_Y+\Xi)|\wedge\Xi=0$ and $\Mob(l(K_Y+\Xi))$ is free. Since $\Fix|l(K_Y+\Xi)|+\Xi$ has simple
normal crossings support, it follows that $\B(K_Y+\Xi)$ contains no log canonical centres of $(Y,\lceil\Xi\rceil)$. Denote
$T=f_*^{-1}S,\Gamma_T=(\Gamma-T)_{|T}$ and $\Xi_T=(\Xi-T)_{|T}$, let $m$ be any positive integer divisible by $l$ and consider a section
$$\sigma\in H^0(T,\OO_T(m(K_T+\Xi_T)))=H^0(T,\mcal J_{\|m(K_T+\Xi_T)\|}(m(K_T+\Xi_T))).$$
By \cite[5.3]{HM08}, there is an ample divisor $H'$ on $Y$ such that if $\tau\in H^0(T,\OO_T(H'))$, then
$$\sigma\cdot\tau\in\im\big(H^0(Y,\OO_Y(m(K_Y+\Xi)+H'))\rightarrow H^0(T,\OO_T(m(K_Y+\Xi)+H'))\big).$$
Therefore
\begin{equation}\label{eq:inclusion}
|m(K_T+\Xi_T)|+m(\Gamma_T-\Xi_T)+H'_{|T}\subset|m(K_Y+\Gamma)+H'|_T.
\end{equation}
We claim that
\begin{equation}\label{eq:inequality}
\Omega+A'_{|S}\geq(f_{|T})_*\Xi_T\geq\Theta+A'_{|S}.
\end{equation}
Assuming the claim, as $(S,\Omega+A'_{|S})$ is canonical, we have
$$|m(K_S+\Theta)|+m((f_{|T})_*\Xi_T-\Theta)\subset|m(K_S+(f_{|T})_*\Xi_T)|=(f_{|T})_*|m(K_T+\Xi_T)|.$$
Pushing forward the inclusion \eqref{eq:inclusion}, we obtain \eqref{eq:6} for $H=f_*H'$.

Now we prove the claim. Since $\Xi_T\leq\Gamma_T$ and $(f_{|T})_*\Gamma_T=\Omega+A'_{|S}$, the first inequality in \eqref{eq:inequality} follows.
In order to prove the second inequality, let $P$ be any prime divisor on $S$ and let $P'=(f_{|T})^{-1}_*P$. Assume that
$P\subset\Supp\Omega$, and thus $P'\subset\Supp\Gamma_T$. Then there is a component $Q$ of the support of $\Gamma$ such that
$$\mult_{P'}\Fix|l(K_Y+\Gamma)|_T=\mult_Q\Fix|l(K_Y+\Gamma)|\quad\textrm{and}\quad\mult_{P'}\Gamma_T=\mult_Q\Gamma,$$
and thus
$$\mult_{P'}\Xi_T=\mult_{P'}\Gamma_T-\min\{\mult_{P'}\Gamma_T,\mult_{P'}\textstyle\frac1l\Fix|l(K_Y+\Gamma)|_T\}.$$
Notice that $\mult_{P'}\Gamma_T=\mult_P(\Omega+A'_{|S})$ and since $E_{|T}$ is exceptional, we have
$$\mult_{P'}\Fix|l(K_Y+\Gamma)|_T=\mult_P\Fix|l(K_X+\Delta+A')|_S.$$
Therefore $(f_{|T})_*\Xi_T=\Omega+A'_{|S}-\Omega\wedge\frac1l\Fix|l(K_X+\Delta+A')|_S$. The inequality now follows from \eqref{eq:33}.\\[2mm]
{\em Step 2.\/}
Let $m\gg0$ be as in Step 1 and divisible by $k$, and such that $A'-\frac{k-1}{m}H$ is ample and $(S,\Omega+\frac{k-1}{m}H_{|S})$ is klt, which is possible since $(S,\Omega)$
is canonical. In particular, $\mcal J_{\Omega+\frac{k-1}{m}H_{|S}}=\OO_S$.

By Step 1, for any $\Sigma\in|k(K_S+\Theta)|$ there is a divisor $G\in|m(K_X+\Delta)+mA'+H|$
such that $G_{|S}=\frac{m}{k}\Sigma+m(\Omega-\Theta)+(mA'+H)_{|S}$. Set $\Lambda=\frac{k-1}{m}G+\Delta-S-A$, and observe that
$$\Lambda_{|S}-(\Sigma+k(\Omega-\Theta))=\textstyle\frac{k-1}{m}G_{|S}+\Omega-A_{|S}-(\Sigma+k(\Omega-\Theta))\leq\Omega+\frac{k-1}{m}H_{|S}.$$
Therefore,
\begin{equation}\label{equ:3}
\mcal I_{\Sigma+k(\Omega-\Theta)}\subset\mcal J_{\Lambda_{|S}}
\end{equation}
by \cite[4.3(3)]{HM08}. Since $k(K_X+\Delta)\sim_\Q K_X+S+\Lambda+(A'-\textstyle\frac{k-1}{m}H)$, the homomorphism
$$H^0(X,\mcal J_{S,\Lambda}(k(K_X+\Delta)))\rightarrow H^0(S,\mcal J_{\Lambda_{|S}}(k(K_X+\Delta)))$$
is surjective by \cite[4.4(3)]{HM08}. This together with \eqref{equ:3} implies
$$\Sigma+k(\Omega-\Theta)\in|k(K_X+\Delta)|_S,$$
which finishes the proof.
\end{proof}

The main result of this section is the following.

\begin{thm}\label{thm:2}
Let $X$ be a smooth variety of dimension $n$, $S$ a smooth prime divisor and $A$ a general ample $\Q$-divisor on $X$.
For $i=1,\dots,\ell$, let $D_i=k_i(K_X+\Delta_i)\in\Div(X)$, where $(X,\Delta_i=S+B_i+A)$ is a log smooth plt pair with $\lfloor \Delta_i\rfloor=S$
and $|D_i|\neq\emptyset$. Assume Theorems A$_{n-1}$, B$_n$ and C$_{n-1}$. Then the algebra $\res_S R(X;D_1,\dots,D_\ell)$ is finitely generated.
\end{thm}
\begin{proof}
{\em Step 1.\/}
I first show that we can assume $S\notin\Fix|D_i|$ for all $i$.

To that end, let $K_X$ be a divisor with $\OO_X(K_X)\simeq\omega_X$ and
$\Supp A\not\subset\Supp K_X$, and let $\Lambda$ be the monoid in $\Div(X)$ generated
by the components of $K_X$ and $\sum\Delta_i$.
Denote $\mcal C_S=\{P\in\Lambda_\R:S\notin\B(P)\}$. By Theorem B$_n$, the set
$\mcal A=\sum_{i=1}^\ell\R_+D_i\cap\mcal C_S$ is a rational polyhedral cone.

The monoid $\sum_{i=1}^\ell\R_+D_i\cap\Lambda$ is finitely generated by Lemma \ref{lem:gordan}, and let $D_{\ell+1},\dots,D_q$ be its generators.
Let $e_i$ be the standard generators of $\R^q$. If $\mu\colon\bigoplus_{i=1}^q\N e_i\rightarrow\Div(X)$ denotes the additive map given by $\mu(e_i)=D_i$,
then $\mcal S=\mu^{-1}(\mcal A\cap\Lambda)\cap\bigoplus_{i=1}^\ell\N e_i$ is a finitely generated monoid, and let $h_1,\dots,h_m$ be generators of
$\mcal S$. Observe that $\mu(h_i)$ is a multiple of an adjoint bundle for every $i$, and that $R(X,\mu(\bigoplus_{i=1}^\ell\N e_i))=R(X;D_1,\dots,D_\ell)$.

The algebra $\res_S R(X,\mu(\bigoplus_{i=1}^\ell\N e_i))$ is finitely ge\-ne\-ra\-ted if and only if
$\res_S R(X,\mu(\mcal S))$ is, since $\res_S H^0(X,\mu(s))=0$ for every $s\in\big(\bigoplus_{i=1}^\ell\N e_i\big)\backslash\mcal S$. Then
it is enough to prove that the restricted algebra $\res_S R(X;\mu(h_1),\dots,\mu(h_m))$ is finitely generated, as we have the natural projection
$$\res_S R(X;\mu(h_1),\dots,\mu(h_m))\rightarrow\res_S R(X,\mu(\mcal S)).$$
By passing to a truncation,
cf.\ Lemma \ref{lem:1}(1), I can assume further that $S\notin\Fix|\mu(h_i)|$ for $i=1,\dots,m$. Now by replacing $\mcal S$ by $\bigoplus_{i=1}^m\N\mu(h_i)$,
I assume $\mcal S=\bigoplus_{i=1}^\ell\N e_i$ and $\mu(e_i)=D_i$ for every $i$.\\[2mm]
{\em Step 2.\/}
For $s=\sum_{i=1}^\ell t_ie_i\in\mcal S_\R$, denote
$$t_s=\sum_{i=1}^\ell t_ik_i,\quad\Delta_s=\sum_{i=1}^\ell\frac{t_ik_i}{t_s}\Delta_i,\quad\textrm{and}\quad\Omega_s=(\Delta_s-S)_{|S},$$
and observe that then
$$R(X;D_1,\dots,D_\ell)=\bigoplus_{s\in\mcal S}H^0(X,t_s(K_X+\Delta_s)).$$
In this step I show that we can assume that the pair $(S,\Omega_s+A_{|S})$ is terminal for every $s\in\mcal S_\R$.

Let $\sum F_k=\bigcup_i\Supp B_i$, and denote $\B_i=\B(X,\Delta_i)$ and $\B=\B(X,S+\nu\sum_kF_k+A)$,
where $\nu=\max_{i,k}\{\mult_{F_k}B_i\}$. By Lemma \ref{disjoint} there is a log resolution $f\colon Y\rightarrow X$
such that the components of $\{\B_Y\}$ do not intersect, and denote $D_i'=k_i(K_Y+\B_{iY})$. Observe that
\begin{equation}\label{eq:4}
R(X;D_1,\dots,D_\ell)\simeq R(Y;D_1',\dots,D_\ell').
\end{equation}
Since $B_i\leq\nu\sum_kF_k$, by comparing discrepancies we see that the
components of $\{\B_{iY}\}$ do not intersect for every $i$, and notice that
$f^*A=f_*^{-1}A\leq\B_{iY}$ since $A$ is general. Denote $\Delta_s'=\sum_{i=1}^\ell\frac{t_ik_i}{t_s}\B_{iY}$.
Let $H$ be a small effective $f$-exceptional $\Q$-divisor such that $f^*A-H$ is ample and let
$A'\sim_\Q f^*A-H$ be a general ample $\Q$-divisor. Let $T=f_*^{-1}S$, $\Psi_s=\Delta_s'-f^*A-T+H\geq0$ and $\Omega_s'=(\Psi_s+A')_{|T}$. Then
the pair $(T,\Omega_s'+A_{|T}')$ is terminal and
$$K_Y+T+\Psi_s+A'\sim_\Q K_Y+\Delta_s'.$$
Now replace $X$ by $Y$, $S$ by $T$, $\Delta_s$ by $T+\Psi_s+A'$ and $\Omega_s$ by $\Omega_s'$.\\[2mm]
{\em Step 3.\/}
Write
$$\res_S R(X;D_1,\dots,D_\ell)=\bigoplus_{s\in\mcal S}\mcal R_s.$$
Then, denoting $\theta_s=\Omega_s-\Omega_s\wedge\frac{1}{t_s}\Fix|t_s(K_X+\Delta_s)|_S$, we have
\begin{equation}\label{tag:9}
\mcal R_s=H^0(S,t_s(K_S+\theta_s))
\end{equation}
by Corollary \ref{cor:hmck}. Let $\m\colon\mcal S\rightarrow\Div(S)$ be the map given by $\m(s)=\bMob(t_s(K_S+\theta_s))$. Since
$\mcal R_{s_1}\mcal R_{s_2}\subset\mcal R_{s_1+s_2}$ for all $s_1,s_2\in\mcal S$, $\m$ is superadditive, cf.\ \cite[2.3.34]{Cor07}.

For $s\in\mcal S$, set $\Theta_s=\limsup\limits_{m\rightarrow\infty}\theta_{ms}$. Then similarly as in the proof of \cite[7.1]{HM08},
by Theorem \ref{thm:hmck} and Lemma \ref{bounded} we obtain that $\Theta_s$ is rational and
\begin{equation}\label{eq:restriction}
\bigoplus_{p\in\N}\mcal R_{p\ell_ss}\simeq R(S,\ell_st_s(K_S+\Theta_s)),
\end{equation}
where $\ell_s$ is a positive integer such that $\ell_s\Delta_s$ and $\ell_s\Theta_s$ are Cartier. By Theorem A$_{n-1}$, the algebra
$R(S,\ell_st_s(K_S+\Theta_s))$ is finitely generated, and since $\mcal R_{p\ell_ss}=H^0(S,\m(p\ell_ss))$, the map $\m|_{\N s}$ is
additive up to truncation by \cite[2.3.53]{Cor07}. Therefore, there is a well-defined straightening $\m^\sharp\colon\mcal S_\R\rightarrow\Div(S)_\R$ by
Lemma \ref{concave}.

Define the maps $\Theta,\lambda\colon\mcal S\rightarrow\Div(S)_\Q$ by
$$\Theta(s)=\Theta_s,\qquad\lambda(s)=t_s(K_S+\Theta_s)$$
for $s\in\mcal S$. Note that, by definition of $\theta_s$ and by \eqref{tag:9}, for every component $G$ of $\theta_s$ we have
$G\notin\Fix|t_s(K_S+\theta_s)|$, and so $\mult_G(t_s(K_S+\theta_s))=\mult_G\m(s)$. Therefore, by the construction of $\m^\sharp$ from the proof of
Lemma \ref{concave}, $\mult_G(t_s(K_S+\Theta_s))=\mult_G\m^\sharp(s)$ for every component $G$ of $\Theta_s$, and
thus $\Theta$ and $\lambda$ extend to $\mcal S_\R$.

I claim that there exists a finite rational polyhedral subdivision $\mcal S_\R=\bigcup\mcal C_i$ such that
$\lambda$ is linear on each $\mcal C_i$. Grant this for the moment.
By Lemma \ref{lem:gordan}, let $s_1^i,\dots,s_z^i$ be generators of $\mcal S_i=\mcal S\cap\mcal C_i$, and let $\kappa$ be a sufficiently divisible
positive integer such that $\lambda(\kappa s_j^i)=\kappa\lambda(s_j^i)\in\Div(S)$ for all $i$ and $j$.
Then the ring $R(S;\lambda(\kappa s_1^i),\dots,\lambda(\kappa s_z^i))$ is finitely generated by Theorem A$_{n-1}$,
and so is $\bigoplus_{s\in\mcal S_i^{(\kappa)}}\mcal R_s$ by projection.
Thus the algebra $\bigoplus_{s\in\mcal S_i}\mcal R_s$ is finitely generated by Lemma \ref{lem:1}(1), and so is
$\res_S R(X;D_1,\dots,D_\ell)$ by putting all those generators together.

The claim stated above is Theorem \ref{lem:PL}, and this is proved in the remainder of this section.
\end{proof}

\begin{rem}\label{rem:12}
Note that for every $s\in\mcal S$ we have $\Theta_s-A_{|S}\in\mcal B_{V_S,A_{|S}}^G$ for every component $G$ of $\Theta_s$,
since $\theta_{ms}-A_{|S}\in\mcal B_{V_S,A_{|S}}^G$ for every component $G$ of $\theta_s$, and each $\mcal B_{V_S,A_{|S}}^G$ is a rational polytope
by Lemma \ref{bounded}, and in particular closed.
\end{rem}

\begin{nt}\label{nt}
With notation from the previous proof, for $s\in\mcal S_\R$ I usually denote $\Theta(s)$ and $\lambda(s)$ by $\Theta_s$ and $\lambda_s$, respectively.
Denote $\Pi=\{s\in\mcal S_\R:t_s=1\}$; this is a rational polytope in $\R^\ell$.
Let $\Delta\colon\R\Pi\rightarrow\Div(X)_\R$ be the linear map given by $\Delta(q_i)=\Delta_{q_i}$ for linearly independent
points $q_1,\dots,q_\ell\in\Pi$, and then extended linearly. This is well defined since $\Delta$ is an affine map on $\Pi$.
Similarly, observe that, since the function $\ord_P\|\cdot\|_S$ is convex for every $P$, the set $\{s\in\Pi:\Theta_s>0\}$ is convex
and $\Theta$ is concave on it. Let $L$ denote the norm of the linear map $\Delta$, i.e.\ the smallest global Lipschitz constant of $\Delta$.
Denote by $V_S\subset\Div(S)_\R$ the vector space spanned by the components of $\bigcup_{s\in\mcal S_\R}\Supp(\Omega_s-A_{|S})$.
For a prime divisor $P$ on $S$, let $\lambda_P\colon\mcal S_\R\rightarrow\R$ be the function given by
$\lambda_P(s)=\mult_P\lambda(s)$, and similarly for $\Theta_P$.
\end{nt}

\paragraph{\bf Proof of piecewise linearity}
To finish the proof of Theorem \ref{thm:2}, it remains to prove that the map $\lambda$ is rationally piecewise linear. I first briefly sketch
the strategy of the proof of this fact, which occupies the rest of this section.

Until the end of the section I fix a prime divisor $Z$ on $S$, and the goal is to prove that
$\lambda_Z$ is rationally piecewise linear -- it is clear that then $\lambda$ is rationally piecewise linear by taking a subdivision of
the cone $\mcal S_\R$ that works for all prime divisors.
By suitably replacing $\mcal S_\R$ and $\lambda_Z$, I can assume that $\lambda_Z$ is a superlinear map, see
the proof of Theorem \ref{lem:PL}, and also that $\Theta_s-A_{|S}\in\mcal B^Z_{V_S,A_{|S}}$ for every $s\in\mcal S_\R$.
In order to prove that $\lambda_Z$ is piecewise linear, it is enough to show that $\lambda_Z|_{\mcal S_\R\cap H}$ is piecewise linear
for every $2$-plane $H\subset\R^\ell$ by Theorem \ref{piecewise}, and the first step is Theorem
\ref{lem:lipschitz}(1), which claims that $\lambda|_{\mcal S_\R\cap H}$ is continuous.

The method of the proof is as follows: starting from a point
$s\in\mcal S_\R$ and a $2$-plane $H\ni s$, I approximate $(s,\Theta_s)\in\mcal S_\R\times\Div(S)_\R$ by points
$(t_i,\Theta_{t_i}')\in\mcal S_\Q\times\Div(S)_\Q$ such that $\R_+s\subsetneq\mcal C_{s,H}\cap H$, where $\mcal C_{s,H}=\sum\R_+t_i$. Furthermore, if
the approximation is sufficiently good, I can assume that $\Theta_{t_i}'\in\mcal B^Z_{V_S,A_{|S}}$ by Theorems A$_{n-1}$ and C$_{n-1}$. Then there are
suitable inclusions of linear series which force $\lambda_Z$ to be convex on $\mcal C_{s,H}$. However, since $\lambda_Z$ is concave,
this implies it is linear on $\mcal C_{s,H}$, and thus on $\mcal S_\R\cap H$ by an easy compactness argument.
The fact that $\lambda_Z$ is {\em rationally\/} piecewise linear then follows easily, and this is done in Step 3 of the proof of Theorem \ref{lem:PL}.

\begin{thm}\label{lem:lipschitz}
Fix $s\in\Pi$, let $U\subset\R^\ell$ be the smallest rational affine space containing $s$ and let $P$ be a prime divisor on $S$. Then:
\begin{enumerate}
\item for any $t\in\Pi$ we have $\lim\limits_{\varepsilon\downarrow0}\Theta_{(1-\varepsilon)s+\varepsilon t}=\Theta_s$,
\item if $\Theta_P(s)>0$, then the map $\lambda_P$ is linear in a neighbourhood of $s$ contained in $U$.
\end{enumerate}
\end{thm}
\begin{proof}
First note that $U\cap\mcal S_\R\subset\Pi$, and let $r$ be a positive integer as in Lemma \ref{bounded}, with respect to the vector space $V_S$ and the
ample divisor $A_{|S}$.

Note that in order to prove the claim (1), it is enough to show that for every $u\in\mcal S$,
\begin{equation}\label{tag:5}
\Theta_u=\Theta^\sigma_u,
\end{equation}
where $\Theta^\sigma_u=\Omega_u-\Omega_u\wedge N_\sigma\|K_X+\Delta_u\|_S$, cf.\ Remark \ref{rem:11}, since then
$$\lim_{\varepsilon\downarrow0}\Theta_{(1-\varepsilon)s+\varepsilon t}=\lim_{\varepsilon\downarrow0}\Theta_{(1-\varepsilon)s+\varepsilon t}^\sigma=\Theta_s^\sigma=\Theta_s$$
by Lemma \ref{lem:restrictedord}(3). Therefore I concentrate on proving \eqref{tag:5} and the claim (2). Without loss of generality I assume $u=s$.
In Step 1 I am closely following \cite{Hac08}.\\[2mm]
{\em Step 1.\/}
Let $0<\phi<1$ be the smallest positive coefficient of $\Omega_s-\Theta_s^\sigma$ if it exists, and set $\phi=1$ otherwise.
Let $W\subset\Div(S)_\R$ be the smallest rational affine space containing $\Theta_s^\sigma$.
Let $0<\eta\ll1$ be a rational number such that $(L+1)\eta(K_X+\Delta')+\frac12A$ and $\Delta'-\Delta_s+\frac12A$ are ample divisors
whenever $\Delta'\in B(\Delta_s,L\eta)$, cf.\ Notation \ref{nt}.

By Lemma \ref{lem:surround}, there exist rational points $(t_i,\Theta_{t_i}')\in U\times W$, integers $p_{t_i}\gg0$ and $r_{t_i}\in\R_{>0}$ such that
$(t_i,\Theta_{t_i}',r,p_{t_i},r_{t_i})$ uniformly approximate $(s,\Theta_s^\sigma)\in U\times W$ with error $\phi\eta$.
Note that then $(\Omega_{t_i},\Theta_{t_i}',r,p_{t_i},r_{t_i})$ uniformly approximate $(\Omega_s,\Theta_s^\sigma)\in\Div(S)_\R\times W$ with error
$\max\{\phi\eta,L\phi\eta\}$, and thus $\Theta_{t_i}'\leq\Omega_{t_i}$ by Remark \ref{rem:6}. Furthermore, for every prime divisor
$P$ on $S$ we have
\begin{equation}\label{tag:3}
\textstyle\big(1-\frac{(L+1)\eta}{p_{t_i}}\big)\mult_P(\Omega_s-\Theta_s^\sigma)\leq\mult_P(\Omega_{t_i}-\Theta_{t_i}').
\end{equation}
To see this, note that \eqref{tag:3} is trivial when $\mult_P(\Omega_s-\Theta_s^\sigma)=0$. Therefore I can assume that $0<\phi<1$ and
$\mult_P(\Omega_s-\Theta_s^\sigma)\geq\phi$. Since $\|\Omega_s-\Omega_{t_i}\|<L\phi\eta/p_{t_i}$ and $\|\Theta_{t_i}'-\Theta_s^\sigma\|<\phi\eta/p_{t_i}$,
by triangle inequality we have
\begin{multline*}
\mult_P(\Omega_s-\Theta_s^\sigma)\leq\mult_P(\Omega_{t_i}-\Theta_{t_i}')+\|\Omega_s-\Omega_{t_i}\|+\|\Theta_{t_i}'-\Theta_s^\sigma\|\\
\leq\mult_P(\Omega_{t_i}-\Theta_{t_i}')+\textstyle\frac{(L+1)\phi\eta}{p_{t_i}}\leq\mult_P(\Omega_{t_i}-\Theta_{t_i}')+
\frac{(L+1)\eta}{p_{t_i}}\mult_P(\Omega_s-\Theta_s^\sigma),
\end{multline*}
and \eqref{tag:3} follows.

I claim that
\begin{equation}\label{eq:11}
|p_{t_i}(K_S+\Theta_{t_i}')|+p_{t_i}(\Omega_{t_i}-\Theta_{t_i}')\subset|p_{t_i}(K_X+\Delta_{t_i})|_S
\end{equation}
for every $i$. To that end, set $A_{t_i}=A/p_{t_i}$, and recall that $S\not\subset\B(K_X+\Delta_u)$ for every $u\in\mcal S_\R$ by Step 1 of the proof of
Theorem \ref{thm:2}. Therefore $S\not\subset\B(K_X+\Delta_{t_i})$ for every $i$ since $t_i\in V$, $p_{t_i}\gg0$ and $\mcal S_\R$ is a rational
polyhedral cone, and so $S\not\subset\B(K_X+\Delta_{t_i}+A_{t_i})$.

Thus, to prove \eqref{eq:11}, by Proposition \ref{pro:2} it is enough to show that for any component $P\subset\Supp\Omega_s$,
and for any $l>0$ sufficiently divisible,
$$\mult_P(\Omega_{t_i}\wedge\textstyle\frac1l\Fix|l(K_X+\Delta_{t_i}+A_{t_i})|_S)\leq\mult_P(\Omega_{t_i}-\Theta_{t_i}'),$$
and so by \eqref{tag:3} it suffices to prove that
\begin{equation}\label{equ:7}
\mult_P(\Omega_{t_i}\wedge\textstyle\frac1l\Fix|l(K_X+\Delta_{t_i}+A_{t_i})|_S)\leq\big(1-\frac{(L+1)\eta}{p_{t_i}}\big)\mult_P(\Omega_s-\Theta_s^\sigma).
\end{equation}
Let $\delta>(L+1)\eta/p_{t_i}$ be a rational number such that $\delta(K_X+\Delta_{t_i})+\frac12A_{t_i}$ is ample. Since
$$\textstyle K_X+\Delta_{t_i}+A_{t_i}=(1-\delta)(K_X+\Delta_{t_i}+\frac12A_{t_i})+\big(\delta(K_X+\Delta_{t_i})+\frac{1+\delta}{2}A_{t_i}\big),$$
and $\ord_P\|K_X+\Delta_{t_i}+\frac12A_{t_i}\|_S=\sigma_P\|K_X+\Delta_{t_i}+\frac12A_{t_i}\|_S$ by Remark \ref{rem:11}, we have
$$\textstyle\ord_P\|K_X+\Delta_{t_i}+A_{t_i}\|_S\leq(1-\delta)\sigma_P\|K_X+\Delta_{t_i}+\frac12A_{t_i}\|_S,$$
and thus
\begin{equation}\label{tag:4}
\mult_P\textstyle\frac1l\Fix|l(K_X+\Delta_{t_i}+A_{t_i})|_S\leq\big(1-\frac{(L+1)\eta}{p_{t_i}}\big)\sigma_P\|K_X+\Delta_{t_i}+\frac12A_{t_i}\|_S
\end{equation}
for $l$ sufficiently divisible, cf.\ Lemma \ref{lem:restrictedord}(4). The divisor $\Delta_{t_i}-\Delta_s+\frac12A_{t_i}$ is ample by the choice of $\eta$, so
$$\textstyle\sigma_P\|K_X+\Delta_{t_i}+\frac12A_{t_i}\|_S=\sigma_P\|K_X+\Delta_s+(\Delta_{t_i}-\Delta_s+\frac12A_{t_i})\|_S\leq\sigma_P\|K_X+\Delta_s\|_S.$$
This together with \eqref{tag:4} gives \eqref{equ:7}.\\[2mm]
{\em Step 2.\/}
Let $H$ be a general ample $\Q$-divisor, and let $A_m$ be ample divisors with $\Supp A_m\subset\Supp(\Delta_s-S-A+H)$ such that $\Delta_s+A_m$ are
$\Q$-divisors and $\lim\limits_{m\rightarrow\infty}\|A_m\|=0$.
Denote $\Delta_m=\Delta_s+A_m$, $\Omega_m=(\Delta_m-S)_{|S}$ and
$$\Theta_m^\sigma=\Omega_m-\Omega_m\wedge N_\sigma\|K_X+\Delta_m\|_S.$$
Observe that $\Theta_s^\sigma=\lim\limits_{m\rightarrow\infty}\Theta_m^\sigma$ by Lemma \ref{lem:restrictedord}(2), and that
$$N_\sigma\|K_X+\Delta_m\|_S=\sum\ord_P\|K_X+\Delta_m\|_S\cdot P$$
for all prime divisors $P$ on $S$ and for all $m$, cf.\ Remark \ref{rem:11}. Thus by Remark \ref{rem:12},
$\Theta^\sigma_m-A_{|S}\in\mcal B^G_{V_{S,H},A_{|S}}$ for all $m$ and for every component $G$ of $\Theta^\sigma_m$, where $V_{S,H}=V_S+\R H_{|S}$.
Since $\mcal B^G_{V_{S,H},A_{|S}}$ is a rational polytope by Lemma \ref{bounded}, and in particular is closed, this yields
$\Theta^\sigma_s-A_{|S}\in\mcal B^G_{V_{S,H},A_{|S}}$ for every component $G$ of $\Theta^\sigma_s$.
Since $W$ is the smallest rational affine space containing $\Theta_s^\sigma$ and $p_{t_i}\gg0$, we have $\Theta_{t_i}'-A_{|S}\in\mcal B^G_{V_{S,H},A_{|S}}$ for every $i$.
Now since $p_{t_i}\Theta_{t_i}'/r$ is Cartier, we have $G\not\subset\Fix|p_{t_i}(K_S+\Theta_{t_i}')|$ by Lemma \ref{bounded}(1). In particular, then \eqref{eq:11} implies
$$\Omega_{t_i}-\Theta_{t_i}'\geq\Omega_{t_i}\wedge\textstyle\frac{1}{p_{t_i}}\Fix|p_{t_i}(K_X+\Delta_{t_i})|_S,$$
and since by definition $\Omega_{t_i}\wedge\textstyle\frac{1}{p_{t_i}}\Fix|p_{t_i}(K_X+\Delta_{t_i})|_S\geq\Omega_{t_i}-\Theta_{t_i}$, we obtain
\begin{equation}\label{equ:1}
\Theta_{t_i}\geq\Theta_{t_i}'.
\end{equation}
To prove \eqref{tag:5}, since $\Theta_s^\sigma\geq\Theta_s$ by Lemma \ref{lem:restrictedord}(1), it is enough to show that $\mult_P\Theta_s^\sigma\leq\mult_P\Theta_s$
for every prime divisor $P$ on $S$. If $\mult_P\Theta_s^\sigma=0$, then immediately $\mult_P\Theta_s=0$, and we are done.
If $\mult_P\Theta_s^\sigma>0$, then $\mult_P\Theta_{t_i}'>0$ for all $i$, and thus $\mult_P\Theta_{t_i}>0$ by \eqref{equ:1}. In particular,
$\mult_P\Theta_{t_i}=\mult_P\Omega_{t_i}-\ord_P\|K_X+\Delta_{t_i}\|_S$. Since $\Omega_s=\sum r_{t_i}\Omega_{t_i}$, and
$$\ord_P\|K_X+\Delta_s\|_S=\ord_P\big\|\sum r_{t_i}(K_X+\Delta_{t_i})\big\|_S\leq\sum r_{t_i}\ord_P\|K_X+\Delta_{t_i}\|_S$$
by convexity, using \eqref{equ:1} we have
\begin{align*}
\mult_P\Theta_s&\geq\mult_P\Omega_s-\ord_P\|K_X+\Delta_s\|_S\geq\sum r_{t_i}(\mult_P\Omega_{t_i}-\ord_P\|K_X+\Delta_{t_i}\|_S)\\
&=\sum r_{t_i}\mult_P\Theta_{t_i}\geq\sum r_{t_i}\mult_P\Theta_{t_i}'=\mult_P\Theta_s^\sigma\geq\mult_P\Theta_s.
\end{align*}
Therefore all inequalities are equalities, so this proves \eqref{tag:5}, and also the claim (2), since then $\lambda_P$ is linear on the cone
$\sum\R_+t_i$ by Lemma \ref{linear}.
\end{proof}

Next I prove that, under certain conditions, $\lambda_Z|_{\mcal S_\R\cap H}$ is piecewise linear for every $2$-plane $H\subset\R^\ell$.

\begin{thm}\label{thm:linear1}
Assume that the map $\lambda_Z$ is superlinear and that $\Theta_Z(w)>0$ for all $w\in\mcal S_\R$.
Fix distinct points $s,u\in\Pi$. Then there exists $t\in(s,u)$ such that the map $\lambda_Z|_{\R_+s+\R_+t}$ is linear.
In particular, for every $2$-plane $H\subset\R^\ell$, the map $\lambda_Z|_{\mcal S_\R\cap H}$ is piecewise linear.
\end{thm}
\begin{proof}
In Step 1 I prove the first claim in the case $s\in\Pi_\Q$, and in Step 2 when $s\notin\Pi_\Q$.
Then this is put together in Step 3 to prove the second claim.

Let $r$ be a positive integer as in Lemma \ref{bounded}, with respect to the vector space
$V_S$ and the ample divisor $A_{|S}$.\\[2mm]
{\em Step 1.\/}
Assume $s\in\Pi_\Q$. Let $W$ be the smallest rational affine subspace containing $s$ and $u$, and note that $W\cap\mcal S_\R\subset\Pi$.

Let $\mcal P$ be the set of all prime divisors $P$ on $S$ such that $\mult_P(\Omega_s-\Theta_s)>0$.
If $\mcal P\neq\emptyset$, by Theorem \ref{lem:lipschitz}(1) there is a positive number $\varepsilon\ll1$ such that
$$\phi=\min\{\mult_P(\Omega_v-\Theta_v):P\in\mcal P,v\in[s,u]\cap B(s,\varepsilon)\}>0,$$
and set $\phi=1$ if $\mcal P=\emptyset$. We can further assume that $\varepsilon$ is small enough to that $(L+1)\varepsilon(K_X+\Delta')+\frac12A$
and $\Delta'-\Delta_s+\frac12A$ are ample divisors whenever $\Delta'\in B(\Delta_s,2L\varepsilon)$.

Let $p_s$ be a positive integer such that $p_s\Delta_s/r$ and $p_s\Theta_s/r$ are integral, and
\begin{equation}\label{tag:12}
|p_s(K_S+\Theta_s)|+p_s(\Omega_s-\Theta_s)=|p_s(K_X+\Delta_s)|_S,
\end{equation}
cf.\ relation \eqref{eq:restriction} in Step 3 of the proof of Theorem \ref{thm:2}.
Pick $t\in(s,u]$ such that the smallest rational affine subspace containing $t$ is precisely $W$, $\|s-t\|<\phi\varepsilon/p_s$, and
$\|\Theta_s-\Theta_t\|<\phi\varepsilon/p_s$, which is possible by Theorem \ref{lem:lipschitz}(1).
Denote by $V\subset\Div(S)_\R$ the smallest rational affine space containing $\Theta_s$ and $\Theta_t$.

Pick $0<\eta\ll1$. Then by Lemma \ref{lem:approximation} there exist rational points
$(t_i,\Theta_{t_i}')\in W\times V$, integers $p_{t_i}\gg0$, and $r_{t_i}\in\R_{>0}$ for $i=1,\dots,w$ such that:
\begin{enumerate}
\item $(t_i,\Theta_{t_i}',r,p_{t_i},r_{t_i})$ uniformly approximate $(t,\Theta_t)\in W\times V$ with error $\phi\varepsilon$, where
$t_1=s$, $\Theta_{t_1}'=\Theta_s$, $p_{t_1}=p_s$,
\item $(t_i,\Theta_{t_i}')$ belong to the smallest rational affine space containing $(t,\Theta_t)$ for $i=3,\dots,w$,
\item $t=\frac{p_{t_1}}{p_{t_1}+p_{t_2}}t_1+\frac{p_{t_2}}{p_{t_1}+p_{t_2}}t_2+\tau$, where $\|\tau\|\leq\frac{\eta}{p_{t_1}+p_{t_2}}$,
\item $\Theta_t=\frac{p_{t_1}}{p_{t_1}+p_{t_2}}\Theta_{t_1}'+\frac{p_{t_2}}{p_{t_1}+p_{t_2}}\Theta_{t_2}'+\Phi$,
where $\|\Phi\|<\frac{\eta}{p_{t_1}+p_{t_2}}$.
\end{enumerate}
Note that $t_i\in\Pi$ since $W$ is the smallest rational affine space containing $t$ and $p_{t_i}\gg0$, thus all divisors above are well defined.
By applying the map $\Delta$ from Notation \ref{nt} to the condition (3), we get
\begin{enumerate}
\item[(5)] $\Delta_t=\frac{p_{t_1}}{p_{t_1}+p_{t_2}}\Delta_{t_1}+\frac{p_{t_2}}{p_{t_1}+p_{t_2}}\Delta_{t_2}+\Psi$,
where $\|\Psi\|<\frac{L\eta}{p_{t_1}+p_{t_2}}$.
\end{enumerate}
Note that then $\Theta_{t_i}'\leq\Omega_{t_i}$ for all $i$ by Remark \ref{rem:6}. Furthermore, for every prime divisor
$P$ on $S$ we have
\begin{equation}\label{tag:10}
\textstyle\big(1-\frac{(L+1)\varepsilon}{p_{t_2}}\big)\mult_P(\Omega_t-\Theta_t)\leq\mult_P(\Omega_{t_2}-\Theta_{t_2}').
\end{equation}
To prove this, note that \eqref{tag:10} is trivial when $\mult_P(\Omega_t-\Theta_t)=0$. Thus I can assume $\mult_P(\Omega_t-\Theta_t)>0$, and then,
by the choice of $\eta$,
\begin{equation}\label{tag:11}
\textstyle\mult_P\big(\Omega_t-\Theta_t-\frac{p_{t_1}+p_{t_2}}{p_{t_1}}(\Psi_{|S}-\Phi)\big)>0,
\end{equation}
since conditions (4) and (5) give $\|\frac{p_{t_1}+p_{t_2}}{p_{t_1}}(\Psi_{|S}-\Phi)\|<\frac{(L+1)\eta}{p_{t_1}}$.
If $P\notin\mcal P$, then $\mult_P(\Omega_{t_1}-\Theta_{t_1})=0$, and \eqref{tag:11} together with conditions (4) and (5) gives
$$\textstyle\mult_P(\Omega_t-\Theta_t)\leq\mult_P\big(\Omega_t-\Theta_t+\frac{p_{t_1}}{p_{t_2}}\big(\Omega_t-\Theta_t-
\frac{p_{t_1}+p_{t_2}}{p_{t_1}}(\Psi_{|S}-\Phi)\big)\big)=\mult_P(\Omega_{t_2}-\Theta_{t_2}'),$$
which implies \eqref{tag:10}. If $P\in\mcal P$, then $\mult_P(\Omega_t-\Theta_t)\geq\phi$, and \eqref{tag:10} follows similarly as \eqref{tag:3}
in Step 1 of the proof of Theorem \ref{lem:lipschitz}.

I claim that
\begin{equation}\label{equ:8}
|p_{t_i}(K_S+\Theta_{t_i}')|+p_{t_i}(\Omega_{t_i}-\Theta_{t_i}')\subset|p_{t_i}(K_X+\Delta_{t_i})|_S
\end{equation}
for all $i$. Granting this for the moment, note that $\mcal B_{V_S,A_{|S}}^Z$ is a rational polytope by Lemma \ref{bounded},
and $\Theta_p-A_{|S}\in\mcal B_{V_S,A_{|S}}^Z$ for every $p\in\Pi$ by Remark \ref{rem:12}. Therefore when $\varepsilon\ll1$, as in Step 2
of the proof of Theorem \ref{lem:lipschitz}
we have that $\lambda_Z$ is linear on the cone $\sum_{i=1}^w\R_+t_i$, and in particular on the cone $\R_+s+\R_+t$, so we are done.

To prove the claim, note that \eqref{equ:8} follows from \eqref{tag:12} for $i=1$, and for $i=3,\dots,w$ it is proved as \eqref{eq:11} in Step 1 of
the proof of Theorem \ref{lem:lipschitz}. For $i=2$, by Proposition \ref{pro:2} it is enough to show that for a prime divisor $P$ and for $l>0$
sufficiently divisible we have
$$\mult_P(\Omega_{t_2}\wedge\textstyle\frac1l\Fix|l(K_X+\Delta_{t_2}+A_{t_2})|_S)\leq\mult_P(\Omega_{t_2}-\Theta_{t_2}'),$$
where $A_{t_2}=A/p_{t_2}$, and so by \eqref{tag:10} it suffices to prove that
$$\mult_P(\Omega_{t_2}\wedge\textstyle\frac1l\Fix|l(K_X+\Delta_{t_2}+A_{t_2})|_S)\leq\big(1-\frac{(L+1)\varepsilon}{p_{t_2}}\big)\mult_P(\Omega_t-\Theta_t).$$
But this is proved similarly as \eqref{equ:7} in Step 1 of the proof of Theorem \ref{lem:lipschitz}.\\[2mm]
{\em Step 2.\/}
Assume in this step that $s\notin\Pi_\Q$. By Theorem \ref{lem:lipschitz}(2) there is a cone $\mcal{C}=\sum\R_+g_i$ with finitely many $g_i\in\mcal S_\Q$,
such that $s=\sum\alpha_ig_i$ with all $\alpha_i>0$, and $\lambda_Z|_{\mcal{C}}$ is linear. Then $\lambda_Z(g)=\sum\lambda_Z(g_i)$, where $g=\sum g_i\in\mcal C_\Q$.
By Step 1 there is a point $v=\alpha g+(1-\alpha)u$ with $0<\alpha<1$ such that the map $\lambda_Z|_{\R_+g+\R_+v}$ is linear, and in particular
$\lambda_Z(g+v)=\lambda_Z(g)+\lambda_Z(v)$. Now we have
$$\lambda_Z\big(\sum g_i+v\big)=\lambda_Z(g+v)=\lambda_Z(g)+\lambda_Z(v)=\sum \lambda_Z(g_i)+\lambda_Z(v),$$
so the map $\lambda_Z|_{\mcal{C}+\R_+v}$ is linear by Lemma \ref{linear}. Let $\mu=\max_i\{\frac{\alpha}{\alpha_i(1-\alpha)}\}$,
and set $t=\frac{\mu}{\mu+1}s+\frac{1}{\mu+1}u\in(s,u)$. Then it is easy to check that
$$t=\sum\textstyle\frac{\alpha_i}{\mu+1}\big(\mu-\frac{\alpha}{\alpha_i(1-\alpha)}\big)g_i+\frac{1}{(\mu+1)(1-\alpha)}v\in\mcal C+\R_+v,$$
and so the map $\lambda_Z|_{\R_+s+\R_+t}$ is linear since $s\in\mcal C+\R_+v$.\\[2mm]
{\em Step 3.\/}
Finally, let $H$ be any $2$-plane in $\R^\ell$. By Steps 1 and 2, for every $s\in\Pi\cap H$ there is a positive number $\varepsilon_s$
such that $\lambda_Z|_{\R_+(\Pi\cap H\cap B(s,\varepsilon_s))}$ is piecewise linear.
By compactness, there are finitely many points $s_i\in\Pi\cap H$ such that $\Pi\cap H\subset\bigcup_iB(s_i,\varepsilon_{s_i})$,  and thus
$\lambda_Z|_{\mcal S_\R\cap H}$ is piecewise linear.
\end{proof}

Finally, we have

\begin{thm}\label{lem:PL}
For every prime divisor $Z$ on $S$, the map $\lambda_Z$ is rationally piecewise linear. Therefore, $\lambda$ is rationally piecewise linear.
\end{thm}
\begin{proof}
{\em Step 1.\/}
Let $\sum G_i=\bigcup_{s\in\mcal S}\Supp(\Delta_s-S-A)$, and set $\nu=\max_{i,s}\{\mult_{G_i}\Delta_s\}<1$. Let $0<\eta\ll1-\nu$ be a rational
number such that $A-\eta\sum G_i$ is ample, and let $\tilde A\sim_\Q A-\eta\sum G_i$ be a general ample $\Q$-divisor. Denote
$\tilde\Delta_s=\Delta_s-A+\eta\sum G_i+\tilde A\geq0$ for every $s\in\mcal S$, and observe that $\tilde\Delta_s\sim_\Q\Delta_s$,
$\lfloor\tilde\Delta_s\rfloor=S$ and $(S,\tilde\Omega_s=(\tilde\Delta_s-S)_{|S})$ is terminal since $\eta\ll1$.

Fix a sufficiently divisible positive integer $\kappa$ such that $\kappa t_s(K_X+\tilde\Delta_s)\in \Div(X)$ for all $s\in\mcal S$, and define
$$\tilde\Theta_s=\limsup_{m\rightarrow\infty}\textstyle\big(\tilde\Omega_s-\tilde\Omega_s\wedge\frac{1}{m\kappa t_s}\Fix|m\kappa t_s(K_X+\tilde\Delta_s)|_S\big).$$
Then as in Step 3 of the proof of Theorem \ref{thm:2} and in Notation \ref{nt}, we have associated maps
$\tilde\Theta,\tilde\lambda\colon\mcal S_\R\rightarrow\Div(S)_\R$ and $\tilde\Theta_Z,\tilde\lambda_Z\colon\mcal S_\R\rightarrow\R$.
Let $\mcal L_Z$ and $\tilde{\mcal L}_Z$ be the closures of sets $\{s\in\mcal S_\R:\Theta_Z(s)>0\}$ and $\{s\in\mcal S_\R:\tilde\Theta_Z(s)>0\}$, respectively;
they are closed cones. By construction, $\ord_Z\|\tilde\lambda(s)/\kappa t_s\|_S=\ord_Z\|\lambda(s)/t_s\|_S$, and thus
$\tilde\Theta_Z(s)=\Theta_Z(s)+\eta$ for every $s\in\mcal L_Z$. In particular, $\mcal L_Z$ is the closure of the set
$\{s\in\mcal S_\R:\tilde\Theta_Z(s)>\eta\}$, and $\mcal L_Z\subset\tilde{\mcal L}_Z$. Therefore, for every $s\in\mcal L_Z$ there is a sequence $s_m$
such that $\lim\limits_{m\rightarrow\infty}s_m=s$ and $\tilde\Theta_Z(s_m)>\eta$, thus similarly as in \cite[2.1.4]{Nak04}, we have
$\tilde\Theta_Z(s)\geq\limsup\limits_{m\rightarrow\infty}\tilde\Theta_Z(s_m)\geq\eta$.\\[2mm]
{\em Step 2.\/}
If $\mcal L_Z=\emptyset$, then $\lambda_Z$ is trivially a linear map, so until the end of the proof I assume $\mcal L_Z\neq\emptyset$.
In this step I prove that there is a rational polyhedral cone $\mcal M_Z$ such that $\mcal L_Z\subset\mcal M_Z\subset\tilde{\mcal L}_Z$.

I first show that for every point $x\in\Pi\cap\mcal L_Z$ there is a neighbourhood $\mcal U\subset\R^\ell$ of $x$, in the sup-norm, such that
$\mcal U\cap\mcal S_\R\subset\tilde{\mcal L}_Z$. To that end, recall that $\mcal S_\R=\sum\R_+e_i$, and choose points $x_i\in\R_+e_i\backslash\{x\}$.
Since $\tilde\Theta_Z(x)>0$, by Theorem \ref{lem:lipschitz}(1) there exists a point $y_i\in(x,x_i)$ such that $\tilde\Theta_Z(y_i)>0$
for each $i$. Therefore $\sum\R_+y_i\subset\tilde{\mcal L}_Z$, and it is sufficient to take any neighbourhood
$\mcal U$ of $x$ such that $\mcal U\cap\mcal S_\R\subset\sum\R_+y_i$.

By compactness, there is a rational number $0<\xi\ll1$ and finitely many rational points $z_1,\dots,z_p\in\Pi\cap\mcal L_Z$ such that
$\mcal L_Z\subset\bigcup\big(\R_+B(z_i,\xi)\big)\cap\mcal S_\R\subset\tilde{\mcal L}_Z$.
The convex hull $\mcal B$ of $\bigcup B(z_i,\xi)$ is a rational polytope, and define $\mcal M_Z=\R_+\mcal B\cap\mcal S_\R$.\\[2mm]
{\em Step 3.\/}
Note that, by construction, $\tilde\Theta_Z(s)>0$ for all $s\in\mcal M_Z$, and that the map $\tilde\lambda_Z|_{\mcal M_Z}$ is {\em superlinear\/},
cf.\ the argument in Notation \ref{nt}.

I claim that it is enough to prove that $\tilde\lambda_Z|_{\mcal M_Z}$ is rationally piecewise linear. To that end, since $\mcal L_Z$ is the
closure of the set $\{s\in\mcal S_\R:\tilde\Theta_Z(s)>\eta\}$ and $\eta\in\Q$, we have that then $\mcal L_Z$ is a rational polyhedral cone, and thus
the map $\tilde\lambda_Z|_{\mcal L_Z}$ is rationally piecewise linear. Therefore so is $\lambda_Z$, since
$\tilde\Theta_Z(s)=\Theta_Z(s)+\eta$ for every $s\in\mcal L_Z$, and this proves the claim.

By Lemma \ref{lem:gordan}, there are finitely many generators $g_i$ of $\mcal M_Z\cap\mcal S$, and let
$\varphi\colon\bigoplus_i\N g_i\rightarrow\mcal M_Z\cap\mcal S$ be the projection map. Replacing $\mcal S$ by $\bigoplus_i\N g_i$, $\lambda_Z$ by
$\tilde\lambda_Z\circ\varphi$ and $\Theta_Z$ by $\tilde\Theta_Z\circ\varphi$, I can assume that $\lambda_Z$ is a superlinear function on
$\mcal S_\R$ and $\Theta_Z(s)>0$ for all $s\in\mcal S_\R$.

By Theorem \ref{thm:linear1}, for any $2$-plane $H\subset\R^\ell$ the map $\lambda_Z|_{\mcal S_\R\cap H}$ is piecewise linear, and thus
$\lambda_Z$ is piecewise linear by Theorem \ref{piecewise}.

Finally, to prove that $\lambda_Z$ is {\em rationally\/} piecewise linear, let $\mcal S_\R=\bigcup\mcal C_m$ be a finite polyhedral decomposition
such that the maps $\lambda_Z|_{\mcal C_m}$ are linear, and their linear extensions to $\R^\ell$ are pairwise different.
Let $\mcal F$ be a common $(\ell-1)$-dimensional face of cones $\mcal{C}_i$ and $\mcal{C}_j$, and assume $\mcal F$ does not belong to a rational
hyperplane. Let $\mcal H$ be the smallest affine space containing $\mcal F_\Pi=\mcal F\cap\Pi$, and note that $\mcal H$ is not rational and
$\dim\mcal H=\ell-2$. If for every $f\in\mcal F_\Pi$ there existed a rational affine space $\mcal H_f\ni f$ of dimension $\ell-2$,
this would contradict Remark \ref{rem:4} since countably many $\mcal H_f\cap\mcal H$ would cover $\mcal F_\Pi\subset\mcal H$.

Therefore, there is a point $s\in\mcal F_\Pi$ and an $\ell$-dimensional cone $\mcal{C}_s$ such that
$s\in\Int\mcal{C}_s$ and the map $\tilde\lambda_Z|_{\mcal{C}_s}$ is linear, by Theorem \ref{lem:lipschitz}(2).
But then the cones $\mcal{C}_s\cap\mcal{C}_i$ and $\mcal{C}_s\cap\mcal{C}_j$ are $\ell$-dimensional and linear
extensions of $\lambda_Z|_{\mcal{C}_i}$ and $\lambda_Z|_{\mcal{C}_j}$
coincide since they are equal to the linear extension of $\lambda_Z|_{\mcal{C}_s}$, a contradiction.
Thus all $(\ell-1)$-dimensional faces of the cones $\mcal{C}_m$ belong to rational $(\ell-1)$-planes, so $\mcal{C}_m$ are rational cones.
\end{proof}

\section{Stable base loci}\label{sec:stable}

\begin{thm}\label{thm:EimpliesLGA}
Theorems A$_{n-1}$ and C$_{n-1}$ imply Theorem B$_n$.
\end{thm}
\begin{proof}
Let $X$ be a smooth variety as in Theorem B$_n$, let $K_X$ be a divisor with $\OO_X(K_X)\simeq\omega_X$ and $A\not\subset\Supp K_X$, and denote
$\mcal C=\R_+(K_X+A+\mcal B_{V,A}^{G=1})$. It suffices to prove that the cone $\mcal C$ is rational polyhedral.\\[2mm]
{\em Step 1.\/}
Denote
$$\mcal D_{V,A}^{G=1}=\{\Phi\in\mcal L_V:\mult_G\Phi=1,\,\sigma_G\|K_X+\Phi+A\|=0\}.$$
This is a convex set, and it is also closed: if $D_m\in K_X+A+\mcal D_{V,A}^{G=1}$ is a sequence such that $\lim\limits_{m\rightarrow\infty}D_m=D\in K_X+A+\mcal L_V$,
then $\sigma_G\|D\|\leq\liminf\limits_{m\rightarrow\infty}\sigma_G\|D_m\|=0$ by \cite[2.1.4]{Nak04}, so $D\in K_X+A+\mcal D_{V,A}^{G=1}$.

In this step I show that $\mcal B_{V,A}^{G=1}=\mcal D_{V,A}^{G=1}$, and also that $\mcal C$ is a rational cone, i.e.\ that its extremal rays are rational.
Note that $\mcal B_{V,A}^{G=1}\subset\mcal D_{V,A}^{G=1}$ is trivial by Lemma \ref{lem:restrictedord}(1), so I concentrate on proving the reverse inclusion.

Let $\Delta\in\mcal L_V+A$ be a divisor such that $\mult_G\Delta=1$ and $\sigma_G\|K_X+\Delta\|=0$. I first claim that we can assume that
$(X,\Delta)$ is plt, $\lfloor\Delta\rfloor=G$, and $(X,\Omega+A_{|G})$ is terminal, where $\Omega=(\Delta-G)_{|G}$.

To that end, let $\mcal F$ be the set of prime divisors $F\neq G$ with $\mult_F\Delta=1$, and choose $0<\eta\ll1$ such that
$A+\Xi$ is ample, where $\Xi=\eta\sum_{F\in\mcal F}F$. Replacing $A$ by a general ample $\Q$-divisor $\Q$-linearly equivalent to $A+\Xi$ and
$\Delta$ by $\Delta-\Xi$, we can assume that $(X,\Delta)$ is plt and $\lfloor\Delta\rfloor=G$.
Let $f\colon Y\rightarrow X$ be a log resolution such that the components of $\{\B(X,\Delta)_Y\}$ are disjoint as in Lemma \ref{disjoint},
and in particular $(Y,(\Delta'-G')_{|G'})$ is terminal, where $G'=f^{-1}_*G$ and $\Delta'=\B(X,\Delta)_Y$. Note that $f^*A=f^{-1}_*A\leq\Delta'$ since $A$ is
general, let $H$ be a small effective $f$-exceptional divisor such that $f^*A-H$ is ample, and let $A'\sim_\Q f^*A-H$ be a general ample $\Q$-divisor.
Let $V'$ be the vector space spanned by proper transforms of elements of $V$ and by exceptional divisors.
Then $\Delta'\in\mcal D_{V',A'}^{G'=1}+A'$ by Remark \ref{rem:5}, so it is enough to show that $\Delta'\in\mcal B_{V',A'}^{G'=1}+A'$ and that the cone
$\R_+(K_Y+A'+\mcal B_{V',A'}^{G'=1})$ is rational locally around $K_Y+\Delta'$.
Replacing $X$ by $Y$, $G$ by $G'$, $\Delta$ by $\Delta'-f^*A+H+A'$ and $V$ by $V'$ proves the claim.

Since $\sigma_G\|K_X+\Delta\|=0$, by Remark \ref{rem:5} the formal sum $N_\sigma\|K_X+\Delta\|_G$ is well-defined and $K_G+\Theta$ is
pseudo-effective, where $\Theta=\Omega-\Omega\wedge N_\sigma\|K_X+\Delta\|_G$.
Let $\phi<1$ be the smallest positive coefficient of $\Omega-\Theta$ if it exists, and set $\phi=1$ otherwise.
Denote by $V_G\subset\Div(G)_\R$ the vector space spanned by components of divisors in $\{F_{|G}:F\in V,G\not\subset\Supp F\}$.
Let $r$ be a positive integer as in Lemma \ref{bounded} with respect to $V_G$ and $A_{|G}$, and let $W\subset\Div(X)_\R$ and $U\subset\Div(G)_\R$
be the smallest rational affine subspaces containing $\Delta$ and $\Theta$, respectively.
Choose $\varepsilon>0$ such that $\varepsilon(K_X+\tilde\Delta)+\frac12A$ and $\tilde\Delta-\Delta+\frac12A$ are ample divisors whenever
$\tilde\Delta\in B(\Delta,\varepsilon)$.

Then by Lemma \ref{lem:surround} there exist rational points $(\Delta_i,\Theta_i)\in W\times U$, integers $k_i\gg0$, and $r_i\in\R_{>0}$ such that
$(\Delta_i,\Theta_i,r,k_i,r_i)$ uniformly approximate $(\Delta,\Theta)\in W\times U$ with error $\phi\varepsilon/2$.
Note that then, for each $i$, $(X,\Delta_i)$ is plt, $(G,\Omega_i+A_{|G})$ is terminal with $\Omega_i=(\Delta_i-G)_{|G}$, and $\Theta_i\leq\Omega_i$ by Remark
\ref{rem:6}.

Since $\sigma_G\|K_X+\Delta\|=0$ we have $G\not\subset\B(K_X+\Delta+\frac12A_i)$ by Remark \ref{rem:5},
and since $\Delta-\Delta_i+\frac12A_i$ is ample, it follows that $G\not\subset\B(K_X+\Delta_i+A_i)$. Therefore,
similarly as in Step 1 of the proof of Theorem \ref{lem:lipschitz},
\begin{equation}\label{eq:25}
|k_i(K_G+\Theta_i)|+k_i(\Omega_i-\Theta_i)\subset|k_i(K_X+\Delta_i)|_G.
\end{equation}
In particular, since $U$ is the smallest rational affine space containing $\Theta$, $k_i\gg0$ and $\mcal E_{V_G,A_{|G}}$ is a rational polytope
by Theorem C$_{n-1}$, we have $\Theta_i-A_{|G}\in\mcal E_{V_G,A_{|G}}$, and Lemma \ref{bounded}(2) yields $|k_i(K_G+\Theta_i)|\neq\emptyset$.
Thus \eqref{eq:25} implies that there is an effective divisor $D_i\in|k_i(K_X+\Delta_i)|$ with $G\not\subset\Supp D_i$.
But then $K_X+\Delta\sim_\R\sum\frac{r_i}{k_i}D_i$ and $G\not\subset\B(K_X+\Delta)$, so $\Delta\in\mcal B_{V,A}^{G=1}+A$, as desired.\\[2mm]
{\em Step 2.\/}
It remains to prove that $\mcal C$ is polyhedral. To that end, I will prove it has only finitely many extremal rays.

Assume that there are distinct rational divisors $\Delta_m\in\mcal B_{V,A}^{G=1}+A$ for $m\in\N\cup\{\infty\}$ such that the rays $\R_+\Upsilon_m$ are
extremal in $\mcal C$ and $\lim\limits_{m\rightarrow\infty}\Delta_m=\Delta_\infty$, where $\Upsilon_m=K_X+\Delta_m$. As explained in
Remark \ref{rem:8}, I achieve contradiction by showing that for some
$m\gg0$ there is a point $\Upsilon'_m\in\mcal C$ such that $\Upsilon_m\in(\Upsilon_\infty,\Upsilon'_m)$.

I claim that we can assume that $(X,\Delta_m)$ is plt, $\lfloor\Delta_m\rfloor=G$, and each pair $(G,\Omega_m+A_{|G})$ is canonical for $m\in\N\cup\{\infty\}$,
where $\Omega_m=(\Delta_m-G)_{|G}$. To that end, by passing to a subsequence, as in Step 1 we can assume that $(X,\Delta_m)$ is plt and
$\lfloor\Delta_m\rfloor=G$ for each $m$. Let $f\colon Y\rightarrow X$ be a log resolution such that the components of $\{\B(X,\Delta_\infty)_Y\}$ are disjoint
as in Lemma \ref{disjoint}, and in particular $(Y,(\B(X,\Delta_m)_Y-G')_{|G'})$ is terminal for $m\gg0$, where $G'=f^{-1}_*G$.
Let $H$, $A'$ and $V'$ be as in Step 1, and denote $\Delta_m'=\B(X,\Delta_m)_Y-f^*A+H+A'$. Now, if for every $m\gg0$ there is
a divisor $\tilde\Delta_m\in\mcal B_{V',A'}^{G'=1}+A'$ such that $K_Y+\Delta_m'\in(K_Y+\Delta_\infty',K_Y+\tilde\Delta_m)$,
then $f_*\tilde\Delta_m\in\mcal B_{V,A}^{G=1}+A$ and $[\Upsilon_m]\in([\Upsilon_\infty],[K_X+f_*\tilde\Delta_m])$
as $\Upsilon_m\sim_\Q K_X+f_*\Delta_m'$ for all $m$. Therefore, since $\sigma_G\|K_X+f_*\tilde\Delta_m\|=0$ by Step 1, the ray $\R_+\Upsilon_m$
is not extremal in $\mcal C$, as explained in Remark \ref{rem:8}. Replacing $X$ by $Y$, $G$ by $G'$ and $\Delta_m$ by $\Delta_m'$ proves the claim.

Let $\Theta_m=\Omega_m-\Omega_m\wedge N_\sigma\|\Upsilon_m\|_G$, and note that $\Theta_m=\Omega_m-\Omega_m\wedge(\sum\ord_P\|\Upsilon_m\|_G\cdot P)$
by the relation \eqref{tag:5} in Theorem \ref{lem:lipschitz}.
By Step 3 of the proof of Theorem \ref{thm:2}, each $\Theta_m$ is a rational divisor,
and $\Theta_\infty\geq\limsup\limits_{m\rightarrow\infty}\Theta_m$ as in the proof of \cite[2.1.4]{Nak04}.
By passing to a subsequence, we can assume that there is a divisor $\Theta_\infty^0$ such that
$\lim\limits_{m\rightarrow\infty}\Theta_m=\Theta_\infty^0\leq\Theta_\infty$. If we define $V_G$ as in Step 1, then $\mcal E_{V_G,A_{|G}}$ is a rational
polytope by Theorem C$_{n-1}$, and thus
\begin{equation}\label{equ:9}
\Theta_\infty^0-A_{|G}\in\mcal E_{V_G,A_{|G}}
\end{equation}
since $\Theta_m-A_{|G}\in\mcal E_{V_G,A_{|G}}$ for all $m$.

Let $\mcal P$ be the set of all prime divisors $P$ on $S$ such that $\mult_P(\Omega_\infty-\Theta_\infty^0)>0$.
If $\mcal P\neq\emptyset$, by passing to a subsequence we can assume that
$$\phi=\min\{\mult_P(\Omega_m-\Theta_m):P\in\mcal P,m\in\N\cup\{\infty\}\}>0,$$
and set $\phi=1$ if $\mcal P=\emptyset$. Let $r$ be a positive integer as in Lemma \ref{bounded} with
respect to $V_G$ and $A_{|G}$, and let $U^0$ be the smallest rational affine space containing $\Theta_\infty^0$.

Let $0<\varepsilon\ll1$ be a rational number such that $\varepsilon(K_X+\tilde\Delta)+\frac12A$ and
$\tilde\Delta-\Delta_\infty+\frac12A$ are ample divisors whenever $\tilde\Delta\in B(\Delta_\infty,2\varepsilon)$.
Let $q$ be a positive integer such that $q\Delta_\infty/r$ is integral. By Lemma \ref{lem:surround} there exist
a $\Q$-divisor $\tilde\Theta\in U^0$ and a positive integer $k_\infty\gg0$ such that $\|\tilde\Theta-\Theta_\infty^0\|<\phi\varepsilon/2k_\infty$ and
$k_\infty\tilde\Theta/q$ is integral; in particular $k_\infty\Delta_\infty/r$ and $k_\infty\tilde\Theta/r$ are integral, and
$\tilde\Theta-A_{|G}\in\mcal E_{V_G,A_{|G}}$ by \eqref{equ:9} since $\mcal E_{V_G,A_{|G}}$ is a rational polytope by Theorem C$_{n-1}$. By passing to a subsequence again,
we can assume that $\|\Delta_\infty-\Delta_m\|<\phi\varepsilon/2k_\infty$ and $\|\tilde\Theta-\Theta_m\|<\phi\varepsilon/2k_\infty$ for all $m$.

Then by Lemma \ref{lem:approximation}, for every $m\in\N$ there is a point $(\Delta_m',\Theta_m')\in\Div(X)_\Q\times\Div(G)_\Q$ and a positive integer
$k_m\gg0$ such that:
\begin{enumerate}
\item $\Delta_m=\frac{k_\infty}{k_\infty+k_m}\Delta_\infty+\frac{k_m}{k_\infty+k_m}\Delta_m'$ and $\Theta_m=\frac{k_\infty}{k_\infty+k_m}\tilde\Theta+
\frac{k_m}{k_\infty+k_m}\Theta_m'$,
\item $k_m\Delta_m'/r$ is integral and $\|\Delta_m-\Delta_m'\|<\phi\varepsilon/2k_m$,
\item $k_m\Theta_m'/r$ is integral and $\|\Theta_m-\Theta_m'\|<\phi\varepsilon/2k_m$.
\end{enumerate}
Denote $\Omega_m'=(\Delta_m'-G)_{|G}$, and note that $\Theta_m'\leq\Omega_m'$ by Remark \ref{rem:6}.
Furthermore, since $\mcal E_{V_G,A_{|G}}$ is a rational polytope, for $m\gg0$ we have
$$[\tilde\Theta,\Theta_m]\subsetneq\big(\tilde\Theta+\R_+(\Theta_m-\tilde\Theta)\big)\cap\mcal E_{V_G,A_{|G}},$$
so in particular $\Theta_m'\in\mcal E_{V_G,A_{|G}}$ since $k_m\gg0$. I claim that
\begin{equation}\label{eq:7}
|k_m(K_G+\Theta_m')|+k_m(\Omega_m'-\Theta_m')\subset|k_m(K_X+\Delta_m')|_G.
\end{equation}
Grant the claim for the moment. Then $|k_m(K_G+\Theta_m')|\neq\emptyset$ by Lemma \ref{bounded}(2), and thus
$G\not\subset\B(K_X+\Delta_m')$ by \eqref{eq:7}. But then by the condition (1) above, the ray $\R_+\Upsilon_m$ is not extremal in $\mcal C$, a contradiction.

Now I prove the claim. By Proposition \ref{pro:2}, it is enough to show that for a prime divisor $P$ on $S$ and for $l>0$
sufficiently divisible we have
\begin{equation}\label{equ:4}
\mult_P(\Omega_m'\wedge\textstyle\frac1l\Fix|l(K_X+\Delta_m'+A_m)|_S)\leq\mult_P(\Omega_m'-\Theta_m'),
\end{equation}
where $A_m=A/k_m$. First I show
\begin{equation}\label{equ:6}
\textstyle\big(1-\frac{\varepsilon}{k_m}\big)\mult_P(\Omega_m-\Theta_m)\leq\mult_P(\Omega_m'-\Theta_m').
\end{equation}
To see this, note that \eqref{equ:6} is trivial when $\mult_P(\Omega_m-\Theta_m)=0$, so I assume $\mult_P(\Omega_m-\Theta_m)>0$.
If $P\notin\mcal P$, then $\mult_P\Theta_\infty^0=\mult_P\Omega_\infty\in\Q$, so in particular $\mult_P\tilde\Theta=\mult_P\Omega_\infty$ as $\tilde\Theta\in U^0$.
Therefore, by condition (1) above,
$$\textstyle\mult_P(\Omega_m-\Theta_m)\leq\frac{k_\infty+k_m}{k_m}\mult_P(\Omega_m-\Theta_m)=\mult_P(\Omega_m'-\Theta_m'),$$
which implies \eqref{equ:6}. If $P\in\mcal P$, then $\mult_P(\Omega_m-\Theta_m)\geq\phi$, and \eqref{equ:6} follows similarly as \eqref{tag:3}
in Step 1 of the proof of Theorem \ref{lem:lipschitz}.

Therefore, by \eqref{equ:4} and \eqref{equ:6} it suffices to prove that
$$\mult_P(\Omega_m'\wedge\textstyle\frac1l\Fix|l(K_X+\Delta_m'+A_m)|_S)\leq\big(1-\frac{\varepsilon}{k_m}\big)\mult_P(\Omega_m-\Theta_m).$$
But this is proved similarly as \eqref{equ:7} in Step 1 of the proof of Theorem \ref{lem:lipschitz}, and we are done.
\end{proof}

\section{Pseudo-effectivity and non-vanishing}\label{sec:non-vanishing}

In this section I prove the following.

\begin{thm}\label{cor:linear}
Theorems A$_{n-1}$, B$_n$ and C$_{n-1}$ imply Theorem C$_n$.
\end{thm}

I first make a few remarks that will be used in the proof.

\begin{rem}\label{rem:9}
Let $D\leq0$ be a divisor on a smooth variety $X$. I claim that then $D$ is pseudo-effective if and only if $D=0$. To that end,
if $A$ is an ample divisor, then $D+\varepsilon A$ is big for every $\varepsilon>0$. In particular, $|D+\varepsilon A|_\R\neq\emptyset$,
and thus $\deg(D+\varepsilon A)\geq0$. But then letting $\varepsilon\downarrow0$ implies $\deg D\geq0$, so $D=0$.
\end{rem}

\begin{rem}\label{rem:7}
With notation from Theorem B, let $0<\xi\ll1$ be a rational number such that $A-\Xi$ is ample for all $\Xi\in V$ with $\|\Xi\|\leq\xi$, let
$\mcal L_{V,\xi}$ be the $\xi$-neighbourhood of $\mcal L_V$ in the sup-norm, and set
$$\mcal B_{V,A,\xi}^{G=1}=\{\Phi\in\mcal L_{V,\xi}:\mult_G\Phi=1,\,G\not\subset\B(K_X+\Phi+A)\}.$$
Then I claim Theorem B implies that $\mcal B_{V,A,\xi}^{G=1}$ is a rational polytope. To that end, fix $\Phi\in\mcal B_{V,A,\xi}^{G=1}$.
Let $\mcal Z$ be the set of all prime divisors $Z\in V\backslash\{G\}$ such that $\mult_Z\Phi\geq1$, and let $A'\sim_\Q A+\xi\sum_{Z\in\mcal Z}Z$ be
a general ample $\Q$-divisor. Then for every $\Phi'\in\mcal B_{V,A,\xi}^{G=1}$ with $\|\Phi-\Phi'\|<\xi$, we have
$\Phi'-\xi\sum_{Z\in\mcal Z}Z\in\mcal B_{V,A'}^{G=1}$ since $\xi\ll1$. As $\mcal B_{V,A'}^{G=1}$ is a rational polytope
by Theorem B, this implies that $\mcal B_{V,A,\xi}^{G=1}$ is locally a rational polytope around $\Phi$, and the claim follows by compactness of
$\mcal B_{V,A,\xi}^{G=1}$.
\end{rem}

\begin{proof}[Proof of Theorem \ref{cor:linear}]
Let $X$ be a smooth variety and $B$ a divisor on $X$ as in Theorem C$_n$.
Fix a divisor $K_X$ such that $\OO_X(K_X)\simeq\omega_X$ and $A\not\subset\Supp K_X$. It suffices to prove that the cone
$\mcal C=\R_+(K_X+A+\mcal E_{V,A})\subset\Div(X)_\R$ is rational polyhedral. Observe that $\mcal C$ is closed since $\mcal E_{V,A}$ is.\\[2mm]
{\em Step 1.\/}
Fix $\Delta\in\mcal E_{V,A}+A$. I first show that there exists an effective divisor $D\in\Div(X)_\R$ such that $K_X+\Delta\equiv D$.
This was proved essentially in \cite{Hac08}, and I will sketch the proof here for completeness.

First I claim that we can assume $(X,\Delta)$ is klt. To see this, let $\mcal G$ be the set of all prime divisors $G$ with $\mult_G\Delta=1$
and choose $0<\eta\ll1$ such that $A+\eta\sum_{G\in\mcal G}G$ is ample.
Let $A'\sim_\Q A+\eta\sum_{G\in\mcal G}G$ be a general ample $\Q$-divisor and set $\Delta'=\Delta-\eta\sum_{G\in\mcal G}G+A'$.
Then $K_X+\Delta\sim_\Q K_X+\Delta'$ and $(X,\Delta')$ is klt, so replace $\Delta$ by $\Delta'$ and $A$ by $A'$.

Now, if $\nu(X,D)=0$, cf.\ Definition \ref{dfn:2}, then the result follows from \cite[3.3.2]{BCHM}.
If $\nu(X,D)>0$, then by \cite[6.2]{BCHM} we can assume
that $(X,\Delta)$ is plt, $A$ is a general ample $\Q$-divisor, $\lfloor\Delta\rfloor=S$, $(S,(\Delta-S)_{|S})$ is canonical,
and $\sigma_S\|K_X+\Delta\|=0$. But now as in Step 1 of the proof of Theorem \ref{thm:EimpliesLGA} we have $|K_X+\Delta|_\R\neq\emptyset$.\\[2mm]
{\em Step 2.\/}
In this step we assume further that $\Delta\in\Div(X)_\Q$, and prove that $|K_X+\Delta|_\Q\neq\emptyset$.
This argument uses Shokurov's trick from his proof of the classical Non-vanishing theorem, and I
will present an algebraic proof following the analytic version from \cite{Pau08}.

By Step 1, $K_X+\Delta\equiv D$ for some effective $\R$-divisor $D$, and write $\Delta=\Phi+A$.
Let $W\subset\Div(X)_\R$ be the vector space spanned by the components of $K_X$, $A$, $D$ and by the prime divisors in $V$,
and let $\phi\colon W\rightarrow N^1(X)$ be the linear map sending a divisor to its numerical class. Since $\phi^{-1}([K_X+\Delta])$
is a rational affine subspace of $W$, we can assume that $D$ is an effective $\Q$-divisor.

Let $m$ be a positive integer such that $m\Delta$ and $mD$ are integral. By Nadel vanishing
$$H^i\big(X,\mcal J_{(m-1)D+\Phi}(m(K_X+\Delta))\big)=0\quad\textrm{and}\quad H^i\big(X,\mcal J_{(m-1)D+\Phi}(mD)\big)=0$$
for $i>0$, and since the Euler characteristic is a numerical invariant,
\begin{equation}\label{eq:12}
h^0\big(X,\mcal J_{(m-1)D+\Phi}(m(K_X+\Delta))\big)=h^0\big(X,\mcal J_{(m-1)D+\Phi}(mD)\big).
\end{equation}
Let $\sigma\in H^0(X,mD)$ be the section with $\ddiv\sigma=mD$. Since
$$((m-1)D+\Phi)-mD\leq\Phi,$$
by \cite[4.3(3)]{HM08} we have $\mcal I_{mD}\subset\mcal J_{(m-1)D+\Phi}$, and thus
$$\sigma\in H^0\big(X,\mcal J_{(m-1)D+\Phi}(mD)\big).$$
Therefore \eqref{eq:12} implies $h^0(X,m(K_X+\Delta))>0$, and we are done.\\[2mm]
{\em Step 3.\/}
In this step I prove that $\mcal C$ is a rational cone, and that
$$\mcal E_{V,A}=\{\Phi\in\mcal L_V:|K_X+\Phi+A|_\R\neq\emptyset\}.$$
Fix $\Delta\in\mcal E_{V,A}+A$. By Step 1 there is an effective $\R$-divisor $D$ such that $K_X+\Delta\equiv D$. Write $\Delta=A+\sum\delta_iF_i$
with $\delta_i\in[0,1]$, and $D=\sum f_iF_i$, where we can assume $F_i\neq\Supp A$ for all $i$ since $A$ is general.

I claim that we can assume that $\sum F_i$ has simple normal crossings. To that end,
let $f\colon Y\rightarrow X$ be a log resolution of $(X,\sum F_i)$, and denote $G'=f^{-1}_*G$ and $\Delta'=\B(X,\Delta)_Y$. Note that
$f^*A=f^{-1}_*A\leq\Delta'$ since $A$ is
general, let $H$ be a small effective $f$-exceptional divisor such that $f^*A-H$ is ample, and let $A'\sim_\Q f^*A-H$ be a general ample $\Q$-divisor.
Let $V'$ be the vector space spanned by proper transforms of elements of $V$ and by exceptional divisors.
Then $K_Y+\Delta'\equiv f^*D+E$, where $E=K_Y+\Delta'-f^*(K_X+\Delta)$ is effective and $f$-exceptional, and $\Delta'\in\mcal E_{V',A'}^{G'=1}+A'$,
so it is enough to show that the cone $\R_+(K_Y+A'+\mcal E_{V',A'}^{G'=1})$ is rational locally around $K_Y+\Delta'$.
Replacing $X$ by $Y$, $G$ by $G'$, $\Delta$ by $\Delta'-f^*A+H+A'$ and $V$ by $V'$ proves the claim.

Define $W$ and $\phi$ as in Step 2, and let $0<\varepsilon\ll1$ be a rational number such that $A+\Phi$ is ample for any divisor $\Phi\in W$ with $\|\Phi\|\leq\varepsilon$.
Choose $0\leq f_i'\leq f_i$ be rational numbers such that $f_i-f_i'<\varepsilon$. Then
$$K_X+\Delta'\equiv\sum f_i'F_i,$$
where $\Delta'=\Delta-\sum(f_i-f_i')F_i$. Since $\mcal P=\phi^{-1}\big(\sum f_i'[F_i]\big)$ is a rational affine subspace of $W$,
there are rational divisors $\Delta_j\in V+A$ such that $\|\Delta'-\Delta_j\|\ll\varepsilon$,
$K_X+\Delta_j\in\mcal P$ and $K_X+\Delta'=\sum\rho_j(K_X+\Delta_j)$ for some positive numbers $\rho_j$ with $\sum\rho_j=1$.
Setting $\Phi_j=\sum\max\{0,\mult_{F_j}\Delta_j-\varepsilon\}F_j$, the divisor $\Delta_j-\Phi_j$ is ample since
$\|(\Delta_j-A)-\Phi_j\|\leq\varepsilon$, and let $A'\sim_\Q\Delta_j-\Phi_j$ be a general ample $\Q$-divisor.
Therefore each $K_X+\Delta_j\sim_\Q K_X+\Phi_j+A'$ is a rational pseudo-effective divisor, and since $(X,\Phi_j+A')$ is klt, it is $\Q$-linearly
equivalent to an effective divisor by Step 2.
For each $j$, denote $\mcal B_j=\sum[\mult_{F_i}\Delta_j,1]F_i$, and let $\mcal B$ be the convex hull of $\bigcup\mcal B_j$;
observe that $\mcal B$ is a rational polytope. Then, since $V\subset W$,
\begin{equation}\label{tag:2}
K_X+\Delta\in(K_X+A+\mcal B)\cap(K_X+A+\mcal L_V).
\end{equation}
Therefore, \eqref{tag:2} shows that $K_X+\Delta$ is $\R$-linearly equivalent to an effective divisor, and that
$\mcal C$ is locally rational around every $K_X+\Delta$, and thus it is a rational cone.\\[2mm]
{\em Step 4.\/}
It remains to prove that the cone $\mcal C$ is polyhedral, i.e.\ that it has finitely many extremal rays.
Let $G_1,\dots,G_N$ be prime divisors on $X$ such that $\Supp K_X\cup\Supp B\subset\sum G_i$.

Assume that $\mcal C$ has infinitely many extremal rays. Thus, since $\mcal C$ is a rational cone,
there are distinct rational divisors $\Delta_m\in\mcal E_{V,A}+A$ for $m\in\N\cup\{\infty\}$ such that the rays $\R_+\Upsilon_m$
are extremal in $\mcal C$ and $\lim\limits_{m\rightarrow\infty}\Delta_m=\Delta_\infty$, where $\Upsilon_m=K_X+\Delta_m$.
As explained in Remark \ref{rem:8}, I achieve contradiction by showing that for some
$m\gg0$ there is a point $\Upsilon'_m\in\mcal C$ such that $\Upsilon_m\in(\Upsilon_\infty,\Upsilon'_m)$.

By Step 2, there is an effective divisor $D_\infty$ such that $\Upsilon_\infty\sim_\Q D_\infty$.
By possibly adding components, cf.\ Remark \ref{rem:8}, I can assume that $\Supp D_\infty\subset\sum G_j$ and that
$V=\sum\R G_j$. Similarly as in Step 1, and possibly by passing to a subsequence, we can assume that $(X,\Delta_m)$ is klt for all $m$,
and by taking a log resolution as in Step 3, we can assume further that $\sum G_j$ has simple normal crossings.\\[2mm]
{\em Step 5.\/}
If $D_\infty=0$, then for all $m\in\N$ the class $[\Upsilon_m]$ belongs to the segment $([\Upsilon_\infty],[2\Upsilon_m])$, and $2\Upsilon_m$ is
pseudo-effective, so we derive contradiction as in Remark \ref{rem:8}.

Thus, until the end of the proof I assume that $D_\infty\neq0$, and write $D_\infty=\sum d_jG_j$.
Assume that $\Supp N_\sigma\|\Upsilon_\infty\|=\Supp D_\infty$. Then $N_\sigma\|\Upsilon_\infty\|=N_\sigma\|D_\infty\|=D_\infty$
by \cite[2.1.6]{Nak04}, and $[G_j]$ are linearly independent in $N^1(X)$. Let $\mcal E\subset N^1(X)$ denote the
pseudo-effective cone. Similarly as in the proof of \cite[3.19]{Bou04}, we have $\mcal E=\sum\R_+[G_j]+\bigcap_j\mcal E_{G_j}$, where
$\mcal E_{G_j}=\{\Xi\in\mcal E:\sigma_{G_j}\|\Xi\|=0\}$ is a closed cone for every $j$. I claim that if
\begin{equation}\label{tag:7}
[D_\infty]=\sum d_j'[G_j]+\Phi
\end{equation}
with $d_j'\geq0$ and $\Phi\in\bigcap_j\mcal E_{G_j}$, then $\Phi=0$. To that end, denote $\alpha_j=d_j-d_j'$ and let $J=\{j:\alpha_j>0\}$.
Then \eqref{tag:7} gives
$$\sum\nolimits_{j\in J}\alpha_j[G_j]={-}\sum\nolimits_{j\notin J}\alpha_j[G_j]+\Phi.$$
Assume that there exists $i_0\in J$. Then, by \cite[2.1.6]{Nak04} again,
\begin{align*}
0<\alpha_{i_0}&\textstyle=\sigma_{G_{i_0}}\|\sum_{j\in J}\alpha_jG_j\|=\sigma_{G_{i_0}}\|{-}\sum_{j\notin J}\alpha_j[G_j]+\Phi\|\\
&\textstyle\leq\sigma_{G_{i_0}}\|{-}\sum_{j\notin J}\alpha_jG_j\|+\sigma_{G_{i_0}}\|\Phi\|=0,
\end{align*}
a contradiction. Therefore $J=\emptyset$ and $\Phi=\sum_{j\notin J}\alpha_j[G_j]$, thus $\Phi=0$ by Remark \ref{rem:9}.

Therefore, by Lemma \ref{lem:4} applied to the cone $\mcal E$ and to the sequence $[\Upsilon_m]$, there exists $\widehat\Phi_m\in\mcal E$ such that
$[\Upsilon_m]\in([\Upsilon_\infty],\widehat\Phi_m)$, a contradiction by Remark \ref{rem:8}.\\[2mm]
{\em Step 6.\/}
Therefore, from now on I assume that $\Supp N_\sigma\|\Upsilon_\infty\|\neq\Supp D_\infty$, and in particular, there is an index $j_0$ such that
$d_{j_0}>0$ and $\sigma_{G_{j_0}}\|\Upsilon_\infty\|=0$.
For $m\in\N\cup\{\infty\}$, write $\Delta_m=A+\sum\delta_j^mG_j$ with $\delta_j^m\in[0,1)$. Then from $\Upsilon_\infty\sim_\Q D_\infty$ we have
$$K_X+A\sim_\Q\sum f_jG_j$$
with $f_j=d_j-\delta_j^\infty$.
In this step I prove that there exist pseudo-effective divisors $\Sigma_m$, for $m\in\N\cup\{\infty\}$, with the following properties:
\begin{enumerate}
\item[(i)] $\lim\limits_{m\rightarrow\infty}\Sigma_m=\Sigma_\infty$,
\item[(ii)] $\Sigma_\infty\in\sum f_jG_j+\mcal B_{V,A}^{G_{k_0}=1}$ for some $k_0$ with $\mult_{G_{k_0}}\Sigma_\infty>0$,
\item[(iii)] if for $m\gg0$ there exists a pseudo-effective divisor $\Sigma_m'$ such that $\Sigma_m\in(\Sigma_\infty,\Sigma_m')$,
then $\R_+\Upsilon_m$ is not an extremal ray of $\mcal C$.
\end{enumerate}
For each $t\in\R_+$, let $\Delta_\infty^t=\Delta_\infty+tD_\infty$ and $\Theta_\infty^t=\Delta_\infty^t-\Delta_\infty^t\wedge N_\sigma\|(t+1)\Upsilon_\infty\|$.
Note that $K_X+\Delta_\infty^t\sim_\Q(t+1)\Upsilon_\infty$, $\Theta_\infty^t$ is a continuous function in $t$, and
$\mult_{G_{j_0}}\Theta_\infty^t=\delta^\infty_{j_0}+td_{j_0}$ for all $t$.
Therefore, since $(X,\Delta_\infty^0)$ is klt, there exists $t_0\in\R_{>0}$ such that
$$t_0=\sup\{t\in\R_+:(X,\Theta_\infty^t)\textrm{ is log canonical}\}.$$
By construction, there is $k_0$ with $d_{k_0}>0$ such that $G_{k_0}$ is a log canonical centre of
$(X,\Theta_\infty^{t_0})$ and $\sigma_{G_{k_0}}\|K_X+\Theta_\infty^{t_0}\|=0$, thus by Theorem B$_n$ we have
\begin{equation}\label{equ:2}
\Theta_\infty^{t_0}-A\in\mcal B_{V,A}^{G_{k_0}=1}.
\end{equation}
Define $D_m=\sum(f_j+\delta_j^m)G_j$ and $\Xi_m=(t_0+1)D_m$ for $m\in\N\cup\{\infty\}$, and observe that
\begin{equation}\label{tag:8}
(t_0+1)\Upsilon_m\sim_\Q\Xi_m=\sum f_jG_j+\Delta_m-A+t_0D_m\sim_\Q K_X+\Delta_m+t_0D_m
\end{equation}
and $\lim\limits_{m\rightarrow\infty}\Xi_m=\Xi_\infty$. Denote $\Lambda_\infty=(\Delta_\infty+t_0D_\infty)\wedge N_\sigma\|\Xi_\infty\|$ and
$$\Lambda_m=(\Delta_m+t_0D_m)\wedge\sum_{Z\subset\Supp\Lambda_\infty}\sigma_Z\|\Xi_m\|$$
for $m\in\N$. Note that $0\leq\Lambda_m\leq N_\sigma\|\Xi_m\|$ for $m\gg0$, and therefore $\Xi_m-\Lambda_m$
is pseudo-effective. Similarly as in \cite[2.1.4]{Nak04} we have $\Lambda_\infty\leq\liminf\limits_{m\rightarrow\infty}\Lambda_m$, and in particular,
$\Supp\Lambda_m=\Supp\Lambda_\infty$ for $m\gg0$. Therefore, there exists a sequence of rational numbers $\varepsilon_m\uparrow1$ such that
$\Lambda_m\geq\varepsilon_m\Lambda_\infty$, and set $\varepsilon_\infty=1$.

Now define $\Sigma_m=\Xi_m-\varepsilon_m\Lambda_\infty$ for $m\in\N\cup\{\infty\}$, and note that
$\Sigma_m\geq\Xi_m-\Lambda_m$ are pseudo-effective divisors satisfying (1). Also, note that $\Sigma_\infty=\sum f_jG_j+\Theta_\infty^{t_0}-A$, and
$$\mult_{G_{k_0}}\Sigma_\infty=f_{k_0}+1\geq f_{k_0}+\delta_{k_0}^\infty=d_{k_0}>0,$$
so this together with \eqref{equ:2} gives (ii).

In order to show (iii), let $0<\alpha_m<1$ be such that $\Sigma_m=\alpha_m\Sigma_\infty+(1-\alpha_m)\Sigma_m'$.
Since every point on the segment $[\Sigma_m,\Sigma_m']$ is pseudo-effective, we can assume $\alpha_m\ll1$. Then setting
$\Upsilon_m'=\Sigma_m'+\frac{\varepsilon_m-\alpha_m}{1-\alpha_m}\Lambda_\infty$,
we have $\Sigma_m=\alpha_m\Sigma_\infty+(1-\alpha_m)\Upsilon_m'$, and this together with \eqref{tag:8} gives
$[\Upsilon_m]=\alpha_m[\Upsilon_\infty]+(1-\alpha_m)[\frac{1}{1+t_0}\Upsilon_m']$, so $\R_+\Upsilon_m$ is not an extremal ray of $\mcal C$ by
Remark \ref{rem:8}.\\[2mm]
{\em Step 7.\/}
Let $0<\xi\ll1$ be a rational number such that $A-\Xi$ is ample for all $\Xi\in V$ with $\|\Xi\|\leq\xi$, and let
$\mcal L_{V,\xi}$ be the $\xi$-neighbourhood of $\mcal L_V$ in the sup-norm. With notation from Remark \ref{rem:7},
set
$$\textstyle\mcal D_\xi=\R_+(\sum f_jG_j+\mcal B_{V,A,\xi}^{G=1})\subset V.$$
By Remark \ref{rem:7}, $\mcal D_\xi$ is a rational polyhedral cone. Note that
$\{\Theta_\infty^{t_0}-A+\Xi:0\leq\Xi\in V,\|\Xi\|\leq\xi,\mult_{G_{k_0}}\Xi=0\}\subset\mcal B_{V,A,\xi}^{G_{k_0}=1}$,
so $\dim\mcal D_\xi=\dim V$ and $\Sigma_\infty\in\mcal D_\xi$. If $\Sigma_\infty\in\Int\mcal D_\xi$, then it is obvious that
for $m\gg0$ there exists $\Sigma_m'\in\mcal D_\xi$ such that $\Sigma_m\in(\Sigma_\infty,\Sigma_m')$, which is a contradiction by (iii) above.

Otherwise, let $\mcal H_i$, for $i=1,\dots,\ell$, be the supporting hyperplanes of codimension $1$ faces of the cone $\mcal D_\xi$
which contain $\Sigma_\infty$, where $\ell\leq\dim V-1$. Let $\mcal W_i$ be the half-spaces determined by $\mcal H_i$
which contain $\mcal D_\xi$, and denote $\mcal Q=\bigcap_i\mcal W_i$.
If $\Sigma_m\in\mcal Q$ for infinitely many $m$, then for some $m\gg0$ there exists $\Sigma_m'\in\mcal D_\xi$ such that
$\Sigma_m\in(\Sigma_\infty,\Sigma_m')$ since $\mcal D_\xi$ is polyhedral, a contradiction again.

Therefore, by passing to a subsequence, I can assume that $\Sigma_m\notin\mcal Q$ for all $m$.
For each $m\in\N$, denote $\Gamma_m=\Sigma_m-\sigma_{G_{k_0}}\|\Sigma_m\|\cdot G_{k_0}$. I claim that
$\mcal D_\xi\cap(\R_{>0}\Gamma_m+\R_{>0}\Sigma_\infty)\neq\emptyset$ for each $m\in\N$, and in particular
$\mcal W_i\cap(\R_{>0}\Gamma_m+\R_{>0}\Sigma_\infty)\neq\emptyset$ for all $i$. Granting the claim, let me show how it yields contradiction.

Since $\Sigma_\infty\in\mcal H_i$ for every $i$, and the cone $\R_+\Gamma_m+\R_+\Sigma_\infty$ is convex,
the claim implies $\Gamma_m\in\mcal W_i$, and thus $\Gamma_m\in\mcal Q$. Since $\Sigma_m\notin\mcal Q$,
segments $[\Gamma_m,\Sigma_m]$ intersect $\partial\mcal Q$, and in particular there exists a point $P_m\in(\Sigma_m+\R_-G)\cap\partial\mcal Q$
closest to $\Sigma_m$. By passing to a subsequence, we have $\lim\limits_{m\rightarrow\infty}P_m=\Sigma_\infty$ by Remark \ref{rem:3}, and
thus for every $m\gg0$ there exists
a codimension $1$ face of $\mcal D_\xi$ that contains $\Sigma_\infty$ and $P_m$. Since $\mcal D_\xi$ is polyhedral, for $m\gg0$ there are points
$Q_m\in\mcal D_\xi$ such that $P_m=\mu_mQ_m+(1-\mu_m)\Sigma_\infty$ for some $0<\mu_m<1$. Set
$\Sigma_m'=Q_m+\frac{1}{\mu_m}(\Sigma_m-P_m)$, and note that $\Sigma_m'\geq Q_m$ is pseudo-effective. Then
$\Sigma_m=\mu_m\Sigma_m'+(1-\mu_m)\Sigma_\infty$, and this is a contradiction by (iii) above.

Finally, let me prove the claim stated above. Observe that for every $\Psi\in\R_+\Gamma_m+\R_+\Upsilon_\infty$ we have $\sigma_{G_{k_0}}\|\Psi\|=0$.
Therefore, as $\Sigma_\infty\in\sum f_jG_j+\mcal B_{V,A}^{G=1}$, it is enough to find
$\Pi_m\in(\R_{>0}\Gamma_m+\R_{>0}\Sigma_\infty)\cap B(\Sigma_\infty,\xi)$ such that $\mult_{G_{k_0}}\Pi_m=\mult_{G_{k_0}}\Sigma_\infty$.
Write $\Gamma_m=\sum\gamma_{m,j}G_j\geq0$ and $\Sigma_\infty=\sum\sigma_jG_j$, where $\sigma_{k_0}>0$ by the condition (ii) above.
If $\gamma_{m,k_0}\neq0$, choose $0<\beta_m<1$ so that
$(1-\beta_m)|\sigma_{k_0}\gamma_{m,j}-\sigma_j\gamma_{m,k_0}|<\xi\gamma_{m,k_0}$ for all $j$, and set $\alpha_m=(1-\beta_m)\sigma_{k_0}/\gamma_{m,k_0}$.
If $\gamma_{m,k_0}=0$, let $\beta_m=1$, and pick $\alpha_m>0$ so that $|\alpha_m\gamma_{m,j}|<\xi$ for all $j$. Then it is easy to check that
$\Pi_m=\alpha_m\Gamma_m+\beta_m\Sigma_\infty$ is the desired one.
\end{proof}

\begin{rem}
If $(X,\Delta)$ is a klt pair such that $\Delta$ is big, the existence of an effective divisor $D\in\Div(X)_\R$ such that $K_X+\Delta\equiv D$
was proved in \cite{Pau08} with analytic tools.
\end{rem}

\section{Finite generation}\label{proofmain}

\begin{thm}\label{thm:3}
Theorems A$_{n-1}$, B$_n$ and C$_{n-1}$ imply Theorem A$_n$.
\end{thm}
\begin{proof}
Let $F_1,\dots,F_N$ be prime divisors on $X$ such that $\Delta_i=\sum_j\delta_{ij}F_j$ with $\delta_{ij}\in[0,1]$, and
$K_X+\Delta_i+A\sim_\Q\sum_j\gamma_{ij}F_j\geq0$.\\[2mm]
{\em Step 1.\/}
I first show that we can assume $A$ is a general ample $\Q$-divisor, all pairs $(X,\Delta_i+A)$ are klt, and the divisor $\sum F_i$ has simple normal crossings.

Fix an integer $p\gg0$ such that $\Delta_i+pA$ is ample for every $i$, and let $A_i\sim_\Q\frac{1}{p+1}\Delta_i+\frac{p}{p+1}A$ and $A'\sim_\Q\frac{1}{p+1}A$
be general ample $\Q$-divisors. Set $\Delta_i'=\frac{p}{p+1}\Delta_i+A_i$. Then the pairs $(X,\Delta_i'+A')$ are klt and
$K_X+\Delta_i+A\sim_\Q K_X+\Delta_i'+A'$.

Let $g\colon Y\rightarrow X$ be a log resolution of the pair $(X,\sum F_i)$, denote $\B_i=\B(X,\Delta_i'+A')$ for all $i$, and note that
$g^*A'=g_*^{-1}A'\leq\B_i$ since $A'$ is general.
Let $H$ be a small effective $g$-exceptional $\Q$-divisor such that $g^*A'-H$ is ample, and let $A_Y\sim_\Q g^*A'-H$ be a general ample $\Q$-divisor.
Denote $\Delta_{i,Y}=\B_i-g^*A'+H\geq0$, and note that the divisor $E_i=K_Y+\B_i-g^*(K_X+\Delta_i'+A')$ is effective and $g$-exceptional for every $i$. Then
$$K_Y+\Delta_{i,Y}+A_Y\sim_\Q K_Y+\B_{iY}\sim_\Q g^*(K_X+\Delta_i+A)+E_i\sim_\Q g^*\big(\sum\gamma_{ij}F_j\big)+E_i,$$
and $g^*(\sum\gamma_{ij}F_j)+E_i$ has simple normal crossings support. Choose $q\in\Z_{>0}$
such that $D_i'=qk_i(K_Y+\Delta_{i,Y}+A_Y)$ is Cartier for every $i$, and $D_i'\sim qg^*D_i+qk_iE_i$.

Then $R(X;qD_1,\dots,qD_\ell)\simeq R(Y;D_1',\dots,D_\ell')$, and by Lemma \ref{lem:1}(1)
it suffices to prove that $R(Y;D_1',\dots,D_\ell')$ is finitely generated. Now replace $X$ by $Y$, $A$ by $A_Y$ and $\Delta_i$ by $\Delta_{i,Y}$.\\[2mm]
{\em Step 2.\/}
Denote $\mcal T=\{(t_1,\dots,t_\ell):t_i\geq0,\sum t_i=1\}\subset\R^\ell$ and $f_{ij}=\gamma_{ij}-\delta_{ij}$, and note that $f_{ij}>-1$.
For each $\tau=(t_1,\dots,t_\ell)\in\mcal T$, set
$$\delta_{\tau j}=\sum\nolimits_i t_i\delta_{ij}\quad\textrm{and}\quad f_{\tau j}=\sum\nolimits_i t_if_{ij},$$
and observe that
\begin{equation}\label{tag:1}
K_X+A\sim_\R\sum\nolimits_j f_{\tau j}F_j.
\end{equation}
Let $\Lambda=\bigoplus_j\N F_j\subset\Div(X)$, and denote $\mcal B_\tau=\sum_j[f_{\tau j}+\delta_{\tau j},f_{\tau j}+1]F_j\subset\Lambda_\R$
and $\mcal B=\bigcup_{\tau\in\mcal T}\mcal B_\tau$. Since every point in $\mcal B$
is a barycentric combination of the vertices of $\mcal B_{e_i}$, where $e_i$ are the standard basis vectors of $\R^\ell$,
$\mcal B$ is a rational polytope, and thus $\mcal C=\R_+\mcal B$ is a rational polyhedral cone.

For every $j=1,\dots,N$, let
$$\mcal F_{\tau j}=(f_{\tau j}+1)F_j+\sum\nolimits_{k\neq j}[f_{\tau k}+\delta_{\tau k},f_{\tau k}+1]F_k,$$
and set $\mcal F_j=\bigcup_{\tau\in\mcal T}\mcal F_{\tau j}$, which is a rational polytope similarly as above. Then $\mcal C_j=\R_+\mcal F_j$ is
a rational polyhedral cone,
and I claim that $\mcal C=\bigcup_j\mcal C_j$. To see this, fix $s\in\mcal C\backslash\{0\}$. Then there exists $\tau\in\mcal T$ such that
$s\in\R_+\mcal B_\tau$, hence $s=r_s\sum_j(f_{\tau j}+b_{\tau j})F_j$ for some $r_s\in\R_{>0}$, $b_{\tau j}\in[\delta_{\tau j},1]$.
Setting
$$r_\tau=\max_j\Big\{\frac{f_{\tau j}+b_{\tau j}}{f_{\tau j}+1}\Big\}\quad\text{and}\quad
b_{\tau j}'=-f_{\tau j}+\frac{f_{\tau j}+b_{\tau j}}{r_\tau},$$
we have
$$s=r_sr_\tau\sum\nolimits_j(f_{\tau j}+b_{\tau j}')F_j.$$
Note that $r_\tau\in(0,1]$, $b_{\tau j}'\in[\delta_{\tau j},1]$ for all $j$, and there exists $j_0$ such that $b_{\tau j_0}'=1$. Therefore
$s\in\R_+\mcal F_{\tau j_0}\subset\mcal C_{j_0}$, and the claim is proved.\\[2mm]
{\em Step 3.\/}
In this step I prove that for each $j$, the restricted algebra $\res_{F_j}R(X,\mcal C_j\cap\Lambda)$ is finitely generated.

Fix $1\leq j_0\leq N$. By Lemma \ref{lem:gordan}, pick finitely many generators $h_1,\dots,h_m$ of $\mcal C_{j_0}\cap\Lambda$. Similarly as in Step 1 of
the proof of Theorem \ref{thm:2}, it is enough to prove that the restricted algebra $\res_{F_{j_0}}R(X;h_1,\dots,h_m)$ is finitely generated.

By definition of $\mcal C_{j_0}$, for every $h_w$ there exist $r_w\in\Q_+$, $\tau=(t_1,\dots,t_\ell)\in\mcal T_\Q$, and
$b_{\tau j}^w\in[\delta_{\tau j},1]$ for $j\neq j_0$, such that $h_w=r_w\big((f_{\tau j_0}+1)F_{j_0}+\sum_{j\neq j_0}(f_{\tau j}+b_{\tau j}^w)F_j\big)$.
Denote $\Phi_w'=\sum_{j\neq j_0}b_{\tau j}^wF_j$. Fix an integer $p_{j_0}\gg0$ such that $\Phi_w'+p_{j_0}A$ is ample for every $w=1,\dots,m$, and let
$A_w\sim_\Q\frac{1}{p_{j_0}+1}\Phi_w'+\frac{p_{j_0}}{p_{j_0}+1}A$ and $H\sim_\Q\frac{1}{p_{j_0}+1}A$ be general ample $\Q$-divisors. Set
$\Phi_w=\frac{p_{j_0}}{p_{j_0}+1}\Phi_w'+A_w$. Then by \eqref{tag:1},
$$h_w\sim_\Q r_w(K_X+F_{j_0}+\Phi_w+H),$$
and note that $(X,F_{j_0}+\Phi_w+H)$ is a log smooth plt pair with $\lfloor F_{j_0}+\Phi_w+H\rfloor=F_{j_0}$ for every $w$. Furthermore, we have
$$h_w\geq r_w\sum\nolimits_j(f_{\tau j}+\delta_{\tau j})F_j=r_w\sum\nolimits_it_i\sum\nolimits_j(f_{ij}+\delta_{ij})F_j
=r_w\sum\nolimits_it_i\sum\nolimits_j\gamma_{ij}F_j\geq0,$$
so $|K_X+F_{j_0}+\Phi_w+H|_\Q\neq\emptyset$. Choose $q_{j_0}\in\Z_{>0}$ such that $q_{j_0}h_w\sim H_w$ for all $w$, where
$H_w=q_{j_0}r_w(K_X+F_{j_0}+\Phi_w+H)$. Then
$$\res_{F_{j_0}}R(X;q_{j_0}h_1,\dots,q_{j_0}h_m)\simeq\res_{F_{j_0}}R(X;H_1,\dots,H_m),$$
and this last algebra is finitely generated by Theorem \ref{thm:2}. Thus $\res_{F_{j_0}}R(X;h_1,\dots,h_m)$ is finitely generated by Lemma \ref{lem:1}(1).\\[2mm]
{\em Step 4.\/}
Let $\sigma_j\in H^0(X,F_j)$ be the section such that $\ddiv\sigma_j=F_j$ for each $j$.
Consider the $\Lambda$-graded algebra $\mathfrak{R}=\bigoplus_{s\in\Lambda}\mathfrak R_s\subset R(X;F_1,\dots,F_N)$ such that
every element of $\mfrak R$ is a polynomial in elements of $R(X,\mcal C\cap\Lambda)$ and in $\sigma_1,\dots,\sigma_N$.
Note that $\mathfrak R_s=H^0(X,s)$ for every $s\in\mcal C\cap\Lambda$.
In this step I show that the algebra $\mathfrak R$ is finitely generated.

Let $V=\sum_j\R F_j\simeq\R^N$, and let $\|\cdot\|$ be the Euclidean norm on $V$. Since the polytopes $\mcal F_j\subset V$ are compact,
there is a positive constant $C$ such that $\mcal F_j\subset B(0,C)$ for all $j$. Let $\deg\colon\Lambda\rightarrow\N$ be the function given by
$\deg(\sum_j\alpha_jF_j)=\sum_j\alpha_j$, and for a section $\sigma\in\mfrak R_s$ set $\deg\sigma=\deg s$. For every $\mu\in\N$, denote
$\Lambda_{\leq\mu}=\{s\in\Lambda:\deg s\leq\mu\}$, and $\mfrak R_{\leq\mu}=\bigoplus\limits_{s\in\Lambda_{\leq\mu}}\mfrak R_s$.

By Step 3, for each $j$ there exists a finite set $\mcal H_j\subset R(X,\mcal C_j\cap\Lambda)$ such that $\res_{F_j}R(X,\mcal C_j\cap\Lambda)$ is generated
by the set $\{\sigma_{|F_j}:\sigma\in\mcal H_j\}$.
Let $M$ be a sufficiently large positive integer such that $\mcal H_j\subset\mfrak R_{\leq M}$ for all $j$,
and $M\geq CN^{1/2}\max\limits_{i,j}\{\frac{1}{1-\delta_{ij}}\}$.
By H\"{o}lder's inequality we have $\|s\|\geq N^{-1/2}\deg s$ for all $s\in\Lambda$, and thus
\begin{equation}\label{eq:22}
\|s\|/C\geq\max_{i,j}\Big\{\frac{1}{1-\delta_{ij}}\Big\}
\end{equation}
for all $s\in\Lambda\backslash\Lambda_{\leq M}$. Let $\mcal H$ be a finite subset of $\mfrak R$ such that
$\{\sigma_1,\dots,\sigma_N\}\cup\mcal H_1\cup\dots\cup\mcal H_N\subset\mcal H$, and that $\mcal H$ is a set of generators of the finite dimensional
vector space $\mfrak R_{\leq M}$. Let $\C[\mcal H]$ be the ring consisting of polynomials in the elements of $\mcal H$, and observe that trivially
$\C[\mcal H]\subset\mfrak R$.

I claim that $\mathfrak R=\C[\mcal H]$, and the proof is by induction on $\deg\chi$, where $\chi\in\mfrak R$.

Fix $\chi\in\mfrak R$. By definition of $\mfrak R$, write $\chi=\sum_i\sigma_1^{\lambda_{1,i}}\dots\sigma_N^{\lambda_{N,i}}\chi_i$,
where $\chi_i\in R(X,\mcal C\cap\Lambda)$, and note that $\deg\chi_i\leq\deg\chi$. Then it is enough to show that $\chi_i\in\C[\mcal H]$.
By replacing $\chi$ by $\chi_i$, I assume that $\chi\in H^0(X,c)$, where $c\in\mcal C\cap\Lambda$.
If $\deg\chi\leq M$, then $\chi\in\C[\mcal H]$ by definition of $\mcal H$.

Now assume $\deg\chi>M$. By Step 2 there exists $j_0$ such that $c\in\mcal C_{j_0}\cap\Lambda$, and thus,
by definition of $\mcal H$, there are $\theta_1,\dots,\theta_z\in\mcal H$ and a polynomial
$\varphi\in\C[X_1,\dots,X_z]$ such that $\chi_{|F_{j_0}}=\varphi(\theta_{1|F_{j_0}},\dots,\theta_{z|F_{j_0}})$. Therefore,
$$\chi-\varphi(\theta_1,\dots,\theta_z)=\sigma_{j_0}\cdot\chi'$$
for some $\chi'\in H^0\big(X,c-F_{j_0}\big)$ by the relation \eqref{eq:1} in Remark \ref{rem:1}.
Since $\deg\chi'<\deg\chi$, it is enough to prove that $\chi'\in\mfrak R$, since then $\chi'\in\C[\mcal H]$ by induction,
and so $\chi=\sigma_{j_0}\cdot\chi'+\varphi(\theta_1,\dots,\theta_z)\in\C[\mcal H]$.

To that end, since $c\in\mcal C_{j_0}\cap\Lambda$, there exist $\tau\in\mcal T_\Q$, $r_c\in\Q_+$, and $b_{\tau j}\in[\delta_{\tau j},1]$ for $j\neq j_0$,
such that $c=r_cc_{\tau j_0}$, where $c_{\tau j_0}=(f_{\tau j_0}+1)F_{j_0}+\sum\nolimits_{j\neq j_0}(f_{\tau j}+b_{\tau j})F_j\in\mcal F_{j_0}$. Then
$$c-F_{j_0}=r_c\big(\big(f_{\tau j_0}+{\textstyle\frac{r_c-1}{r_c}}\big)F_{j_0}+\sum\nolimits_{j\neq j_0}(f_{\tau j}+b_{\tau j})F_j\big),$$
and observe that $r_c=\|c\|/\|c_{\tau j_0}\|\geq\max_{i,j}\{\frac{1}{1-\delta_{ij}}\}$ by \eqref{eq:22} since $\|c_{\tau j_0}\|\leq C$ by definition of $C$.
In particular $\frac{r_c-1}{r_c}\geq\delta_{\tau j_0}$, and therefore $c-F_{j_0}\in\R_+\mcal B_\tau\cap\Lambda\subset\mcal C\cap\Lambda$.
Thus $\chi'\in R(X,\mcal C\cap\Lambda)\subset\mfrak R$, and we are done.\\[2mm]
{\em Step 5.\/}
Finally, in this step I derive that $R(X;D_1,\dots,D_\ell)$ is finitely generated.

To that end, choose $r\in\Z_{>0}$ such that $rD_i\sim\omega_i$ for $i=1,\dots,\ell$, where $\omega_i=rk_i\sum\nolimits_j\gamma_{ij}F_j$.
Set $\mcal G=\sum_{i=1}^\ell\R_+\omega_i\cap\Lambda$ and note that $\mcal G_\R\subset\mcal C$.
Since $\mfrak R$ is finitely generated by Step 4, the algebra $R(X,\mcal C\cap\Lambda)$ is finitely generated by Lemma \ref{lem:1}(2),
and therefore by Proposition \ref{pro:1} there is a finite rational polyhedral subdivision
$\mcal G_\R=\bigcup_k\mcal G_k$ such that the map $\bMob_{\iota|\mcal G_k\cap\Lambda}$ is additive up to truncation for every $k$,
where $\iota\colon\Lambda\rightarrow\Lambda$ is the identity map.

By Lemma \ref{lem:gordan}, there are finitely many elements $\omega_{\ell+1},\dots,\omega_q\in\mcal G$ that generate $\mcal G$, and denote by
$\pi\colon\bigoplus_{i=1}^q\N\omega_i\rightarrow\mcal G$ the natural projection.
Then the map $\bMob_{\pi|\pi^{-1}(\mcal G_k\cap\Lambda)}$ is additive up to truncation for every $k$,
and thus the algebra $R(X,\pi(\bigoplus_{i=1}^q\N\omega_i))$ is finitely generated by Lemma \ref{lem:1}(3). Since $\bigoplus_{i=1}^\ell\N\omega_i$
is a saturated submonoid of $\bigoplus_{i=1}^q\N\omega_i$, the algebra
$R(X,\pi(\bigoplus_{i=1}^\ell\N\omega_i))\simeq R(X;rD_1,\dots,rD_\ell)$ is finitely generated by Lem\-ma \ref{lem:1}(2),
and finally $R(X;D_1,\dots,D_\ell)$ is finitely generated by Lemma \ref{lem:1}(1).
\end{proof}

Finally, we have:

\begin{proof}[Proof of Theorem \ref{thm:cox}]
Similarly as in Step 1 of the proof of Theorem \ref{thm:3}, by passing to a log resolution $f\colon Y\rightarrow X$ of $(X,\sum\Delta_i)$,
I can assume that $A$ is a general ample $\Q$-divisor and $(X,\Delta_i+A)$ is log smooth for every $i$.

Let $K_X$ be a divisor with $\OO_X(K_X)\simeq\omega_X$ and $\Supp A\not\subset\Supp K_X$, let $V\subset\Div(X)_\R$ be the vector space spanned by
the components of $\sum\Delta_i$, and let $\Lambda\subset\Div(X)$ be the monoid
spanned by the components of $K_X$, $\sum\Delta_i$ and $A$. The set $\mcal C=\sum\R_+D_i\subset\Lambda_\R$ is a rational polyhedral cone.
Similarly as in Step 5 of the proof of Theorem \ref{thm:3} it is enough to prove that the algebra $R(X,\mcal C\cap\Lambda)$ is finitely generated.
By Theorem C the set $\mcal E_{V,A}$ is a rational polytope, and denote $\mcal D=\R_+(K_X+A+\mcal E_{V,A})\cap\mcal C\subset\Lambda_\R$.
Then the algebra $R(X,\mcal C\cap\Lambda)$ is finitely generated if and only if the algebra $R(X,\mcal D\cap\Lambda)$ is finitely generated.
Let $H_1,\dots,H_m$ be generators of the monoid $\mcal D\cap\Lambda$. Then it suffices to prove that the ring $R(X;H_1,\dots,H_m)$
is finitely generated, and this follows from Theorem A.
\end{proof}

\begin{proof}[Proof of Theorem \ref{cor:can}]
By \cite[5.2]{FM00} and by induction on $\dim X$, we may assume $K_X+\Delta$ is big.
Write $K_X+\Delta\sim_\Q A+B$ with $A$ ample and $B$ effective. Let $\varepsilon$ be a small positive rational number and set
$\Delta'=(\Delta+\varepsilon B)+\varepsilon A$. Then $K_X+\Delta'\sim_\Q(\varepsilon+1)(K_X+\Delta)$, thus $R(X,K_X+\Delta)$ and
$R(X,K_X+\Delta')$ have isomorphic truncations, so the result follows from Theorem \ref{thm:cox}.
\end{proof}

\begin{proof}[Proof of Corollary \ref{cor:cor}]
Theorem \ref{cor:can} implies the claim (1) by \cite[3.9]{Fuj09}, and (2) by \cite[1.2(II)]{Rei80}. The claim (3) follows by Theorem \ref{thm:cox}
and \cite[2.9]{HK00}.
\end{proof}

\appendix

\section{History and the alternative}

In this appendix I briefly survey the Minimal Model Program, and then present an alternative approach to the classification
of varieties. There are many works describing Mori theory, and I merely skim through it. My principal goal is to outline a different strategy,
whose philosophy is greatly influenced and advocated by A.~Corti. I do not intend to be exhaustive, but rather to put together results and ideas that
I particularly find important, some of which are scattered throughout the literature or cannot be found in written form.

For many years the guiding philosophy of the Minimal Model Program was to prove finite generation of the canonical ring as a standard consequence of
the theory, namely as a corollary to the existence of minimal models and the Abundance Conjecture. Efforts in this direction culminated in
\cite{BCHM}, which derived the finite generation in the case of klt singularities from the existence of minimal
models for varieties of log general type. However, passing to the case of log canonical singularities, as well as trying to prove the Abundance
Conjecture, although seemingly slight generalisations, seem to be substantially harder problems where different techniques and methods are welcome,
if not needed. The aim of the new approach is to invert the conventional logic of the theory, where finite generation is not at the end, but at the
beginning of the process, and the standard theorems and conjectures of Mori theory are derived as consequences. I hope the results of this paper
give substantial ground to such claims.

There are many contributors to the initial development of Mori theory, Mori, Reid, Kawamata, Shokurov, Koll\'ar, Corti to name a few.
In the MMP one starts with a $\Q$-factorial log canonical pair $(X,\Delta)$, and then constructs a birational map $\varphi\colon X\dashrightarrow Y$
such that the pair $(Y,\varphi_*\Delta)$ has exceptionally nice properties. Namely we expect that in the case of log canonical singularities, there
is the following dichotomy:
\begin{enumerate}
\item if $\kappa(X,K_X+\Delta)\geq0$, then $K_Y+\varphi_*\Delta$ is nef ($Y$ is a {\em minimal model\/}),
\item if $\kappa(X,K_X+\Delta)=-\infty$, then there is a contraction $Y\rightarrow Z$ such that $\dim Z<\dim Y$ and $-(K_Y+\varphi_*\Delta)$ is
ample over $Z$ ($Y$ is a {\em Mori fibre space\/}).
\end{enumerate}
If $Y$ is a Mori fibre space, then it is known that $\kappa(X,K_X+\Delta)=-\infty$ and $X$ is uniruled.
The reverse implication is much harder to prove. The greatest contributions in that direction are \cite{BDPP},
which proves that if $X$ is smooth and $K_X$ is not pseudo-effective, then $X$ is uniruled, and \cite{BCHM}, which proves
that if $K_X+\Delta$ is klt and not pseudo-effective, then there is a map to $Y$ as in (2) above.

The classical strategy is as follows: if $K_X+\Delta$ is not nef, then by the Cone theorem (known for log canonical pairs by the work of Ambro and Fujino,
see \cite{Amb03}) there is a $(K_X+\Delta)$-negative extremal ray $R$ of $\overline{\NE}(X)$,
and by the contraction theorem there is a morphism $\pi\colon X\rightarrow W$ which contracts curves whose classes belong to $R$, and only them.
If $\dim W<\dim X$, then we are done. Otherwise $\pi$ is birational, and there are two cases. If $\codim_X\Exc\pi=1$, then $\pi$ is a {\em divisorial
contraction\/}, $W$ is $\Q$-factorial and $\rho(X/W)=1$, and we continue the process starting from the pair $(W,\pi_*\Delta)$. If $\codim_X\Exc\pi\geq2$,
then $\pi$ is a {\em flipping contraction\/}, $\rho(X/W)=1$, but $K_W+\pi_*\Delta$ is no longer $\Q$-Cartier. In order to proceed,
one needs to construct the {\em flip\/} of $\pi$, namely a birational map $\pi^+\colon X^+\rightarrow W$ such that $X^+$ is $\Q$-factorial, $\rho(X^+/W)=1$ and
$K_{X^+}+\phi_*\Delta$ is ample over $W$, where $\phi\colon X\dashrightarrow X^+$ is the birational map which completes the diagram. Continuing the
procedure, one hopes that it ends in finitely many steps.

Therefore there are two conjectures that immediately arise in the theory: existence and termination of flips. Existence of the flip of a flipping
contraction $\pi\colon X\rightarrow W$ is known to be equivalent to the finite generation of the {\em relative canonical algebra\/}
$$R(X/W,K_X+\Delta)=\bigoplus_{m\in\N}\pi_*\OO_X(\lfloor m(K_X+\Delta)\rfloor),$$
and the flip is then given by $X^+=\Proj_WR(X/W,K_X+\Delta)$. The termination of flips is related to conjectures about the behaviour of the coefficients
in the divisor $\Delta$, but I do not discuss it here.

Since the paper \cite{Zar62}, one of the central questions in higher dimensional birational geometry is the following:

\begin{con}\label{con:1}
Let $(X,\Delta)$ be a projective log canonical pair. Then the canonical ring $R(X,K_X+\Delta)$ is finitely generated.
\end{con}

Finite generation implies existence of flips \cite[3.9]{Fuj09}; moreover, one only needs to assume finite generation for pairs $(X,\Delta)$
with $K_X+\Delta$ big.

The proof of the finite generation in the case of klt singularities along the lines of the classical philosophy in \cite{BCHM}
is as follows: by \cite[5.2]{FM00} one can assume that $K_X+\Delta$ is big. Then by applying
carefully chosen flipping contractions, one proves that the corresponding flips exist and terminate ({\em termination with scaling\/}),
and since the process preserves the canonical ring, the finite generation follows from the basepoint free theorem.

Now consider a flipping contraction $\pi\colon(X,\Delta)\rightarrow W$ with additional properties that $(X,\Delta)$ is a plt pair such that
$S=\lfloor\Delta\rfloor$ is an irreducible divisor which is negative over $Z$. This contraction is called {\em pl flipping\/}, and the corresponding
flip is the pl flip. Following the work of Shokurov, one of the steps in the proof in \cite{BCHM} is showing that pl flips exist, and the starting point
is Lemma \ref{lem:restricted} below. Note that in the context of pl flips, the issues which occur in the problem of global finite generation outlined
in the introduction to this paper do not exist. I give a slightly modified proof of the lemma below than the one present elsewhere in the literature
in order to stress the following point: I do not {\em calculate\/} the kernel
of the restriction map, but rather {\em chase\/} the generators. This reflects the basic principle: if our algebra is large enough so that it contains
the equation of the divisor we are restricting to, then it is automatically finitely generated assuming the restriction to the divisor is.
This is one of the main ideas guiding the proof in \S\ref{proofmain}.

\begin{lem}\label{lem:restricted}
Let $(X,\Delta)$ be a plt pair of dimension $n$, where $S=\lfloor\Delta\rfloor$ is a prime divisor, and let
$f\colon X\rightarrow Z$ be a pl flipping contraction with $Z$ affine. Then
$R(X/Z,K_X+\Delta)$ is finitely generated if and only if $\res_SR(X/Z,K_X+\Delta)$ is finitely generated.
\end{lem}
\begin{proof}
We will concentrate on sufficiency, since necessity is obvious.

Numerical and linear equivalence over $Z$ coincide by the basepoint free theorem. Since $\rho(X/Z)=1$, and both $S$ and $K_X+\Delta$ are $f$-negative,
there exists a positive rational number $r$ such that $S\sim_{\Q,f}r(K_X+\Delta)$. By considering an open cover of $Z$ we can assume that
$S-r(K_X+\Delta)$ is $\Q$-linearly equivalent to a pullback of a principal divisor.

Therefore $S\sim_\Q r(K_X+\Delta)$, and since then $R(X,S)$ and $R(X,K_X+\Delta)$ have isomorphic truncations, it is enough to prove that
$R(X,S)$ is finitely generated. As a truncation of $\res_S R(X,S)$ is isomorphic to a truncation of
$\res_SR(X,K_X+\Delta)$, we have that $\res_S R(X,S)$ is finitely generated. Let $\sigma_S\in H^0(X,S)$ be a section such that $\ddiv\sigma_S=S$ and let
$\mcal H$ be a finite set of generators of the finite dimensional vector space $\bigoplus_{i=1}^d\res_S H^0(X,iS)$, for some $d$, such that
the set $\{h_{|S}:h\in\mcal H\}$ generates $\res_S R(X,S)$. Then it is easy to see that $\mcal H\cup\{\sigma_S\}$ is a set of generators of $R(X,S)$,
since $\ker(\rho_{kS,S})=H^0(X,(k-1)S)\cdot\sigma_S$ for all $k$, in the notation of Remark \ref{rem:1}.
\end{proof}

One of the crucial unsolved problems in higher dimensional geometry is the following Abundance Conjecture.

\begin{con}
Let $(X,\Delta)$ be a projective log canonical pair such that $K_X+\Delta$ is nef. Then $K_X+\Delta$ is semiample.
\end{con}

Until the end of the appendix I discuss this conjecture more thoroughly. There are, to my knowledge, two different ways to approach this problem.

The first approach is close to the classical strategy, and goes back to \cite{Kaw85}. First let us recall the following definition from \cite{Nak04};
the corresponding analytic version can be found in \cite{Pau08}.

\begin{dfn}\label{dfn:2}
Let $X$ be a projective variety. For $D\in\overline{\bigcone(X)}$ denote
$$\sigma(D,A)=\sup\big\{k\in\N:\liminf_{m\rightarrow\infty}\textstyle\frac{1}{m^k}h^0(X,\lfloor mD\rfloor+A)>0\big\}.$$
Then the {\em numerical dimension\/} of $D$ is
$$\nu(X,D)=\sup\{\sigma(D,A):A\textrm{ is ample}\}.$$
\end{dfn}

We know that $\nu(X,D)=0$ if and only if $D\equiv N_\sigma\|D\|$, and that $\nu(X,D)$ is the standard numerical dimension when $D$ is nef by
\cite[6.2.8]{Nak04}. It is well known that abundance holds when $\nu(X,K_X+\Delta)$ is equal to $0$ or $\dim X$ by \cite[8.2]{Kaw85b},
and when $\nu(X,K_X+\Delta)=\kappa(X,K_X+\Delta)$ by \cite[6.1]{Kaw85}, cf.\ \cite{Fuj09b}.

\begin{thm}
Let $(X,\Delta)$ be a projective klt pair of dimension $n$ such that $K_X+\Delta$ is nef. Assume that $\nu(Y,K_Y+\Delta_Y)>0$ implies
$\kappa(Y,K_Y+\Delta_Y)>0$ for any klt pair $(Y,\Delta_Y)$ of dimension at most $n$. Then $K_X+\Delta$ is semiample.
\end{thm}
\begin{proof}
Let $(S,\Delta_S)$ be a $\Q$-factorial $(n-1)$-dimensional klt pair with $\kappa(S,K_S+\Delta_S)=0$. Then $\nu(S,K_S+\Delta_S)=0$ by
the assumption in dimension $n-1$, and thus $K_S+\Delta_S\equiv N_\sigma\|K_S+\Delta_S\|$. By \cite[3.4]{Dru09} a minimal model
of $(S,\Delta_S)$ exists. Now the result follows along the lines of \cite[7.3]{Kaw85}.
\end{proof}
The assumption in the theorem can be seen as a stronger version of non-vanishing.

Now I present a different approach, where one derives abundance from the finite generation. It is a result of J.~M\textsuperscript{c}Kernan and
C.~Hacon, and I am grateful to them for allowing me to include it here.

\begin{thm}
Assume that for every $(n+1)$-dimensional projective log canonical pair $(X,\Delta)$ with $K_X+\Delta$ nef and big, the canonical ring
$R(X,K_X+\Delta)$ is finitely generated. Then abundance holds for klt pairs in dimension $n$.
\end{thm}
\begin{proof}
Let $(Y,\Phi)$ be an $n$-dimensional projective klt pair such that $K_Y+\Phi$ is nef, and let $Y\subset\PP^N$ be some projectively normal embedding.
Let $X_0$ be the cone over it, let $X=\PP(\OO_Y\oplus\OO_Y(1))$ be the blowup of $X_0$ at the origin, and let $H'\subset\PP^N$
be a sufficiently ample divisor which does not contain the origin. Let $\Delta$ and $H$ be the proper transforms in $X$ of $\Phi$ and $H'$, respectively,
and let $E\subset X$ be the exceptional divisor.

Then by inversion of adjunction the pair $(X,\Upsilon=E+\Delta+H)$ is log canonical, and of log general type since $H'$ is ample enough.
We have $Y\simeq E$, and this isomorphism maps $K_Y+\Phi$ to $K_E+\Delta_{|E}$.
The divisor $K_X+\Upsilon$ is also nef: since $(K_X+E+\Delta)_{|E}$ is identified with $K_Y+\Phi$, this deals with curves
lying in $E$ by nefness, and for those curves which are not in $E$, the ampleness of $H$ away from $E$ ensures that the intersection product with
$K_X+\Upsilon$ is positive. By assumption, the algebra $R(X,K_X+\Upsilon)$ is finitely generated, therefore $K_X+\Upsilon$ is
semiample by \cite[2.3.15]{Laz04b}, and thus so is $K_E+\Delta_{|E}=(K_X+\Upsilon)_{|E}$.
\end{proof}

Finally a note about the general alternative philosophy. Since \cite{HK00} it has become clear that adjoint rings encode many important geometric
information about the variety. In particular, by Corollary \ref{cor:cor}(3) and \cite[1.11]{HK00},
all the main theorems and conjectures of Mori theory hold on $X$, such as the Cone and Contraction
theorems, existence and termination of flips, abundance. In particular, the following conjecture applied to Mori dream regions
\cite[2.12, 2.13]{HK00} seems to encode the whole Mori theory.

\begin{con}
Let $X$ be a projective variety, and let $D_i=k_i(K_X+\Delta_i)\in\Div(X)$, where $(X,\Delta_i)$ is a log canonical pair for $i=1,\dots,\ell$.
Then the adjoint ring $R(X;D_1,\dots,D_\ell)$ is finitely generated.
\end{con}

\bibliography{biblio}
\pagestyle{plain}
\end{document}